\documentclass[11pt,reqno]{amsart}

 \usepackage{amsmath,amsthm,amsfonts}
\usepackage{latexsym,mathabx,txfonts,graphicx}
\usepackage{amssymb}

\newcommand{\om}{\omega}

\newcommand{\fy}{\varphi}

\newcommand{\p}{\partial}
\newcommand{\I}{\infty}
\newcommand{\wt}[1]{\widetilde {#1}}

\renewcommand\tilde{\wt}
\newcommand{\R}{\mathbb{R}}

\newcommand{\C}{\mathbb{C}}
\newcommand{\N}{\mathbb{N}}

\renewcommand{\Re}{\mathop{\mathrm{Re}}}

\renewcommand{\bar}{\overline}
\renewcommand{\hat}{\widehat}

\def\calL{{\mathcal L}}

\def\la{\lambda}

\numberwithin{equation}{section}

\newtheorem{thm}{Theorem}[section]

\newtheorem{lem}[thm]{Lemma}
\newtheorem{prop}[thm]{Proposition}

\theoremstyle{remark}
\newtheorem{rem}{Remark}[section]
\newtheorem{defn}{Definition}[section]

\newcommand{\EQ}[1]{\begin{equation}  \begin{split} #1 \end{split} \end{equation} }
\newcommand{\Del}[1]{}

\def\calR{\mathcal{R}}

\def\eps{\varepsilon}

\def\f{\frac}

\def\til{\tilde}
\def\fy{\varphi}
\def\calR{\mathcal{R}}

\def\wt{\widetilde}

\def\cD{\mathcal{D}}
\renewcommand\ln{\log}

\def\cQ{\mathcal{Q}}

\begin{document}

\title[Stable Type II solutions]{Type II blow up solutions with optimal stability properties for the critical focussing nonlinear wave equation on $\R^{3+1}$}
\author{Stefano Burzio, Joachim Krieger}


\thanks{Support of the Swiss National Fund is gratefully acknowledged.}
\begin{abstract}
We show that the finite time type II blow up solutions for the energy critical nonlinear wave equation 
\[
\Box u = -u^5
\]
on $\R^{3+1}$ constructed in \cite{KST}, \cite{KS1} are stable along a co-dimension one Lipschitz manifold of data perturbations in a suitable topology, provided the scaling parameter $\lambda(t) = t^{-1-\nu}$ is sufficiently close to the self-similar rate, i. e. $\nu>0$ is sufficiently small. This result is qualitatively optimal in light of the result of \cite{CNLW4}. The paper builds on the analysis of \cite{CondBlow}. 
\end{abstract}

\maketitle

\section{Introduction}

The critical focussing nonlinear wave equation on $\R^{3+1}$ given by 
\begin{equation}\label{eq:critmain}
\Box u = -u^5,\,\Box  = -\partial_t^2 + \triangle,
\end{equation}
has received a lot of attention recently as a key model for a critical nonlinear wave equation displaying interesting type II dynamics, the latter referring to energy class Shatah-Struwe type solutions $u(t, x)$ which have a priori bounded $\dot{H}^1$ norm on their life-span $I$, i. e. with the property 
\begin{equation}\label{eq:typeII}
\sup_{t\in I}\big\|\nabla_{t,x}u(t, \cdot)\big\|_{L_x^2}<\infty.
\end{equation}
Throughout the paper, we shall be interested exclusively in the case of radial solutions. In that case, a rather complete abstract classification theory for type II dynamics in terms of the ground state
\[
W(x) = \frac{1}{\big(1+\frac{|x|^2}{3}\big)^{\frac12}}
\]
has been developed in \cite{DKM4}, see the discussion in \cite{CondBlow}. On the other hand, the first 'non-trivial' type II dynamics, were constructed explicitely in \cite{KS}, \cite{KST}, \cite{KS1}, \cite{DoKr}, \cite{DoMHKS}
. As far as finite time type II blow up solutions are concerned, the issue of their {\it{stability properties}}  has been shrouded in some mystery. The fact that there is a {\it{continuum of blow up rates}} in the works \cite{KST}, \cite{KS1}, seemed to suggest that these solutions, and maybe also their analogues for critical Wave Maps and other models, such as in \cite{KST1}, \cite{KST2}, are intrinsically less stable than 'generic type II blow ups', and that the {\it{requirement of optimal stability}} of some sort may in fact single out a more or less {\it{unique blow up dynamics}} for type II solutions. An example of 'optimally stable' type II blow up was exhibited in the context of the $4+1$-dimensional critical NLW in the work \cite{HR}, see also the brief historical comments in \cite{CondBlow}. Note that the linearisation of \eqref{eq:critmain} around the ground state $W$
has a unique unstable eigenmode $\phi_d$, and in accordance with this, \cite{HR} exhibits a co-dimensional one manifold of data perturbations of $W$ (in the $4+1$-dimensional context) resulting in the stable blow up. 
\\
In this article we show that the solutions constructed in \cite{KST}, \cite{KS1}, corresponding to $\lambda(t) = t^{-1-\nu}$ and {\it{with $\nu>0$ small enough}} are also optimally stable in a suitable sense. However, due to the fact that these solutions are only of finite regularity, and in effect experience a shock along the light cone centered at the singularity, an appreciation of our result requires carefully reviewing the nature of them. 

\subsection{The type II blow up solutions of \cite{KST}, \cite{KS1}}

Solutions of \eqref{eq:critmain} are divided into those of type II, satisfying \eqref{eq:typeII}, as well as solutions of type I which violate this condition. The celebrated result in \cite{DKM4} provides a general criterion characterising abstract radial type II solutions in terms of the ground states $\pm W(x)$. In particular, assuming that $u(t, x)$ is a type II solution of \eqref{eq:critmain} which is radial and develops a singularity at time $t = T$, then near $T$, we can write 
\begin{equation}\label{eq:DKMstructure}
u(t, x) = \sum_{j=1}^N \kappa_j W_{\lambda_j(t)}(x) + \epsilon(t, x),\,W_{\lambda}(x) = \lambda^{\frac12}W(\lambda x), 
\end{equation}
where $\epsilon(t, x)$ can be extended continuously as an energy class solutions past the singularity $t = T$, $\kappa_j = \pm 1$, $\lim_{t\rightarrow T}(T-t)\lambda_j(t) = +\infty$, and $\lim_{t\rightarrow T}\big|\log(\frac{\lambda_j(t)}{\lambda_k(t)}\big| = +\infty$, provided $j\neq k$. 
\\
We note here that this appears the only result for a non-integrable PDE where this kind of a continuous in time solution resolution has been proved. 
\\
We also observe that the solutions in \cite{KST}, \cite{KS1}, \cite{DoMHKS} appear to be the only known finite time type II blow up solutions for \eqref{eq:critmain}, all with $N = 1$, and that in fact solutions of the form \eqref{eq:DKMstructure} with $N\geq 2$ are not known at this time (and might not exist). 
\\

We now detail briefly these specific blow up solutions. Let $\nu>0$, but otherwise arbitrary, and denote $\lambda(t) = t^{-1-\nu}$. 
\begin{thm}\label{thm:KST}(\cite{KST}, \cite{KS1}) There is $t_0>0$ and a radial solution $u_{\nu}(t, x)$ of the form 
\[
u_{\nu}(t, x) = W_{\lambda(t)}(x) + \eta(t, x),
\]
where the term\footnote{The notation $H^{s-}$ means in any $H^{s'}$, $s'<s$.} $\eta(t,\cdot)\in H^{1+\frac{\nu}{2}-}, \eta_t(t,\cdot)\in H^{\frac{\nu}{2}-}$ for any $t\in (0, t_0]$, and we have the asymptotic vanishing relation 
\[
\lim_{t\rightarrow 0}\int_{|x|<t}[\big|\nabla_{t,x}\eta\big|^2 + \frac{1}{6}\eta^6]\,dx = 0.
\]
The correction term $\eta$ satisfies $\eta|_{\{|x|<t\}}\in C^\infty$, while it is only of regularity $H^{1+\frac{\nu}{2}-}$ across the light cone. In fact, 
there is a splitting $\eta(t, x) = \eta_e(t, x) + \epsilon(t, x)$, with (here $N$ may be picked arbitrarily large, depending on the number of steps used to construct $\eta_e$)
\[
\big\|\epsilon(t, \cdot)\big\|_{H^{1+\frac{\nu}{2}-}} + \big\|\epsilon_t(t, \cdot)\big\|_{H^{\frac{\nu}{2}-}} <t^N,
\]
and such that using the new variables $R = \lambda(t)r$, $r = |x|$, $a = \frac{r}{t}$,  there is an expansion near $a = 1$ of the form 
\begin{equation}\label{eq:expansion1}
\eta_e(t, x) = \frac{\lambda^{\frac12}}{(\lambda t)^2}\cdot \sum_{l\geq 0}\sum_{0\leq k\leq k(l)}c_l(a, t)(\log R)^k R^{1-l}, 
\end{equation}
and such that 
\begin{equation}\label{eq:expansion2}
c_l(a, t) = q_0^{(l)}(t, a) + \sum_{i=1}^\infty (1-a)^{\beta(i)+1}\sum_{j=0}^{j(i)}q_{ij}^{(l)}(t, a)\big(\log(1-a)\big)^j.
\end{equation}
Here the coefficients $q_0^{(l)}(t, a), q_{ij}^{(l)}(t, a)$ are of class $C^\infty$, while the exponents $\beta(i)$ are of the form
\[
\beta(i) = \sum_{k\in K_i}\big((2k-\frac{3}{2})\nu - \frac12)\big) + \sum_{k\in K_i'}\big((2k-\frac{1}{2})\nu - \frac12)\big)
\]
for suitable finite sets of positive integers $K_i, K_i'$. In particular, $\beta(i)\geq \frac{\nu-1}{2}$. The sums in \eqref{eq:expansion1},  \eqref{eq:expansion2} are absolutely convergent, and the most singular terms in \eqref{eq:expansion2} are of the form 
\[
(1-a)^{\frac{\nu+1}{2}}\log(1-a).
\]
\end{thm} 

We observe that a similar asymptotic expansion as in \eqref{eq:expansion1}, \eqref{eq:expansion2} near $a = 1$ may also be inferred for the error $\epsilon(t, x)$, and thus the singularity of $\eta(t, x)$ is indeed confined exactly to the forward light cone $|x| = t$ centered in the singularity, see \cite{KS1}. However, the methods for determining $\eta_e(t,x)$ and $\epsilon(t, x)$ differ importantly. The first is in fact obtained by approximating the wave equation by a finite number of elliptic equations approximating the wave equation in a suitable sense, while the second quantity is obtained as solution of a wave equation via a suitable parametrix method. Both of these techniques will play an important role in this paper. 
We shall see next that the limited regularity and more precisely the shock across the light cone entails a certain rigidity for such solutions, which will be reflected in terms of the stability properties of this kind of blow up. 

\subsection{The effect of symmetries on the solutions of Theorem~\ref{thm:KST}} In the sequel, we shall assume $0<\nu\ll 1$. Restricting to the radial setting, the symmetry group acting on solutions of \eqref{eq:critmain} is restricted to time translations as well as scaling transformations $u(t, x)\longrightarrow \lambda^{\frac12}u(\lambda t, \lambda x)$, and it is then natural to subject the special solutions $u_{\nu}(t, x)$ to such transformations. Let us consider the effect on the principal singular term, which is of the schematic form 
\[
\frac{\lambda^{\frac12}(t)}{(\lambda(t)\cdot t)^2}\cdot R\log(1+R^2)\cdot (1-a)^{\frac12+\frac{\nu}{2}}\log(1-a),\,a = \frac{r}{t},\,R = \lambda(t)r,\,\lambda(t) = t^{-1-\nu}. 
\]
Calling this term $\eta_{p}(t, x)$, we find by simple inspection that for $T\neq 0$
\begin{align*}
&\eta_{p}(t,\cdot) - \eta_{p}(t-T,\cdot)\in H^{1+\frac{\nu}{2}-}_{loc}
\end{align*}
and this is in effect optimal, i. e. the preceding difference is in no $H^s_{loc}$ for any $s\geq 1+\frac{\nu}{2}$. This is of course simply due to the fact that time translating $u_{\nu}$ will shift the forward light cone on which the solution experiences a shock, and so the difference will be no smoother than $u_{\nu}$. The same phenomenon occurs for the difference
\[
\eta_{p}(t, x) - \lambda^{\frac12}\eta_p(\lambda t, \lambda x),\,\lambda\neq 1. 
\]

What we shall intend in this article is to consider {\it{smooth perturbations}} of the solutions $u_{\nu}(t, x)$, i. e. consider the evolution corresponding to the initial data on the time slice $t=t_{0}$ 
\[
u_{\nu}[t_0] + \big(\epsilon_0, \epsilon_1\big),\,u_{\nu}[t_0] = \big(u_{\nu}(t_0, \cdot),\,\partial_tu_{\nu}(t_0, \cdot)\big),
\]
where $\big(\epsilon_0, \epsilon_1\big)\in H^{\frac32+}(\R^3)\times H^{\frac12+}(\R^3)$, in a way made more precise in the sequel, see  section \ref{sec:datapert}. In particular, we see that the differences 
\[
u_{\nu}[t_0-T] - u_{\nu}[t_0 ],\,u_{\nu,\lambda}[t_0], - u_{\nu}[t_0],\,T\neq 0,\,\lambda\neq 1, 
\]
are not of this form, since $1+\frac{\nu}{2}<\frac32+$ for $\nu\ll 1$. In Figure \ref{fig:unu} below we have plotted the leading behavior of the function $u_{\nu}(t,x)$, the shock along the forward light cone generated from the origin is evident. Moreover, in Figure \ref{fig:translation} one can see that the difference $u_{\nu}[t_0-T] - u_{\nu}[t_0 ]$ manifest cusp type singularities at $|x|=t_{0}$ and $|x|=t_0-T$. 

\begin{figure}[h!]
        \centering
        \includegraphics[width=\textwidth]{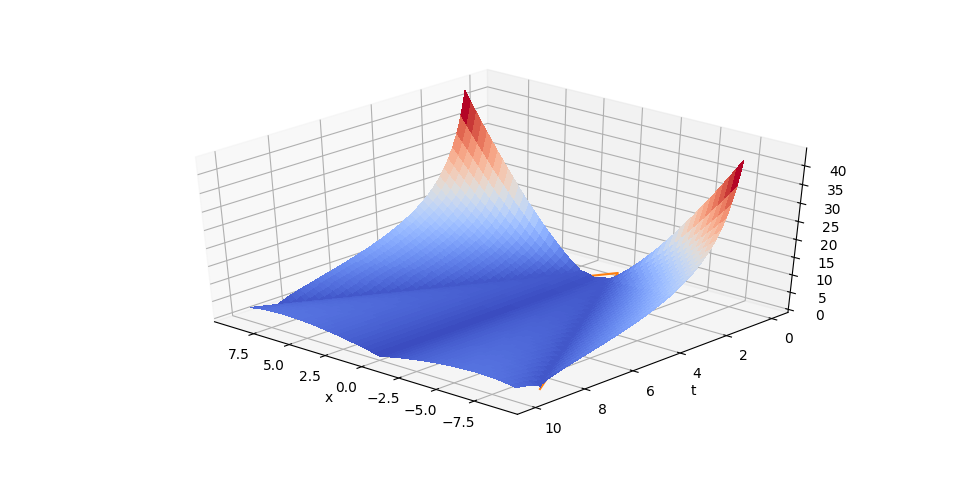}
        \caption{Graph of $u_{\nu}(t,x)$ for $t>0$ and $\nu=0.1$.}\label{fig:unu}
\end{figure}
   
\begin{figure}[h!]
        \centering
        \includegraphics[width=\textwidth]{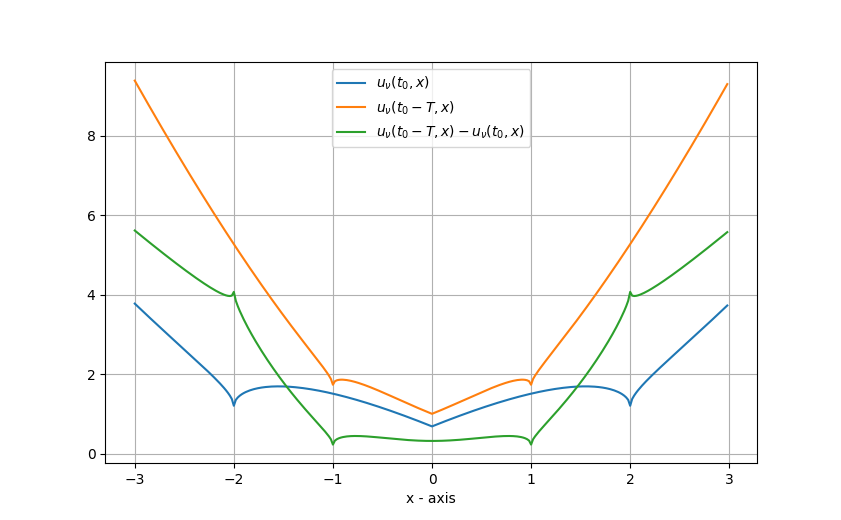}
        \caption{Graphs of initial data with $\nu=0.1$, $t_{0}=2$, and $T=1$.}\label{fig:translation}
\end{figure}

This reveals that the role of the symmetries in describing the evolutions of the data 
\begin{equation}\label{eq:roughstructure}
u_{\nu}[t_0] + \big(\epsilon_0, \epsilon_1\big),\,\big(\epsilon_0, \epsilon_1\big)\in H^{\frac32+}(\R^3)\times H^{\frac12+}(\R^3), 
\end{equation}
is not a priori clear, and in fact, we shall show that the blow up corresponding to (a certain subclass of) such data perturbations takes place in the same space-time location and with the same scaling law, which may sound paradoxical at first, but is explained by the role of the topology of the data. 
\\
In fact, what our main result shall reveal, and what is also borne out by the result \cite{KS}, while the abstract general classification theory by Duyckaerts-Kenig-Merle takes place in the largest possible space $H^1$ in which the problem \eqref{eq:critmain} is well-posed, an understanding of the precise possible dynamics (involving blow up speeds, stability properties, etc) rely crucially on finer topological properties of the data in spaces more restrictive than $H^1$. It is conceivable that such considerations have much broader applicability for certain nonlinear {\it{hyperbolic}} problems. 

\subsection{Conditional stability of type II solutions}

Before stating the main theorem of this paper about stability properties of the solutions in Theorem~\ref{thm:KST} with $\nu\ll 1$, we place it briefly into a broader context. It is intuitively clear that when analysing the stability of any of the type II solutions in \eqref{eq:DKMstructure} with $N = 1$, say, the linearisation of the equation \eqref{eq:critmain} around $W$, and thence the operator\footnote{It arises by passing from radial $u(x) = v(R)$ to $Rv(R)$, $R = |x|$}
\begin{equation}\label{eq:L}
\mathcal{L} = -\partial_R^2 - 5W^4(R),\,W(R): = \frac{1}{(1+\frac{R^2}{3})^{\frac12}},
\end{equation}
will play a pivotal role. This operator, when restricted to functions on $[0,\infty)$ with Dirichlet condition at $R=0$, has a simple negative eigenvalue $\xi_d<0$ (the subscript $d$ referring to 'discrete spectrum'), and a corresponding $L^2$-normalized positive ground state $\phi_d$ with 
\[
\mathcal{L}\phi_d = \xi_d\phi_d, 
\]
see \cite{KS}, \cite{KST}. This mode will cause exponential growth for the linearised flow $e^{it\sqrt{\mathcal{L}}}$, and only a co-dimension one condition will ensure that the forward flow will remain bounded. That a corresponding center-stable manifold may be constructed for perturbations of type II solutions for the nonlinear problem \eqref{eq:critmain} was first shown in the context of the special solution $u(t, x) = W(x)$ and perturbations in a topology which is significantly stronger than $H^1$ in \cite{KS}, and later in vastly larger generality (for perturbations of only regularity $\dot{H}^1$) in \cite{CNLW4}. Here we let $u(t, x)$ be a general solution of regularity $\dot{H}^1$, which may be obtained as limit of a sequence of smooth solutions. We shall refer to such solutions as 'Shatah-Struwe solutions'. 
\begin{thm}\label{thm:codim1stable}(\cite{CNLW4}) Let 
\begin{equation}\label{eq:oneprofile}
u(t, x) = W_{\lambda(t)}(x) + v(t, x)
\end{equation}
be a type II blow up solution on $I\times \R^3$ for \eqref{eq:critmain}, such that 
\[
\sup_{t\in I}\big\|\nabla_{t,x}v(t, \cdot)\big\|_{L_x^2}\leq \delta\ll 1
\]
for some sufficiently small $\delta>0$, where as usual $I$ denotes the maximal life span of the Shatah-Struwe solution $u$. Also, assume that $t_0\in I$. Then there exists a co-dimension one Lipschitz manifold $\Sigma$ in a small neighbourhood of the data $\big(u(t_0, \cdot), u_t(t_0, \cdot)\big)\in \Sigma$ in the energy topology $\dot{H}^1(\R^3)\times L^2(\R^3)$ and such that initial data $\big(u_0, u_1\big)\in \Sigma$ result in a type II solution, while initial data 
\[
\big(u_0, u_1\big)\in B_{\delta}\backslash \Sigma,
\]
where $B_{\delta}\subset \dot{H}^1(\R^3)\times L^2(\R^3)$ is a sufficiently small ball centred at $\big(u(t_0, \cdot), u_t(t_0, \cdot)\big)$, either lead to blow up in finite time, or solutions scattering to zero, depending on the 'side of $\Sigma$' these data are chosen from. 
\end{thm}

Note that by contrast to the result in \cite{KS} which precisely describes the dynamics of the perturbed solutions but at the expense of a much more restrictive class of perturbations, there is no description of the perturbed solutions in the preceding theorem other than the assertion that the solutions are of type II. 
\\
The question we shall now address is whether the {\it{specific dynamics}} of the solutions in Theorem~\ref{thm:KST} are preserved for a suitable class of perturbations, essentially as in \eqref{eq:roughstructure}. Note that such perturbations only constitute a very small subset of the surface $\Sigma$ in the preceding theorem, as evidenced by the fact that if $\mathcal{S}_{T, \lambda}$ denotes re-scaling by $\lambda$ and time-tranlsation by $T$, then if $(T,\lambda)\neq (0, 1)$, any two data pairs  
\[
u_{\nu}[t_0] + \big(\epsilon_0, \epsilon_1\big),\,\big(\mathcal{S}_{T, \lambda}u_{\nu}\big)[t_0] + \big(\epsilon_0', \epsilon_1'\big)
\]
with $\big(\epsilon_0, \epsilon_1\big)\in H^{\frac32+}\times H^{\frac12+}, \big(\epsilon_0', \epsilon_1'\big)\in H^{\frac32+}\times H^{\frac12+}$ will be distinct. 
\\
We aim now to understand the evolution of a certain class of  data $u_{\nu}[t_0] + \big(\epsilon_0, \epsilon_1\big)$, with $t_0$ as in Theorem~\ref{thm:KST}, backward in time. Precisely stating the conditions on the perturbation $\big(\epsilon_0, \epsilon_1\big)$ requires certain technical preliminaries involving the spectral theory and representation associated to $\mathcal{L}$, mostly developed in \cite{KST}. 

\subsection{Spectral theory associated with the linearisation $\mathcal{L}$}

Here we quote from \cite{KST}, specifically Lemma 4. 2 as well as Proposition 4. 3 in loc. cit. Let $\mathcal{L}$ be given by \eqref{eq:L}, restricted to $L^2((0,\infty)$, with domain 
\[
\text{Dom}(\mathcal{L}) = \{f\in L^2((0,\infty): f,\,f'\in AC([0, R])\forall R>0,\,f(0) = 0, f''\in L^2((0,\infty)\}
\]
Then $\mathcal{L}$ is self-adjoint with this domain, and its spectrum consists of 
\[
\text{spec}(\mathcal{L}) = \{\xi_d\}\cup [0,\infty),
\]
with $\xi_d<0$ the unique negative eigenvalue of $\mathcal{L}$ and associated $L^2$-normalized and positive ground state $\phi_d(R)$. There is a resonance at zero given by the function 
\[
\phi_0(R) = R(1-\frac{R^2}{3})(1+\frac{R^2}{3})^{-\frac32},\,\mathcal{L}\phi_0 = 0. 
\]
The latter is simply a reflection of the scaling invariance of the problem. 
\\
Importantly, the operator $\mathcal{L}$ induces a 'distorted Fourier transform' $\mathcal{F}(f) = \hat{f}$, which allows for a nice Fourier representation in terms of generalised eigenfunctions $\phi(R, \xi)$. We have the following 
\begin{prop}\label{prop:Fourier}(\cite{KST}) For each $z\in \C$, one can define a basis of generalised eigenfunctions 
\begin{equation}\label{eq:FourierExpansion}
\phi(R, z) = \phi_0(R) + R^{-1}\sum_{j=1}^\infty (R^2 z)^j\phi_j(R^2)
\end{equation}
given by an absolutely convergent sum, with $\phi_j(u)$ holomorphic on the complex numbers with $\Re u>-1/2$, and satisfying bounds 
\[
\big|\phi_j(u)\big|\leq \frac{C^j}{(j-1)!}|u|\langle u\rangle^{-\frac12}.
\]
Denoting the Jost solutions $f_{\pm}(R, \xi)$ which satisfy $\mathcal{L}f_{\pm} = \xi f_{\pm}$ as well as $f_{\pm}(R, \xi)\sim e^{\pm iR\xi^{\frac12}}$ as $R\rightarrow +\infty$, there is a representation 
\[
\phi(R, \xi) = \sum_{\pm}a_{\pm}(\xi)f_{\pm}(R, \xi),
\]
with $a_{\pm}(\xi)\rightarrow 1$ as $\xi\rightarrow 0 $ as well as $\big|a_{\pm}(\xi)\big|\lesssim \xi^{-\frac12}$ as $\xi\rightarrow\infty$. 
Further, there is a function $\rho(\xi)\in C^\infty\big((0,\infty)\big)$ with the asymptotic behaviour 
\[
\rho(\xi)\sim\begin{cases}
\xi^{-\frac12}, \hspace{0,5cm}&0<\xi\ll 1,\\
\xi^{\frac12},&\xi\gg 1, 
\end{cases}
\]
as well as symbol behaviour with respect to differentiation, and such that defining 
\begin{align*}
&\mathcal{F}(f)(\xi) := \hat{f}(\xi) := \lim_{b\rightarrow +\infty}\int_0^b \phi(R, \xi)f(R)\,dR,\,\xi\geq 0,\\
& \hat{f}(\xi_d) = \int_0^\infty \phi_d(R) f(R)\,dR, 
\end{align*}
the map $f\longrightarrow \hat{f}$ is an isometry from $L^2_{dR}$ to $L^2(\{\xi_d\}\cup \R^+, \rho)$, and we have 
\[
f(R) = \hat{f}(\xi_d)\phi_d(R) + \lim_{\mu\rightarrow\infty}\int_0^\mu \phi(R, \xi)\hat{f}(\xi)\rho(\xi)\,d\xi,
\]
the limits being in the suitable $L^2$-sense. 
\end{prop}
We also observe the following estimate describing classical $H^s_{dR}$-norms in terms of the distorted Fourier transform, and which follows by a simple interpolation argument: 
\begin{lem}\label{lem:Sobolev} Assume $s\geq 0$. Then we have 
\[
\big\|f(R)\big\|_{H^s(\R_+)}\lesssim \big\|\langle\xi\rangle^{\frac{s}{2}}\hat{f}(\xi)\big\|_{L^2_{d\rho}} + \big|\xi_d\big|,
\]
\end{lem}

When passing from the standard coordinates $r = |x|, t$ to the new ones $R = \lambda(t)r, \tau = \int_t^\infty \lambda(s)\,ds$, the time derivative will be replaced by a dilation type operator of essentially the form $\partial_{\tau} + \frac{\dot{\lambda}}{\lambda}R\partial_R$, and translation to the Fourier variables will require expressing the operator $R\partial_R$ in terms of the distorted Fourier transform. Specifically, we need to understand how $R\partial_R$ acts on $x_d, x(\xi)$ for a function 
\[
\tilde{\epsilon}(R) = x_d\phi_d(R) + \int_0^\infty x(\xi)\phi(R, \xi)\rho(\xi)\,d\xi.
\]
The precise result here comes also from \cite{KST}: 
\begin{thm}\label{thm:transferenceop}(\cite{KST}
We have the identity 
\[
\left(\begin{array}{c}\hat{(R\partial_R) f}(\xi_d)\\ \hat{(R\partial_R)f}(\xi)\end{array}\right) = (\tilde{\mathcal{A}} + \mathcal{K})\left(\begin{array}{c}\hat{f}(\xi_d)\\ \hat{f}(\xi)\end{array}\right), 
\]
where the matrix operators $\mathcal{A}, \mathcal{K}$ on the right are given by 
\[
\tilde{\mathcal{A}} = \left(\begin{array}{cc}0& 0\\0&-2\xi\partial_{\xi}-\frac32 - \frac{\rho'(\xi)\xi}{\rho(\xi)}\end{array}\right),\,\mathcal{K} = \left(\begin{array}{cc}\mathcal{K}_{dd}&\mathcal{K}_{dc}\\ \mathcal{K}_{cd}&\mathcal{K}_{cc}\end{array}\right),
\]
and the individual components of $\mathcal{K}$ are given by $\mathcal{K}_{dd} = -\frac12$, $\mathcal{K}_{dc}(\xi) = K_d(\xi)$ a smooth function rapidly decaying toward $\xi = +\infty$, 
\[
\mathcal{K}_{dc}f = -\int_0^\infty K_d(\xi) f(\xi)\rho(\xi)\,d\xi,
\]
and finally, $\mathcal{K}_{cc}$ is a Calderon-Zygmund type operator given by a kernel 
\[
K_0(\xi, \eta) = \frac{\rho(\eta)}{\xi - \eta}F(\xi, \eta), 
\]
where the function $F(\cdot, \cdot)$ is of regularity at least $C^2$ on $(0, \infty)\times (0, \infty)$, and satisfies the bounds 
\begin{align*}
\big|F(\xi, \eta)\big|\lesssim \begin{cases}\xi + \eta,\hspace{3.5cm}&\xi + \eta\leq 1,\\
(\xi+\eta)^{-1}(1+\big|\xi^{\frac12} - \eta^{\frac12}\big|)^{-N},&\xi+\eta\geq 1
\end{cases}
\end{align*}
\begin{align*}
\big|\partial_{\xi}F(\xi, \eta)\big| + \big|\partial_{\eta}F(\xi, \eta)\big|\lesssim \begin{cases}1,\hspace{4cm}&\xi + \eta\leq 1,\\
(\xi+\eta)^{-\frac32}(1+\big|\xi^{\frac12} - \eta^{\frac12}\big|)^{-N},&\xi+\eta\geq 1
\end{cases}
\end{align*}
\begin{align*}
\sum_{j+k = 2}\big|\partial_{\xi}^j\partial_{\eta}^kF(\xi, \eta)\big| + \big|\partial_{\eta}F(\xi, \eta)\big|\lesssim \begin{cases}(\xi+\eta)^{-\frac12},\hspace{3cm}&\xi + \eta\leq 1,\\
(\xi+\eta)^{-2}(1+\big|\xi^{\frac12} - \eta^{\frac12}\big|)^{-N},&\xi+\eta\geq 1
\end{cases}
\end{align*}
Here $N$ can be chosen arbitrarily (with the implicit constant depending on $N$). 
\end{thm}

\subsection{Description of the data perturbation in terms of the distorted Fourier transform}
\label{sec:datapert}
In the sequel, we shall mainly describe functions $f(R)$ in terms of their distorted Fourier transform $\hat{f}(\xi)$, $\hat{f}(\xi_d)$. In particular, we shall describe the precise class of data perturbations $(\epsilon_0, \epsilon_1)$ via properties of their distorted Fourier transforms: for a pair of functions $\big(x_0(\xi), x_1(\xi)\big)$, which will represent the continuous spectral part\footnote{Recall that $R = \lambda(t)r$, where for now $\lambda(t) = t^{-1-\nu}$.} of $R\epsilon_0$, and in a more roundabout way the continuous spectral part of  $R\epsilon_1(R)$, we introduce the following norm:
\begin{equation}\label{eq:Stildenorm}\begin{split}
\big\|(x_0, x_1)\big\|_{\tilde{S}}&: = \big\|x_0\big\|_{\tilde{S}_1} + \big\|x_1\big\|_{\tilde{S}_2}\\&: = \big\|\langle\xi\rangle^{\frac12+2\delta_0}\min\{\tau_{0,0}\xi^{\frac12},1\}^{-1}\xi^{\frac12-\delta_0}x_0\big\|_{L^2_{d\xi}} + \big\| \langle\xi\rangle^{\frac{1}{2}+2\delta_0}\xi^{-\delta_0}x_1\big\|_{L^2_{d\xi}}.  
\end{split}\end{equation}
Here $1\gg\delta_0>0$ is a small constant held fixed throughout, and the constant $\tau_{0,0}$ is defined via 
\[
\tau_{0,0}: = \int_{t_0}^\infty s^{-1-\nu}\,ds. 
\]
This norm is in fact exactly the same as the one used in \cite{CondBlow}. We easily observe that 
\begin{align*}
\big\|\langle\xi\rangle^{1+\delta_0}x_0\big\|_{L^2_{d\xi}} + \big\|\langle\xi\rangle^{\frac12+\delta_0}x_1\big\|_{L^2_{d\xi}}\lesssim_{\tau_{0,0}}\big\|(x_0, x_1)\big\|_{\tilde{S}},
\end{align*}
as well as 
\begin{align*}
\big\|\xi^{-\delta_0}x_0\big\|_{L^2_{d\xi}(0<\xi<1)}\lesssim_{\tau_{0,0}}\big\|(x_0, x_1)\big\|_{\tilde{S}},
\end{align*}
and so we find, setting\footnote{The notation $P_c$ means projection onto continuous spectral part, i. e. projecting away the discrete spectral part which is the multiple of $\phi_d$.} $(P_c\tilde{\epsilon}_0)(R) = \int_0^\infty \phi(R, \xi)x_0(\xi)\rho(\xi)\,d\xi$, we have 
\begin{align*}
&\big\|\chi_{R\leq C\tau_{0,0}}(P_c\tilde{\epsilon}_0)(R)\big\|_{H^{\frac32+2\delta_0}_{dR}}\\&\lesssim \big\|\chi_{R\leq C\tau_{0,0}}\int_0^1 \phi(R, \xi)x_0(\xi)\rho(\xi)\,d\xi\big\|_{H^{\frac32+2\delta_0}_{dR}}\\
& + \big\|\chi_{R\leq C\tau_{0,0}}\int_1^\infty \phi(R, \xi)x_0(\xi)\rho(\xi)\,d\xi\big\|_{H^{\frac32+2\delta_0}_{dR}}\\
&\lesssim (C\tau_{0,0})^{\frac12}\big\|\langle\xi\rangle^{\frac34}x_0(\xi)\rho(\xi)\big\|_{L^1_{d\xi}(\xi<1)} + \big\|\langle\xi\rangle^{\frac34+\delta_0}x_0(\xi)\big\|_{L^2_{d\rho}(\xi>1)}.
\end{align*}
We have used here that $\phi(R, \xi)$ us uniformly bounded, which follows from the preceding proposition. Then we have 
\begin{align*}
&\big\|\langle\xi\rangle^{\frac34}x_0(\xi)\rho(\xi)\big\|_{L^1_{d\xi}(\xi<1)}\lesssim \big\|\xi^{-\delta_0}x_0(\xi)\big\|_{L^2_{d\xi}(\xi<1)},\\&\big\|\langle\xi\rangle^{\frac34+\delta_0}x_0(\xi)\big\|_{L^2_{d\rho}(\xi>1)}\lesssim \big\|\langle\xi\rangle^{1+\delta_0}x_0\big\|_{L^2_{d\xi}(\xi>1)},
\end{align*}
where we have used the asymptotics for $\rho$ from the preceding proposition, and so we in fact have 
\begin{equation}\label{eq:HS}
\big\|\chi_{R\leq C\tau_{0,0}}(P_c\tilde{\epsilon}_0)(R)\big\|_{H^{\frac32+2\delta_0}_{dR}}\lesssim_{\tau_{0,0}}\big\|x_0\big\|_{\tilde{S}_1}. 
\end{equation}
Thus the 'physical data' corresponding to the distorted Fourier variable in $\tilde{S}_1$ is actually of regularity $H^{\frac32+}_{loc}$. 
To reconstruct the full perturbation $\epsilon_0$, we also need to prescribe the discrete spectral part $x_{0,d}$, and then set 
\begin{equation}\label{eq:x0epsilon0}
\epsilon_0(r) = R^{-1}\tilde{\epsilon}_0(R) = R^{-1}[x_{0,d}\phi_d(R) + \int_0^\infty x_0(\xi)\phi(R, \xi)\rho(\xi)\,d\xi],
\end{equation}
where $R = \lambda(t_0)r$ and $\lambda(t) = t^{-1-\nu}$. 
\\

The relation of the second Fourier variable $x_1(\xi)$ and $\epsilon_1$ is a bit more complicated, see \cite{CondBlow}, due to the fact that here all of $x_0, x_{0,d}, x_1, x_{1,d}$ are involved. This is due to the fact that the description of the perturbed solution $u(t, x)$ shall actually be in terms of the new variables\footnote{Actually, we shall be more specific later, and in fact introduce slightly perturbed $\lambda_{\gamma_1,\gamma_2}$, $\tau_{\gamma_1,\gamma_2}$ to get the right description.} $R = \lambda(t) r$, $\tau = \int_{t}^\infty \lambda(s)\,ds$, which mix time and space. Specifically, consider a function 
\[
\tilde{\epsilon}(\tau, R) = x_d(\tau)\phi_d(R) +  \int_0^\infty x(\tau, \xi)\phi(R, \xi)\rho(\xi)\,d\xi,\,\tilde{\epsilon} = R\epsilon.
\]
Then we obtain the relation 
\[
-\frac{R\epsilon_t}{\lambda} = \big(\partial_{\tau} + \frac{\dot{\lambda}}{\lambda}(R\partial_R - 1)\big)\tilde{\epsilon},\,\dot{\lambda} = \lambda_{\tau}. 
\]
Introducing the notation $\underline{x}: = \left(\begin{array}{c}x_d\\ x\end{array}\right)$ and passing to the Fourier variables by using Theorem~\ref{thm:transferenceop}, we find 
\[
-\left(\begin{array}{c}\mathcal{F}\big(\frac{R}{\lambda}\epsilon_t\big)\\ \langle\phi_d, \frac{R}{\lambda}\epsilon_t\rangle\end{array}\right) = \mathcal{D}_{\tau}\underline{x}(\tau,\cdot) + \beta_{\nu}(\tau)\mathcal{K}\underline{x}(\tau, \cdot),\,\beta_{\nu}(\tau) = \frac{\dot{\lambda}}{\lambda}(\tau),
\]
where we have introduced the important dilation type operator 
\[
\mathcal{D}_{\tau}: = \partial_{\tau} + \frac{\dot{\lambda}}{\lambda}\mathcal{A},\,\,\,\mathcal{A} = \left(\begin{array}{cc}0& 0\\0&-2\xi\partial_{\xi}-\frac52 - \frac{\rho'(\xi)\xi}{\rho(\xi)}\end{array}\right)
\]
More explicitly, we have 
\begin{equation}\label{eq:datatransference1}
-\mathcal{F}\big(\frac{R}{\lambda}\epsilon_t\big)\big|_{t = t_0} = x_1 + \beta_{\nu}(\tau_{0,0})\mathcal{K}_{cc}x_0 + \beta_{\nu}(\tau_{0,0})\mathcal{K}_{cd}x_{0d}
\end{equation}
\begin{equation}\label{eq:datatransference2}
-\langle\phi_d, \frac{R}{\lambda}\epsilon_t\rangle\big|_{t = t_0} = x_{1d} + \beta_{\nu}(\tau_{0,0})\mathcal{K}_{dd}x_{0d} + \beta_{\nu}(\tau_{0,0})\mathcal{K}_{dc}x_{0}
\end{equation}
where we have set $x_1 = \big(\partial_{\tau} - \frac{\dot{\lambda}}{\lambda}(2\xi\partial_{\xi} +\frac52 + \frac{\rho'(\xi)\xi}{\rho(\xi)})\big)x(\tau, \xi)\big|_{\tau = \tau_{0,0}}$, as well as $x_{1d} = \partial_{\tau}x_d(\tau)\big|_{\tau = \tau_{0,0}}$, and as before we use 
\[
\tau_{0,0} = \int_{t_0}^\infty s^{-1-\nu}\,ds,
\]
which thus corresponds to the new time variable with respect to the scaling law $\lambda(t) = t^{-1-\nu}$ evaluated at initial time $t = t_0$. 
\\

For future reference, we note that we shall sometimes use the notation $\mathcal{D}_{\tau} = \partial_{\tau} - \frac{\dot{\lambda}}{\lambda}(2\xi\partial_{\xi}+\frac52 +\frac{\rho'(\xi)\xi}{\rho(\xi)})$ when this operator acts on scalar functions $x(\tau, \xi)$, while it acts on vector valued functions $\underline{x}$ via the above formula. 
\\

Finally, the relations \eqref{eq:x0epsilon0} in conjunction with \eqref{eq:datatransference1}, \eqref{eq:datatransference2} give the translation from the data quadruple $(x_0, x_1)\in \tilde{S},\,(x_{0d}, x_{1d})\in \R^2$ to a data pair $(\epsilon, \epsilon_t)|_{t = t_0} = (\epsilon_0, \epsilon_1)$. An argument analogous to the one used to establish \eqref{eq:HS} implies then that $\epsilon_1\in H^{\frac12+2\delta_0}_{loc}$. 

\section{The main theorem and outline of the proof}

\subsection{The main theorem}
We shall now consider what happens to the evolution of the perturbed data $u_{\nu}[t_0]$, with $u_{\nu}$ as in Theorem~\ref{thm:KST}. In light of Theorem~\ref{thm:codim1stable}, we only expect such perturbations to yield a type II dynamics (backwards in time, i. e. for $t<t_0$), provided we impose a suitable co-dimension one condition on the perturbation imposed. That this is indeed all that is required follows from

\begin{thm}\label{thm:Main} Assume $0<\nu\ll 1$, and assume $t_0 = t_0(\nu)>0$ is sufficiently small, so that the solutions $u_{\nu}(t, x)$ in Theorem~\ref{thm:KST} exist on $(0, t_0]\times \R^3$. Let $\delta_1 = \delta_1(\nu)>0$ be small enough, and let 
$\mathcal{B}_{\delta_1}\subset \tilde{S}\times \R$ be the $\delta_1$-vicinity of $\big((0,0), 0\big)\in \tilde{S}\times \R$, where $\tilde{S}$ is the Banach space defined as the completion of $C_0^\infty(0,\infty)\times C_0^\infty(0,\infty)$ with respect to the norm \eqref{eq:Stildenorm}. 
\\
Then there is a Lipschitz function $\gamma_1: \mathcal{B}_{\delta_1}\rightarrow\R$, such that for any triple $\big((x_0, x_1), x_{0d}\big)\in \mathcal{B}_{\delta_1}$, the quadruple 
\[
\big((x_0, x_1), (x_{0d}, x_{1d})\big),\,x_{1d} = \gamma_1(x_0, x_1, x_{0d})
\]
determines a data perturbation pair $(\epsilon_0, \epsilon_1)\in H^{\frac32+2\delta_0}_{rad, loc}(\R^3)\times H^{\frac12+2\delta_0}_{rad, loc}(\R^3)$ via \eqref{eq:x0epsilon0}, \eqref{eq:datatransference1}, \eqref{eq:datatransference2}, and such that the perturbed initial data 
\begin{equation}\label{eq:perturbeddata1}
u_{\nu}[t_0] + (\epsilon_0, \epsilon_1)
\end{equation}
lead to a solution $\tilde{u}(t, x)$ on $(0,t_0]\times \R^3$ admitting the description 
\begin{align*}
\tilde{u}(t, x) = W_{\tilde{\lambda}(t)}(x) + \epsilon(t, x),\,\big(\epsilon(t, \cdot), \epsilon_t(t, \cdot)\big)\in H^{1+\frac{\nu}{2}-}_{loc}\times H^{\frac{\nu}{2}-}_{loc}
\end{align*}
where the parameter $\tilde{\lambda}(t)$ equals $\lambda(t)$ asymptotically 
\[
\lim_{t\rightarrow 0}\frac{\tilde{\lambda}(t)}{\lambda(t)} = 1,\,\lambda(t) = t^{-1-\nu}.
\]
In particular, the blow up phenomenon described in Theorem~\ref{thm:KST} is stable under a suitable co-dimension one class of data perturbations. 
\end{thm}
\begin{rem} We could have replaced $\tilde{\lambda}$ by $\lambda$ in the preceding theorem and included the arising modification in the error term $\epsilon(t, x)$. The formulation of the theorem emphasises part of the proof strategy, which shall indeed consist in a (very slight) modification of the scaling law $\lambda(t)= t^{-1-\nu}$ to force two important vanishing conditions. It is this part which is indeed analogous to the usual 'modulation method'. 
\end{rem}

\subsection{Outline of the proof}

The proof will consist of two stages, the first replacing the blow up solution $u_{\nu}(t, x)$ by a two parameter family $u_{approx}^{(\gamma_1,\gamma_2)}$ of approximate blow up solutions, where the parameters $\gamma_1,\gamma_2$ will depend on the perturbation $(\epsilon_0, \epsilon_1)$ and thus on the original data set $(x_0, x_1, x_{0d})$, and the second stage will involve completing the approximate solution $u_{approx}^{(\gamma_1,\gamma_2)}$ to an exact one of the form 
\[
u_{approx}^{(\gamma_1,\gamma_2)} + \epsilon(t, x),
\]
whose data at time $t = t_0$ will coincide with $u_{\nu}[t_0] + (\epsilon_0, \epsilon_1)$ at time $t = t_0$, {\it{provided we restrict the data to a suitable dilate of the light cone}} $r\leq C t_0$. In fact, we do not care about what happens outside of the light cone, as our solutions will remain regular there for simple a priori non-concentration of energy reasons, exactly as in \cite{KST}.  
\\
Explaining the reason for introducing $u_{approx}^{(\gamma_1,\gamma_2)}$ necessitates outlining the strategy for controlling the error term $\epsilon(t, x)$, which will be done via Fourier methods, exactly as was done in \cite{CondBlow}. 
\\

{\it{The method of \cite{CondBlow}.}} Assume we intend to construct a solution of the form $u(t, x) = u_{\nu}(t,x) + \epsilon(t,x)$, with $u_{\nu}$ as in Theorem~\ref{thm:KST}. Recall that $u_{\nu}$ consists of a bulk part $W_{\lambda(t)}(x)$ with $\lambda(t) = t^{-1-\nu}$ and an error part. Passing to the new variables $R = \lambda(t)r, \tau = \int_t^\infty \lambda(s)\,ds$, one derives the following equation for the variable $\tilde{\epsilon}(\tau, R): =R\epsilon(t, r)$, see also \cite{KST}, \cite{KS1}: 
\begin{equation}\label{eq:epseqn}\begin{split}
&(\partial_{\tau} + \dot{\lambda}\lambda^{-1}R\partial_R)^2 \tilde{\eps} - \beta_{\nu}(\tau)(\partial_{\tau} + \dot{\lambda}\lambda^{-1}R\partial_R)\tilde{\eps} + \calL\tilde{\eps}\\
&=\lambda^{-2}(\tau)R N(\eps)+\partial_\tau(\dot{\lambda}\lambda^{-1})\tilde{\eps};\,\beta_{\nu}(\tau) = \dot{\lambda}(\tau)\lambda^{-1}(\tau), 
\end{split}\end{equation}
where the operator $\calL$ is given by 
\[
\calL = -\partial_R^2 - 5W^4(R)
\]
and we have 
\[
RN(\eps) = 5(u_{\nu}^4 - u_0^4)\tilde{\eps} + RN(u_{\nu}, \tilde{\eps}),\,u_0 = W_{\lambda(t)}(x) = \frac{\lambda^{\frac12}(t)}{\big(1+\frac{(\lambda(t)|x|)^2}{3}\big)^{\frac12}},
\]
\[
RN(u_{\nu}, \tilde{\eps}) = R(u_{\nu}+\frac{\tilde{\eps}}{R})^5 - R u_{\nu}^5 - 5u_{\nu}^4\tilde{\eps}
\]
To solve this equation inside the forward light cone centered at the origin, one translates it to the Fourier side, i. e. one writes 
\begin{equation}\label{eq:FourierRep}
\tilde{\eps}(\tau, R) = x_d(\tau)\phi_d(R) + \int_0^\infty x(\tau, \xi)\phi(R, \xi)\rho(\xi)\,d\xi,
\end{equation}
see Prop.~\ref{prop:Fourier}. Taking advantage of Theorem~\ref{thm:transferenceop}, and using simple algebraic manipulations, see (2.3) in \cite{CondBlow}, one derives the following equation system in terms of the Fourier coefficients
\[
\underline{x}(\tau, \xi) = \left(\begin{array}{c}x_d(\tau)\\ x(\tau, \xi)\end{array}\right):
\]
\begin{equation}\label{eq:transport}
\big(\mathcal{D}_{\tau}^2 + \beta_{\nu}(\tau)\mathcal{D}_{\tau} + \underline{\xi}\big)\underline{x}(\tau, \xi) = \calR(\tau, \underline{x}) + \underline{f}(\tau, \underline{\xi}),
\end{equation}
where we have
\begin{equation}\label{eq:Rterms}
\calR(\tau, \underline{x})(\xi) = \Big(-4\beta_{\nu}(\tau)\mathcal{K}\mathcal{D}_{\tau}\underline{x} - \beta_{\nu}^2(\tau)(\mathcal{K}^2 + [\mathcal{A}, \mathcal{K}] + \mathcal{K} +  \beta_{\nu}' \beta_{\nu}^{-2}\mathcal{K})\underline{x}\Big)(\xi)
\end{equation}
 with $\beta_{\nu}(\tau) = \frac{\dot{\lambda}(\tau)}{\lambda(\tau)}$, and we set $\underline{f} = \left(\begin{array}{c}f_d\\ f\end{array}\right)$ where
 \begin{equation}\label{eq:fterms}\begin{split}
 &f(\tau, \xi) = \mathcal{F}\big( \lambda^{-2}(\tau)\big[5(u_{\nu}^4 - u_0^4)\tilde{\eps} + RN(u_{\nu}, \tilde{\eps})\big]\big)\big(\xi\big)\\
 &f_d(\tau) = \langle \lambda^{-2}(\tau)\big[5(u_{\nu}^4 - u_0^4)\tilde{\eps} + RN(u_{\nu}, \tilde{\eps})\big], \phi_d(R)\rangle.
 \end{split}\end{equation}
Also $ \mathcal{D}_{\tau} $ denotes the key operator 
 \[
 \mathcal{D}_{\tau} = \partial_{\tau} + \beta_{\nu}(\tau)\mathcal{A},\quad \mathcal{A} = \left(\begin{array}{cc}0&0\\0&\mathcal{A}_c\end{array}\right)
 \]
 and we have 
 \[
 \mathcal{A}_c = -2\xi\partial_{\xi} - \Big (\frac{5}{2}  + \frac{\rho'(\xi)\xi}{\rho(\xi)} \Big),
 \]
while $\mathcal{K}$ is as in Theorem~\ref{thm:transferenceop}. 
\\
As {\it{initial data}} for the problem \eqref{eq:transport}, we shall of course use 
\begin{equation}\label{eq:initdata}
\underline{x}(\tau_0) = \left(\begin{array}{c}x_{0d}\\ x_0(\xi)\end{array}\right),\,\mathcal{D}_{\tau}\underline{x}(\tau_0) = \left(\begin{array}{c}x_{1d}\\ x_1(\xi)\end{array}\right)
\end{equation}
where the components $x_{0,1}(\xi), x_{0d}$ shall be freely described (within the constraints of Theorem~\ref{thm:Main}), while the last component $x_{1d}$ shall be determined via a suitable Lipschitz function in terms of the first three components. 
This is again due to the exponential growth of the component $x_d(\tau)$ due to the unstable mode. The method of solution of \eqref{eq:transport} uses an iterative scheme, beginning with the zeroth iterate solving 
\begin{equation}\label{eq:linhom1}
\big(\mathcal{D}_{\tau}^2 + \beta_{\nu}(\tau)\mathcal{D}_{\tau} + \underline{\xi}\big)\underline{x}(\tau, \xi) = \underline{0},\,\underline{x}(\tau_0) = \left(\begin{array}{c}x_{0d}\\ x_0(\xi)\end{array}\right),\,\mathcal{D}_{\tau}\underline{x}(\tau_0) = \left(\begin{array}{c}x_{1d}\\ x_1(\xi)\end{array}\right).
\end{equation}
This can be solved explicitly as in Lemma 2.1 in \cite{CondBlow}, which we quote here: 
\begin{lem}\label{lem:linhom} The equation \eqref{eq:linhom1} is solved for the continuous spectral part $x(\tau, \xi)$ via the following parametrix: 
\begin{equation}\label{eq:linhomparam1}\begin{split}
x(\tau, \xi) = &\frac{\lambda^{\frac{5}{2}}(\tau)}{\lambda^{\frac{5}{2}}(\tau_0)}\frac{\rho^{\frac{1}{2}}(\frac{\lambda^{2}(\tau)}{\lambda^{2}(\tau_0)}\xi)}{\rho^{\frac{1}{2}}(\xi)}\cos\Big[\lambda(\tau)\xi^{\frac{1}{2}}\int_{\tau_0}^{\tau}\lambda^{-1}(u)\,du\Big]x_0\big(\frac{\lambda^{2}(\tau)}{\lambda^{2}(\tau_0)}\xi\big)\\
& + \frac{\lambda^{\frac{3}{2}}(\tau)}{\lambda^{\frac{3}{2}}(\tau_0)}\frac{\rho^{\frac{1}{2}}(\frac{\lambda^{2}(\tau)}{\lambda^{2}(\tau_0)}\xi)}{\rho^{\frac{1}{2}}(\xi)}\frac{\sin\Big[\lambda(\tau)\xi^{\frac{1}{2}}\int_{\tau_0}^{\tau}\lambda^{-1}(u)\,du\Big]}{\xi^{\frac12}}x_1\big(\frac{\lambda^{2}(\tau)}{\lambda^{2}(\tau_0)}\xi\big)\\
\end{split}\end{equation}
Moreover, writing $\underline{x}_0 = \left(\begin{array}{c}x_{0d}\\ x_0(\xi)\end{array}\right)$, $\underline{x}_1 = \left(\begin{array}{c}x_{1d}\\ x_1(\xi)\end{array}\right)$ and picking $\tau_0\gg 1$ sufficiently large, there is $c_d = 1 + O(\tau_0^{-1})$ as well as $\gamma_d = -|\xi_d|^{\frac12} + O(\tau_0^{-1})$ such that if we impose the co-dimension one condition 
\begin{equation}\label{eq:cond1}
x_{1d} = \gamma_d x_{0d},
\end{equation}
then the discrete spectral part of $\underline{x}(\tau,\xi)$ admits for any $\kappa>0$ the representation
\[
x_d(\tau) = \big(1+O_{\kappa}(\tau^{-1}e^{\kappa(\tau-\tau_0)})\big)e^{-|\xi_d|^{\frac12}(\tau - \tau_0)}c_d x_{0d}
\]
One also has for $i\geq 1$
\[
(-\partial_{\tau})^ix_d(\tau) = \big(1+O_{\kappa}(\tau^{-1}e^{\kappa(\tau-\tau_0)})\big)|\xi_d|^{\frac{i}{2}}e^{-|\xi_d|^{\frac12}(\tau - \tau_0)}c_d x_{0d}
\]
\end{lem}

This co-dimension one condition will have to be slightly modified in nonlinear ways for the higher iterates, but to leading order remains the same throughout and is responsible for the co-dimension one condition of Theorem~\ref{thm:Main}. 
\\

There are, however, two additional vanishing conditions in the work \cite{CondBlow}, imposed on the {\it{continuous spectral parts}} $x_{0,1}(\xi)$, and which arise due to the need to bound the nonlinear terms in $N(u_{\nu}, \tilde{\eps})$ in \eqref{eq:fterms}. These conditions arise when bounding $\frac{\tilde{\eps}(\tau, R)}{R}$ upon expressing $\tilde{\eps}(\tau, R)$ as in \eqref{eq:FourierRep} and inserting the parametrix \eqref{eq:linhomparam1} for the continuous spectral parts. This is the content of Prop. 3.1 in \cite{CondBlow}:
\begin{prop}\label{prop:lingrowthcond}Assume the data $(x_0, x_1)\in \langle\xi\rangle^{-1-\delta_0}\xi^{0+\delta_0}L^2_{d\xi}\times  \langle\xi\rangle^{-\frac{1}{2}-\delta_0}\xi^{0+\delta_0}L^2_{d\xi}$. Furthermore, assume that we have the vanishing relations 
 \begin{equation}\label{eq:vanishing}
 \int_0^\infty\frac{\rho^{\frac{1}{2}}(\xi)x_0(\xi)}{\xi^{\frac{1}{4}}}\cos[\nu\tau_0\xi^{\frac{1}{2}}]\,d\xi = 0,\,\int_0^\infty\frac{\rho^{\frac{1}{2}}(\xi)x_1(\xi)}{\xi^{\frac{3}{4}}}\sin[\nu\tau_0\xi^{\frac{1}{2}}]\,d\xi = 0. 
 \end{equation}
 at time $\tau = \tau_0$. Assume that $x(\tau, \xi)$ is given by \eqref{eq:linhomparam1}. Then the function\footnote{Here $P_c$ denotes the projection onto the continuous spectral part} $P_c\tilde{\epsilon}(\tau, R)$ represented by the Fourier coefficients $x(\tau, \xi)$ via 
 \[
 P_c\tilde{\epsilon}(\tau, R) = \int_0^\infty \phi(R, \xi)x(\tau, \xi)\rho(\xi)\,d\xi
 \]
 satisfies 
 \[
 P_c\tilde{\epsilon}(\tau, R) = \tilde{\epsilon}_1(\tau, R) + \tilde{\epsilon}_2(\tau, R),
 \]
 where we have 
 \[
 \big\|\frac{\tilde{\epsilon}_1(\tau, R)}{R}\big\|_{L^\infty_{dR}}\lesssim \big\|(\langle\xi\rangle^{\frac12+2\delta_0}\xi^{\frac12-\delta_0}x_0, \langle\xi\rangle^{\frac{1}{2}+2\delta_0}\xi^{-\delta_0}x_1)\big\|_{L^2_{d\xi}}
 \]
 \[
 \big\| \tilde{\epsilon}_2(\tau, R)\big\|_{L^\infty_{dR}}\lesssim \tau \big\|(\langle\xi\rangle^{\frac12+2\delta_0}\xi^{\frac12-\delta_0}x_0, \langle\xi\rangle^{\frac{1}{2}+2\delta_0}\xi^{-\delta_0}x_1)\big\|_{L^2_{d\xi}}
 \]
 Here $\delta_0>0$ is the small constant used to define $\tilde{S}$ in \eqref{eq:Stildenorm}. 
  \end{prop}
  We note here that the growth of $\tilde{\epsilon}_2(\tau, R)$ is precisely due to the growth of the 'resonant part' of $\tilde{\eps}$, i. e. a multiple of the resonance $\phi_0(R)$, see the discussion preceding Prop. ~\ref{prop:Fourier}. 
We also observe that the expressions $\cos[\nu\tau_0\xi^{\frac{1}{2}}]$, $\sin[\nu\tau_0\xi^{\frac{1}{2}}]$ can alternatively be written as 
\[
\cos[\lambda(\tau_0)\xi^{\frac12}\int_{\tau_0}^\infty\lambda^{-1}(s)\,ds],\,\sin[\lambda(\tau_0)\xi^{\frac12}\int_{\tau_0}^\infty\lambda^{-1}(s)\,ds],
\]
upon noting that in terms of the variable $\tau\in [\tau_0, \infty)$, we have (abuse of notation) $\lambda(\tau) = c(\nu)\tau^{-1-\nu^{-1}}$. 
As the parametrix in \eqref{eq:linhomparam1} is valid for arbitrary $\lambda$, we see that the generalisation of the vanishing conditions \eqref{eq:vanishing} to more general $\lambda(t)\sim t^{-1-\nu}$ as $t\rightarrow\infty$ becomes (again upon passing to the new time variable $\tau = \int_t^\infty\lambda(s)\,ds$) 
 \begin{equation}\label{eq:vanishing1}\begin{split}
 &\int_0^\infty\frac{\rho^{\frac{1}{2}}(\xi)x_0(\xi)}{\xi^{\frac{1}{4}}}\cos[\lambda(\tau_0)\xi^{\frac12}\int_{\tau_0}^\infty\lambda^{-1}(s)\,ds]\,d\xi = 0,\\&\int_0^\infty\frac{\rho^{\frac{1}{2}}(\xi)x_1(\xi)}{\xi^{\frac{3}{4}}}\sin[\lambda(\tau_0)\xi^{\frac12}\int_{\tau_0}^\infty\lambda^{-1}(s)\,ds]\,d\xi = 0. 
 \end{split}\end{equation}

The key for proving Theorem~\ref{thm:Main} shall be to get rid of these two conditions on the continuous spectral part of the data, and thence reduce things to the unique condition involving the discrete spectral part. 
\\
To achieve this, we shall pass from the splitting $u(t, x) = u_{\nu}(t, x) + \epsilon(t, x)$ to a slightly modified one 
\begin{equation}\label{eq:modifieddecomp}
u(t, x) = u_{approx}^{(\gamma_1,\gamma_2)}(t, x) + \bar{\epsilon}(t, x), 
\end{equation}
where 
\[
u_{approx}^{(\gamma_1,\gamma_2)}(t, x) = W_{\lambda_{\gamma_1,\gamma_2}(t)}(x) + \eta(t, x)
\]
will be an approximate solution built analogously to $u_{\nu}(t, x)$ (as in Theorem~\ref{thm:KST}), but where the bulk part $W_{\lambda_{\gamma_1,\gamma_2}(t)}(x)$ is now scaled according to 
\begin{equation}\label{eq:lambdagammaonetwo}
\lambda_{\gamma_1,\gamma_2}(t): = \Big(1+\gamma_1\cdot \frac{t^{k_0\nu}}{\langle t^{k_0\nu} \rangle} + \gamma_2\log t\cdot \frac{t^{k_0\nu}}{\langle t^{k_0\nu} \rangle}\Big)t^{-1-\nu}, \,k_0 = [N\nu^{-1}], 
\end{equation}
for some $N\gg 1$, and this is clearly asymptotically equal to $\lambda(t)$: $\lim_{t\rightarrow 0}\frac{\lambda_{\gamma_1,\gamma_2}(t)}{\lambda(t)} = 1$. 
As we shall want to match the data \eqref{eq:perturbeddata1} at time $t = t_0$, at least in the forward light cone, we impose for some $C>1$ the condition 
\begin{equation}\label{eq:newdataperutbation}
\chi_{r\leq Ct_0}u_{\nu}[t_0] + (\epsilon_0, \epsilon_1) = \chi_{r\leq Ct_0}u_{approx}^{(\gamma_1,\gamma_2)}[t_0] + (\bar{\epsilon}_0, \bar{\epsilon}_1)
\end{equation}
on the data $(\bar{\epsilon}_0, \bar{\epsilon}_1) = \bar{\epsilon}[t_0]$ of the 'new perturbation' $\bar{\epsilon}$ at $t = t_0$. 
\\
We note that the proper re-scaled variables to describe $\bar{\epsilon}$ are now given by 
\begin{equation}\label{eq:changeofframe}
\tau_{\gamma_1,\gamma_2}: = \int_{t}^\infty \lambda_{\gamma_1,\gamma_2}(s)\,ds,\,R_{\gamma_1,\gamma_2} = \lambda_{\gamma_1,\gamma_2}(t)r, 
\end{equation}
In analogy to \eqref{eq:x0epsilon0}, \eqref{eq:datatransference1}, \eqref{eq:datatransference2}, we can then determine $(x_0^{(\gamma_1,\gamma_2)}, x_1^{(\gamma_1,\gamma_2)})$ as well as $(x_{0d}^{(\gamma_1,\gamma_2)}, x_{1d}^{(\gamma_1,\gamma_2)})$, such that 
\begin{equation}\label{eq:x0barepsilon0}
\bar{\epsilon}_0(r(R_{\gamma_1,\gamma_2})) = R_{\gamma_1,\gamma_2}^{-1}[x_{0d}^{(\gamma_1,\gamma_2)}\phi_d(R_{\gamma_1,\gamma_2}) + \int_0^\infty x_0^{(\gamma_1,\gamma_2)}(\xi)\phi(R_{\gamma_1,\gamma_2}, \xi)\rho(\xi)\,d\xi]
\end{equation}
\begin{equation}\label{eq:datatransference3}\begin{split}
-\mathcal{F}\big(\frac{R_{\gamma_1,\gamma_2}}{\lambda_{\gamma_1,\gamma_2}}\bar{\epsilon}_1\big)\big|_{t = t_0}& = x_1^{(\gamma_1,\gamma_2)} + \beta_{\nu}^{(\gamma_1,\gamma_2)}(\tau_{\gamma_1,\gamma_2})\big|_{t = t_0}\mathcal{K}_{cc}x_0^{(\gamma_1,\gamma_2)}\\&\hspace{4.5cm} + \beta_{\nu}^{(\gamma_1,\gamma_2)} (\tau_{\gamma_1,\gamma_2} )\big|_{t = t_0}\mathcal{K}_{cd}x_{0d}
\end{split}\end{equation}
\begin{equation}\label{eq:datatransference4}\begin{split}
-\langle\phi_d, \frac{R_{\gamma_1,\gamma_2}}{\lambda_{\gamma_1,\gamma_2}}\bar{\epsilon}_1\rangle\big|_{t = t_0}& = x_{1d}^{(\gamma_1,\gamma_2)} + \beta_{\nu}^{(\gamma_1,\gamma_2)}(\tau_{\gamma_1,\gamma_2})\big|_{t = t_0}\mathcal{K}_{dd}x_{0d}^{(\gamma_1,\gamma_2)}\\&\hspace{4.5cm} + \beta_{\nu}^{(\gamma_1,\gamma_2)}(\tau_{\gamma_1,\gamma_2})\big|_{t = t_0}\mathcal{K}_{dc}x_{0}^{(\gamma_1,\gamma_2)},
\end{split}\end{equation}
and we use the notation $\beta_{\nu}^{(\gamma_1,\gamma_2)} = \frac{\dot{\lambda}_{\gamma_1,\gamma_2}}{\lambda_{\gamma_1,\gamma_2}}$, where the $\dot{\lambda}_{\gamma_1,\gamma_2}$ indicates differentiation with respect to the new time variable $\tau_{\gamma_1,\gamma_2}$. 
\\

At this stage, we can succinctly formulate the key technical steps required to complete the proof of Theorem~\ref{thm:Main}. 
\begin{itemize}
\item {\it{Show that given $x_0, x_1, x_{0d}$ as in the statement of Theorem~\ref{thm:Main}, there are unique choices of $\gamma_1,\gamma_2, x_{1d}$ such that the Fourier variables 
\[
x_0^{(\gamma_1,\gamma_2)} , x_1^{(\gamma_1,\gamma_2)} , x_{0d}^{(\gamma_1,\gamma_2)} , x_{1d}^{(\gamma_1,\gamma_2)}
\]
satisfy the vanishing relations}} 
\begin{equation}\label{eq:vanishing2}\begin{split}
 &\int_0^\infty\frac{\rho^{\frac{1}{2}}(\xi)x_0^{(\gamma_1,\gamma_2)} (\xi)}{\xi^{\frac{1}{4}}}\cos[\lambda_{(\gamma_1,\gamma_2)} (\tau_0)\xi^{\frac12}\int_{\tau_0}^\infty\lambda_{(\gamma_1,\gamma_2)} ^{-1}(s)\,ds]\,d\xi = 0,\\&\int_0^\infty\frac{\rho^{\frac{1}{2}}(\xi)x_1^{(\gamma_1,\gamma_2)} (\xi)}{\xi^{\frac{3}{4}}}\sin[\lambda_{(\gamma_1,\gamma_2)} (\tau_0)\xi^{\frac12}\int_{\tau_0}^\infty\lambda_{(\gamma_1,\gamma_2)} ^{-1}(s)\,ds]\,d\xi = 0. 
 \end{split}\end{equation}
{\it{as well as condition \eqref{eq:cond1}}}.
\item {\it{Using the splitting 
\[
u(t, x) = u_{approx}^{(\gamma_1,\gamma_2)}(t, x) + \bar{\epsilon}(t, x),
\]
pass to the Fourier representation 
\[
R_{\gamma_1,\gamma_2}\bar{\epsilon} = x_d^{(\gamma_1,\gamma_2)}\phi_d(R_{\gamma_1,\gamma_2}) + \int_0^\infty x^{(\gamma_1,\gamma_2)}(\tau_{\gamma_1,\gamma_2}, \xi)\phi(R_{\gamma_1,\gamma_2}, \xi)\rho(\xi)\,d\xi,
\]
and use the analog of \eqref{eq:transport} to construct a solution 
\[
\underline{x}^{(\gamma_1,\gamma_2)}(\tau, \xi) = \left(\begin{array}{c}x_d^{(\gamma_1,\gamma_2)}(\tau)\\ x^{(\gamma_1,\gamma_2)}(\tau, \xi)\end{array}\right)
\]
'closely matching' the initial conditions $x_{0,1}^{(\gamma_1,\gamma_2)}, x_{0,1d}^{(\gamma_1,\gamma_2)}$. In fact, the method from \cite{CondBlow} furnishes such a solution with data
\begin{align*}
&\underline{x}^{(\gamma_1,\gamma_2)}(\tau_{\gamma_1,\gamma_2}, \xi)\big|_{t = t_0} =  \left(\begin{array}{c}x_{0d}^{(\gamma_1,\gamma_2)} + \triangle x_{0d}^{(\gamma_1,\gamma_2)}\\ x^{(\gamma_1,\gamma_2)}_0(\xi) + \triangle x^{(\gamma_1,\gamma_2)}_0(\xi)\end{array}\right)\\
&\mathcal{D}_{\tau}\underline{x}^{(\gamma_1,\gamma_2)}(\tau_{\gamma_1,\gamma_2}, \xi)\big|_{t = t_0} =  \left(\begin{array}{c}x_{1d}^{(\gamma_1,\gamma_2)} + \triangle x_{1d}^{(\gamma_1,\gamma_2)}\\ x^{(\gamma_1,\gamma_2)}_1(\xi) + \triangle x^{(\gamma_1,\gamma_2)}_1(\xi)\end{array}\right)\\
\end{align*}
}} 
\item {\it{Translating things back to the original  perturbation in terms of the old Fourier variables $x_0(\xi), x_1(\xi), x_{0d}, x_{1d}$, show that we have found an initial data pair corresponding to Fourier variables 
\[
x_0(\xi) +\triangle x_0(\xi), x_1(\xi)+\triangle x_1(\xi), x_{0d} +\triangle x_{0d}, x_{1d}+\triangle x_{1d},
\]
where the corrections $\triangle x_0(\xi)$ etc are small and depend in Lipschitz continuous fashion on the original data $x_0$ etc, with small Lipschitz constant. 
}}

\end{itemize}

\section{Construction of a two parameter family of approximate blow up solutions}

Here we construct the approximate blow up solutions $u_{approx}^{(\gamma_1,\gamma_2)}$ which replace the previous $u_{\nu}(t, x)$, see the decomposition \eqref{eq:modifieddecomp}. The idea behind the construction is to closely mimic the steps in section 2 of \cite{KS1}, which in turn follows closely the steps in section 2 of \cite{KST}. in particular, to describe the successive corrections in the construction, we shall rely on the same algebras of functions as in  \cite{KS1}. 
\\

In the sequel, we shall work mostly with respect to the scaling parameter $\lambda_{\gamma_1,\gamma_2}(t)$ given by \eqref{eq:lambdagammaonetwo}. To simplify the notation, we shall henceforth set 
\begin{equation}\label{eq:restlambda}
\lambda(t): = \lambda_{\gamma_1,\gamma_2}(t),\,R: = \lambda_{\gamma_1,\gamma_2}(t)r,\,R_{0,0}: = \lambda_{0,0}(t)r
\end{equation}

\begin{thm}\label{thm:main approximate} Let $\nu>0$, $t_0 = t_0(\nu)>0$ sufficiently small, and $\gamma_{1,2}\ll1 $. Also, let $N\gg 1$, $k_0 = [N\nu^{-1}]$,  $k_* = [\frac12N \nu^{-1}]$. Then there exists an approximate solution $u_{\text{approx}} = u_{\text{approx}}^{(\gamma_{1,2})}
$ for $\Box u = -u^5$ of the form (putting $\lambda(t): = \lambda_{\gamma_{1,2}}(t)$ for simplicity)
\[
u_{\text{approx}}^{(\gamma_{1,2})} = \lambda^{\frac12}(t)\big[W(R) + \frac{c}{(\lambda t)^2}R^2(1+R^2)^{-\frac12} + O((\lambda t)^{-2}\log RR^2(1+R^2)^{-\frac32})\big],
\]
such that the corresponding error 
\[
e_{\text{approx}} = \Box u_{\text{approx}} + u_{\text{approx}}^5
\]
is of the form 
\begin{align*}
&t^2e_{\text{approx}} \\&= [\big|\gamma_1\big| + \big|\gamma_2\big|]\big[O\big(\log t\frac{\lambda^{\frac12}R}{(\lambda t)^{k_0 + 4}}(1+(1-a)^{\frac12 + \frac{\nu}{2}})\big)\\&\hspace{4cm} + O\big(\log t\frac{\lambda^{\frac12}}{(\lambda t)^{k_0 + 2}}R^{-1}(1+(1-a)^{\frac12 + \frac{\nu}{2}})\big)\big]
\end{align*}
and such that the above expansions may be formally differentiated, where we use the notation $a = \frac{r}{t}$. Furthermore, writing $u_{\text{approx}}^{(\gamma_{1,2})} = u_{\text{approx}}^{(\gamma_{1,2})}(t, r, \gamma_{1,2}, \nu)$ we have the $\gamma$-dependence
\[
\partial_{\gamma_1}u_{\text{approx}}^{(\gamma_{1,2})} = O(t^{k_0\nu}\lambda^{\frac12}\frac{R}{(\lambda t)^2}),\,
\]
with symbol type behaviour with respect to the $\partial_{t,r}$ derivatives up to order two, and similarly for 
\[
\partial_{\gamma_2}u_{\text{approx}}^{(\gamma_{1,2})} = O(t^{k_0\nu}\log t\lambda^{\frac12}\frac{R}{(\lambda t)^2}),\,
\]
\end{thm}
\begin{rem}\label{rem:gamma smooth} The key point here is the last part, which ensures that the $\gamma$ dependent part of the solutions $u_{\text{approx}}^{(\gamma_{1,2})}$ is smoother than the solutions themselves (they are only of class $H^{1+\frac{\nu}{2}-}$ regularity). 

\end{rem}
\begin{rem}\label{rem:gamma zero}Observe from the preceding construction that $e_{approx} = 0$ provided $\gamma_1 = \gamma_2 = 0$. Thus in that case the function $u_{approx}^{(0,0)}(t, x) = u_{\nu}(t, x)$ reproduces an exact solution as in Theorem~\ref{thm:KST}. 
\end{rem}
\begin{proof} We shall obtain the functions $u_{\text{approx}}^{(\gamma_{1,2})}$ by adding corrections $v_j$ to the bulk part $u_0: = W_{\lambda(t)}(r)$, the latter as in the paragraph following \eqref{eq:epseqn}. The precise description of these corrections is a bit cumbersome, but in principle elementary, as they arise by solving certain explicit ordinary differential equations. 
\\
The following definitions come directly from \cite{KS1}: 

\begin{defn}\label{def:Q}     
  We define $\mathcal Q$ to be the algebra of continuous functions $q:[0,1] \to
  \R$ with the following properties:

  (i) $q$ is analytic in $[0,1)$ with an even expansion at $0$ and with $q(0)=0$.

  (ii) Near $a=1$ we have an  expansion of the
  form
  \[
\begin{split}
  q(a) = q_0(a) + \sum_{i =1}^\infty  (1-a)^{ \beta(i) +1}\sum_{ j=0}^{\infty}
    q_{ij}(a) (\ln (1-a))^j
  \end{split}
  \]
  with analytic coefficients $q_0$, $q_{ij}$; if $\nu$
  is irrational, then $q_{ij}=0$ if~$j>0$.  The $\beta(i)$ are of the form
  \EQ{\label{menge}
  \sum_{k\in K} \big( (2k-3/2)\nu-1/2 \big) + \sum_{k\in K'} \big( (2k-1/2)\nu-1/2 \big)  
  }
  where $K,K'$ are finite sets of positive integers. 
  Moreover, only finitely many of the $q_{ij}$ are nonzero. 
\end{defn}

We remark that the exponents of $1-a$ in the above series all exceed~$\f12$ 
because of $\nu>0$.  For the errors $e_k$ we
introduce

\begin{defn}
   $\mathcal Q'$ is the space of continuous functions $q:[0,1) \to
  \R$ with the following properties:

  (i) $q$ is analytic in $[0,1)$ with an even expansion at $0$. 

  (ii) Near $a=1$ we have an expansion of the form
  \[\begin{split}
  q(a) =  q_0(a) + \sum_{i =1}^\infty  (1-a)^{ \beta(i)}\sum_{ j=0}^{\I}
    q_{ij}(a) (\ln (1-a))^j
  \end{split}
  \]
  with analytic coefficients $q_0$, $q_{ij}$, of which only finitely many are nonzero. The $\beta(i)$ are as above. 
\label{def:Qp}\end{defn}

By construction, $\mathcal Q\subset \mathcal Q'$. The family
$\cQ'$ is obtained by applying $a^{-1}\partial_a$ to
the algebra~$\cQ$. The exact number of $\log(1-a)$ factors can of course
be determined, but is irrelevant for our purposes. 
\\

The next definition is also taken from \cite{KS1}, except that we formulate it in terms of the variable $R_{0,0}$, which is independent of $\gamma_1,\gamma_2$. This shall be important in clarifying which of the corrections terms are independent of $\gamma_{1,2}$, and which indeed depend on these variables. 
\\
Introduce the variables $b(t) = \mu_{0,0}(t)^{-1},\,\mu_{0,0}(t) = \lambda_{0,0}(t)\cdot t$, as well as $b_1(t)$, which will represent $\frac{\log t}{\mu_{0,0}(t)} = \frac{\log t}{t\lambda_{0,0}(t)}$. Then
\begin{defn}\label{defn:functionspace}(a) $S^m(R_{0,0}^k(\log R_{0,0})^l, \mathcal{Q})$ is the class of analytic functions 
\[
v: [0,\infty)\times [0,1]\times [0, b_0]\times [0, b_0]\longrightarrow \R
\]
such that 
\begin{itemize}
\item $v$ is analytic as a function of $R_{0,0}, b, b_1$ and $v: [0, \infty)\times [0, b_0]\times[0, b_0]\longrightarrow \mathcal{Q}$.
\item $v$ vanishes of order $m$ relative to $R$, and $R^{-m}v$ has an even Taylor expansion at $R_{0,0} = 0$. 
\item $v$ has a convergent expansion  at $R_{0,0} = +\infty$. 
\[
v(R_{0,0}, a, b, b_1) = \sum_{i=0}^\infty \sum_{j=0}^{l+i}c_{ij}(a, b, b_1)R_{0,0}^{k-i}(\log R_{0,0})^j
\]
where the coefficients $c_{ij}(\cdot, b)\in \mathcal{Q}$ and $c_{ij}(a, b, b_1)$ are analytic in $b, b_1\in [0, b_0]$ for all $0\leq a\leq 1$. 
\end{itemize}
(b) $IS^m(R_{0,0}^k(\log R_{0,0})^l, \mathcal{Q})$ is the class of analytic functions $w$ on the cone $C_0$ which can be represented as 
\begin{align*}
&w(r, t) = v(R_{0,0}, a, b, b_1),\,v\in S^m(R_{0,0}^k(\log R_{0,0})^l, \mathcal{Q}),\,b = \frac{1}{\mu_{0,0}(t)},\,b_1 = \frac{\log t}{\mu_{0,0}(t)},\\&\mu_{0,0}(t) = t\cdot\lambda_{0,0}(t).
\end{align*}
(c) Denote by $\mathcal{Q}_{smooth}$ the algebra of continuous functions $q: [0,1]\longrightarrow\R$ with the following properties: 
\begin{itemize}
\item $q$ is analytic in $[0,1)$ with an even expansion at $0$ and with $q(0) = 0$.
\item Near $a = 1$ we have an expansion of the form 
\[
q(a) = q_0(a) + \sum_{i=1}^\infty (1-a)^{\beta(i)+ 1}\sum_{j=0}^\infty q_{ij}(a)\big(\log(1-a)\big)^j
\]
with analytic coefficients $q_0, q_{ij}$. The $\beta(i)$ are of the form 
\[
\sum_{k\in K,\,k\geq [N\nu^{-1}]}a_k\big((k-\frac12)\nu - \frac12\big)
\]
where $K$ consist of finite sets of natural numbers and $a_k\in \N$. Only finitely many of the $q_{ij}(a)$ are non-zero. 
\end{itemize}
Then define $S^m(R_{0,0}^k(\log R_{0,0})^l, \mathcal{Q}_{smooth}), IS^m(R_{0,0}^k(\log R_{0,0})^l, \mathcal{Q}_{smooth})$ as in (a), (b) above. We shall also use the notation $IS^m(R_{0,0}^k(\log R_{0,0})^l)$ to denote functions analytic in $b, b_1, R_{0,0}$ with the indicated vanishing and decay properties. 
\end{defn}

Observe that functions in $\mathcal{Q}_{smooth}$ are at least of regularity $C^{N+\frac12 - \frac{\nu}{2}-}$ at $a = 1$, and we can extend them past the light cone $a = 1$ by replacing $(1-a)$ by $|1-a|$ in the logarithmic terms. 
\\

The proof now proceeds by first building a solution $u_{prelim}$ by solving suitable elliptic problems approximating the wave equation \eqref{eq:critmain}, and finally adding a further correction to produce the $u_{approx}$, by solving a suitable wave equation via the parametrix method of \cite{KST}, \cite{KS1}. The method here in particular makes it clear that when $\gamma_1 = \gamma_2 = 0$ we simply reproduce the solutions if \cite{KST}, \cite{KS1}. 
To construct the preliminary approximate solution, we use 

\begin{lem}\label{lem:corrections} For any $k_*: = [\frac12N\nu^{-1}]\geq k\geq 1$ there exist corrections $v_{2k}, v_{2k-1}$ such that the approximations $u_{2k-1} = u_0+\sum_{j=1}^{2k-1}v_j$, $u_{2k} = u_0+\sum_{j=1}^{2k}v_j$ generate errors $e_{2k-1}, e_{2k}$ as below: 
\begin{align}
  v_{2k-1} &\in \frac{\lambda_{0,0}^{\frac12}}{\mu_{0,0}(t)^{2k}} IS^2(R_{0,0} \, (\log R_{0,0})^{m_{k}},\mathcal{Q})
  \label{v2k-1}\\
  t^2 e_{2k-1} &\in \frac{\lambda_{0,0}^{\frac12}}{\mu_{0,0}(t)^{2k}} IS^0(R_{0,0}\, (\log R_{0,0})^{p_{k}}, \mathcal{Q}')
  \label{e2k-1}\\
  v_{2k} &\in \frac{\lambda_{0,0}^{\frac12}}{\mu_{0,0}(t)^{2k+2}} IS^2(R_{0,0}^3\, (\log R_{0,0})^{p_{k}},\mathcal{Q}) \label{v2k}\\
  t^2 e_{2k} &\in \frac{\lambda_{0,0}^{\frac12}}{\mu_{0,0}(t)^{2k}} \big[ IS^0(R_{0,0}^{-1} \, (\log R_{0,0})^{q_{k}} ,\mathcal{Q})
   + b^2 IS^0(R_{0,0} \, (\log R_{0,0})^{q_{k}} ,\mathcal{Q}')    \big]
  \label{e2k}\end{align}
Here the functions $v_{2k-1}, v_{2k}$ are independent of $\gamma_{1,2}$, but not the errors $e_{2k-1}, e_{2k}$. 
Furthermore, we may pick two more corrections $v_{smooth,1}, v_{smooth, 2}$, such that 
\begin{align}
\partial_{\gamma_1}v_{smooth,1} &\in \frac{\lambda_{0,0}^{\frac12}}{\mu_{0,0}(t)^{k_0+2}} IS^2(R_{0,0},\mathcal{Q}_{smooth}),
  \label{vsmooth1}\\
  \partial_{\gamma_2}v_{smooth,1} &\in \log t\frac{\lambda_{0,0}^{\frac12}}{\mu_{0,0}(t)^{k_0+2}} IS^2(R_{0,0},\mathcal{Q}_{smooth}),
    \label{vsmooth11}\\
 \partial_{\gamma_1}v_{smooth,2} &\in \frac{\lambda_{0,0}^{\frac12}}{\mu_{0,0}(t)^{k_0+4}} IS^2(R_{0,0}^3 ,\mathcal{Q}_{smooth}),
   \label{vsmooth2}\\
  \partial_{\gamma_2}v_{smooth,2} &\in \log t\frac{\lambda_{0,0}^{\frac12}}{\mu_{0,0}(t)^{k_0+4}} IS^2(R_{0,0}^3 ,\mathcal{Q}_{smooth}),
    \label{vsmooth22}\\
 \end{align}
 such that the final error generated by $u_{\text{prelim}}: = u_0+\sum_{j=1}^{2k_*-1}v_j + \sum_{a = 1,2}v_{smooth,a}$ satisfies 
 \begin{align*}
&t^2e_{\text{prelim}}: = t^2(\Box u_{\text{prelim}} + u_{\text{prelim}}^5)\\
& \in \gamma_1  \frac{\lambda_{0,0}^{\frac12}}{\mu_{0,0}(t)^{k_0+2}} \big[IS^0(R_{0,0}^{-1}, \mathcal{Q}) + b^2IS^0(R_{0,0}, \mathcal{Q})\big]\\&
  + \gamma_2\log t \frac{\lambda_{0,0}^{\frac12}}{\mu_{0,0}(t)^{k_0+2}} \big[ IS^0(R_{0,0}^{-1}, \mathcal{Q}) +  b^2IS^0(R_{0,0}, \mathcal{Q})\big]
  + t^2\tilde{e}_{\text{prelim}}, 
 \end{align*}
 where the remaining error $t^2\tilde{e}_{\text{prelim}}$ does not depend on $\gamma_{1,2}$ and resides in 
 \[
 t^2\tilde{e}_{\text{prelim}}\in \frac{\lambda_{0,0}^{\frac12}}{\mu_{0,0}(t)^{2k_*}} IS^0(R_{0,0}\, (\log R_{0,0})^{p_{k_*}}, \mathcal{Q}')
  \label{e2k*-1}
  \]  
\end{lem}
\begin{proof} 
We follow closely the procedure in \cite{KS1}, section 2. The only novelty is that we perturb around $u_0 = \lambda^{\frac12}(t)W(\lambda(t)r)$ as opposed to $\lambda_{0,0}^{\frac12}(t)W(\lambda_{0,0}(t)r)$, which will generate additional error terms during the construction of the $v_j$, $1\leq j\leq 2k_*-1$. We relegate these to the end of the procedure, and use the final two corrections $v_{smooth,a}$ to decimate this remaining error, leaving only $e_{\text{prelim}}$. 
\\

{\bf{Step 0}}: We put $u_0(t, r) = \lambda^{\frac12}(t)W(R)$, $R = \lambda(t)r$, $\lambda(t) = \lambda_{\gamma_1,\gamma_2}(t)$. Then (with $\cD = \frac12 + R\partial_R$)
\EQ{
e_0:=\partial_t^2u_0 &= \la^{\f12}(t) \Big [ \Big( \f{\la'}{\la}\Big)^2(t) (\cD^2 W)(R) + \Big(\f{\la'}{\la}\Big)^{\prime}(t) (\cD W)(R)  \Big] \\
t^2 e_0 &=: \la_{0,0}^{\f12}(t) \Big [ \om_1  \f{1-R_{0,0}^2/3}{(1+R_{0,0}^2/3)^{\f32}}   + \om_2 \f{ 9-30R_{0,0}^2+R_{0,0}^4}{(1+R_{0,0}^2/3)^{\f52} }  \Big]\\& 
+ \epsilon_0 \label{e0} =: t^2e_0^{0} + \epsilon_0
}
where we have 
\[
\epsilon_0\in  \gamma_1\frac{\lambda_{0,0}^{\frac12}}{\mu_{0,0}(t)^{k_0}} IS^0(R_{0,0}^{-1}) +   \gamma_2\frac{\lambda_{0,0}^{\frac12}}{\mu_{0,0}(t)^{k_0}}\log t IS^0(R_{0,0}^{-1})
\]
Further, importantly the constants $\omega_{1,2}$ do not depend on $\gamma_{1,2}$. We shall then treat $\epsilon_0$ as a lower order error which can be neglected in the first $k_0$ stages of the iteration process. 
\\

{\bf{Step 1}}: Introduce the operator 
\EQ{
L_0 &:= \p_{R_{0,0}}^2 + \frac{2}{R_{0,0}}\p_{R_{0,0}} +5 W^4(R_{0,0})
}
Then we solve 
\EQ{
\mu_{0,0}^2(t) L_0 v_1 = t^2 e_0^0,\quad v_1(0)=v_1'(0)=0
}
Introducing the conjugated operator $\tilde{L}_0 : = R_{0,0}L_0 R_{0,0}^{-1}$, which has fundamental system
\EQ{\label{tilde fundsys}
\tilde \fy_1(R_{0,0}) &:=\f{R_{0,0}(1-R_{0,0}^2/3)}{(1+R_{0,0}^2/3)^{\f32}}\\
\tilde \fy_2(R_{0,0}) &: = \f{1-2R_{0,0}^2 +R_{0,0}^4/9}{(1+R_{0,0}^2/3)^{\f32}}, 
}
we find the following expression for $v_1$: 
\begin{align*}
\mu_{0,0}^2(t)v_1(t, R_{0,0}) &= R_{0,0}^{-1} \Big( \tilde\fy_1(R_{0,0}) \int_0^{R_{0,0}} \tilde\fy_2(R_{0,0}')R_{0,0}' t^2 e_0^0(R_{0,0}')\, dR_{0,0}'\\&\hspace{3cm} -  \tilde\fy_2(R_{0,0}) \int_0^{R_{0,0}} \tilde\fy_1(R_{0,0}')R' t^2 e_0^0(R_{0,0}')\, dR_{0,0}'\Big)
\end{align*}
Then using \eqref{e0}, we infer 
\EQ{\label{eq:v1}
v_1(t,r) = \la_{0,0}^{\f12}(t)  \mu_{0,0}^{-2}(t)  (\om_1 f_1(R_{0,0}) + \om_2 f_2(R_{0,0})) =: \la_{0,0}^{\f12}(t) \mu_{0,0}^{-2}(t)  f(R_{0,0})
}
where further 
\EQ{\label{fasymp infty}
f_j(R_{0,0}) &= R_{0,0}(b_{1j} + b_{2j} R_{0,0}^{-1}+ R_{0,0}^{-2}\log R_{0,0}\;\fy_{1j}(R_{0,0}^{-2}) +  R_{0,0}^{-2} \fy_{2j}(R_{0,0}^{-1})  ) \\
&=: R_{0,0}(F_j(\rho) + \rho^{2} G_{j}(\rho^2)\log \rho)
}
where $\fy_{1j}, \fy_{2j}$ and $F_j, G_j$ are analytic around zero, with $\rho:=R_{0,0}^{-1}$. Moreover, the coefficients of these analytic functions do not depend on $\gamma_{1,2}$. 
\\

{\bf{Step 2}}: Here we analyse the error $e_1$ generated by the approximate solution $u_1 = u_0 + v_1$, which equals 
\EQ{\label{e1def}
e_1 =  &\p_t^2 v_1 - 10 u_0^3 v_1^2 -10 u_0^2 v_1^3 -5 u_0 v_1^4 - v_1^5\\
&\hspace{3cm} + 5\lambda_{0,0}^2(t)[\frac{\lambda^2(t)}{\lambda_{0,0}^2(t)}W^4(R) - W^4(R_{0,0})]v_1 + \epsilon_0.
}
Using the \eqref{eq:v1}, we can write $t^2 e_1$ as a sum as follows 
\[
t^2 e_1 = \sum_{j=1}^3 A_j + \epsilon_0,
\]
where up to sign, the terms are given by 
\begin{align*}
&A_1 = \la_{0,0}^{\f12}(t) \mu_{0,0}^{-2}(t)\sum_{k=2}^5\left(\begin{array}{c}5\\ k\end{array}\right)  \mu_{0,0}^{4 - 2k}(t)W^{5-k}(R_{0,0})f^k((R_{0,0}),\\
&A_2 =  \lambda_{0,0}^{\frac12}(t)\left(\left(t\p_t + t\la'_{0,0}(t)\la_{0,0}^{-1}(t)\cD\right)^2 - \left(t\p_t + t\la_{0,0}'(t)\la_{0,0}^{-1}(t)\cD\right) \right)\big(\mu_{0,0}^{-2}(t)  f(R_{0,0})\big),\\
&A_3 = \la_{0,0}^{\f12}(t) \sum_{k=1}^4 \left(\begin{array}{c}5\\ k\end{array}\right)  \mu_{0,0}^{2 - 2k}(t)f^k(R_{0,0})[W^{5-k}(R)\frac{\lambda^{\frac{5-k}{2}}(t)}{\lambda_{0,0}^{\frac{5-k}{2}}(t)} - W^{5-k}(R_{0,0})],
\end{align*}
with $\cD = \frac12 + R_{0,0}\partial_{R_{0,0}}$. Also, $\epsilon_0$ is as in Step 0. Then we can write 
\begin{align*}
A_2 &=  \la_{0,0}^{\f12}(t) \mu_{0,0}^{-2}(t)\big[(2\nu - (1+\nu)\cD)^2 - (2\nu - (1+\nu)\cD)\big]  f(R_{0,0})\\
& =: \la_{0,0}^{\f12}(t) \mu_{0,0}^{-2}(t)g(R_{0,0}),
\end{align*}
where the last term on the right admits an expansion like for $v_1$ in \eqref{eq:v1}, with coefficients that are independent of $\gamma_{1,2}$. 
\\
On the other hand, the term $A_3$ is dependent on $\gamma_{1,2}$, and can in fact be placed in the space 
\[
 \gamma_1\frac{\lambda_{0,0}^{\frac12}}{\mu_{0,0}(t)^{k_0}} IS^0(R_{0,0}^{-1}) +   \gamma_2\frac{\lambda_{0,0}^{\frac12}}{\mu_{0,0}(t)^{k_0}}\log t IS^0(R_{0,0}^{-1})
 \]
 We shall deal with it when we define $v_{smooth, a}$. At any rate, the error $e_1$ satisfies \eqref{e2k-1} for $k = 1$. 
\\ 

{\bf{Step 3}}: Choice of second correction $v_2$. It is in this step where the shock along the light cone, as evidenced by the expansion \eqref{eq:expansion2}, as well as the definition of $\mathcal{Q}$, is introduced into $u_{\text{approx}}^{(\gamma_{1,2})}$ (whence also into $u_{\text{approx}}^{(0,0)} = u_{\nu}$, the solutions being described in Theorem~\ref{thm:KST}). The key in this step shall be to ensure that the singular part of $v_2$ will be independent of $\gamma_{1,2}$. This we can achieve since by our preceding construction the principal part of the error $e_1$ is independent of $\gamma_{1,2}$. Write 
\[
e_1 = e^0_1 + t^{-2}\epsilon_1,\,\epsilon_1: =  A_3 + \epsilon_0.
\]
Then as in \cite{KS1}, equation (2.32),  we infer the leading behaviour of the term $e_1^0$ (where we change the notation with respect to \cite{KS1}), as follows:
\EQ{ \label{e10*}
t^2e_1^{00}(t,r)  &:=  \la^{\f12}_{0,0}(t) \mu_{0,0}^{-1}(t) ( c_1 a + c_2 b)
}
where we have $a = \frac{r}{t}$, $b= b(t) = \frac{1}{\mu_{0,0}(t)}$, and as remarked before the coefficients $c_j$ do not depend on $\gamma_{1,2}$. Also, recall
\[
\mu_{0,0}(t) = (\lambda_{0,0}(t)\cdot t). 
\]
The second correction will then be obtained by neglecting the effect of the potential term $5W^4(R)$, and setting
\EQ{\label{v_2trueeqn}
t^2 \big(v_{2,tt}-v_{2,rr}-\f2r v_{2,r}   \big) = -t^2 e_1^{00}  
}
To solve this we make the ansatz 
\EQ{\label{v2}
v_2(t,r) = \la_{0,0}(t)^{\f12}\big( \mu^{-1}_{0,0}(t) q_1(a) + \mu^{-2}_{0,0}(t) q_2(a) \big)
}
In fact, proceeding exactly as in \cite{KS1},  section 2.5, we then infer the equations
\EQ{\label{q1q2}
L_{\f{\nu-1}{2}} \; q_1 = c_1 a,\quad L_{\f{3\nu-1}{2}}\; q_2 = c_2,
}
where we set 
\EQ{\label{Lbg}
L_{\beta}:= (1-a^2)\p_a^2 + (2(\beta-1)a+2a^{-1})\p_a - \beta^2 + \beta.
}
In fact, our $\lambda_{0,0}, \mu_{0,0}$ are exactly the $\lambda, \mu$ in \cite{KS1}. To uniquely determine $q_{1,2}$, we impose the vanishing conditions
\[
q_j(0) = q_j'(0)= 0,\,j = 1,2. 
\]
As in \cite{KS1}, equation (2.44) one can then write (using $ a= \frac{R_{0,0}}{\mu_{0,0}(t)}$ where $R_{0,0}: = r\lambda_{0,0}(t)$)
\[
v_2 = \frac{ \la_{0,0}(t)^{\f12}}{\mu_{0,0}^2(t)}(R_{0,0}\tilde{q}_1(a) + q_2(a)),
\]
where now $\tilde{q}_1, q_2$ both have even power expansions around $a = 0$. 
In order to ensure the necessary parity of exponents in the power series expansions around $R_{0,0} = 0$ imposed by the definition of $\mathcal{Q}$, we sacrifice some accuracy in the approximation, relabel the preceding expression $v_2^0(t,r)$ (as in \cite{KS1}), and then use for the true correction $v_2$ the formula
\[
v_2 = \frac{ \la_{0,0}(t)^{\f12}}{\mu_{0,0}^2(t)}(R_{0,0}^2\langle R_{0,0}\rangle^{-1}\tilde{q}_1(a) + q_2(a)),\,\langle R_{0,0}\rangle  = \sqrt{R_{0,0}^2 + 1}. 
\]

Again by construction $\tilde{q}_1, q_2$ and thence $v_2$ do not depend on $\gamma_{1,2}$. 
\\

{\bf{Step 4}}: Here we analyse the error generated by the approximate solution $u_2 = u_0 + v_1 + v_2$, which is given by the expression 
\begin{align*}
 e_2 &= e_1 - e_1^{00} - 5u_1^4 v_2 - 10 u_1^3 v_2^2 - 10 u_1^2 v_2^3 - 5 u_1 v_2^4 - v_2^5\\& +  (  \p_{tt} -\p_{rr} - \f{2}{r}\p_r )(v_2-v_2^0) 
\end{align*}
Then according to the preceding we have 
\begin{align*}
& t^2(e_1 - e_1^{00}) - \epsilon_0\\&\in O(R_{0,0}^{-1} \la_{0,0}(t)^{\f12}\mu_{0,0}^{-2}(t))  + \gamma_1\frac{\lambda^{\frac12}(t)}{\mu^{k_0+2}(t)}IS^0(R_{0,0}) +  \gamma_2\frac{\lambda^{\frac12}}{\mu(t)^{k_0+2}}\log t IS^0(R_{0,0}), 
 \end{align*}
 where the first term $O(R_{0,0}^{-1} \la_{0,0}(t)^{\f12}\mu_{0,0}^{-2}(t))$ is independent of $\gamma_{1,2}$. The sum of the last two terms on the right will then be deferred until the last stage, when we define $v_{smooth,a}$. 
 Next, consider
 \begin{align*}
&t^2\big[-5u_1^4 v_2 - 10 u_1^3 v_2^2 - 10 u_1^2 v_2^3 - 5 u_1 v_2^4 - v_2^5 +  (  \p_{tt} -\p_{rr} - \f{2}{r}\p_r )(v_2-v_2^0)\big]\\& 
\end{align*}
Here the interaction terms $u_1^{5-j}v_2^j$, $j\leq 4$, are only of the smoothness implied by $\mathcal{Q}$, but do depend on $\gamma_{1,2}$ on account of $u_1 = u_0 + v_1$ and the $\gamma$-dependence of $u_0$. However, writing 
\[
u_1 = [u_0 - \lambda_{0,0}^{\frac12}W(\lambda_{0,0}r)] +[v_1+ \lambda_{0,0}^{\frac12}W(\lambda_{0,0}r)], 
\]
and expanding out $u_1^{5-j}$, we can place any term of the form 
\[
t^2 [u_0 - \lambda_{0,0}^{\frac12}W(\lambda_{0,0}r)]^{l_1}[v_1+\lambda_{0,0}^{\frac12}W(\lambda_{0,0}r)]^{l_2}v_2^{l_3},\,\sum l_j = 5,
\]
and with $l_1\geq 1,l_3\geq 1$ into 
\begin{align*}
   &\gamma_1  \frac{\lambda_{0,0}^{\frac12}}{\mu_{0,0}(t)^{k_0+2}} \big[IS^0(R_{0,0}^{-1}, \mathcal{Q}) + b^2IS^0(R_{0,0}, \mathcal{Q})\big]\\&
  + \gamma_2\log t \frac{\lambda_{0,0}^{\frac12}}{\mu_{0,0}(t)^{k_0+2}} \big[ IS^0(R_{0,0}^{-1}, \mathcal{Q}) +  b^2IS^0(R_{0,0}, \mathcal{Q})\big],
\end{align*}
and so this can be placed into $t^2 e_{\text{prelim}}$. Finally, the preceding also implies \eqref{e2k} for $k = 1$. 
\\

{\bf{Step 5}}: The inductive step. Here we again follow \cite{KS1}, section 2.7,  closely, but need to carefully keep track of various parts of $e_k$ First consider the case of even indices, i. e. assume $e_{2k-2}$, $2\leq k\leq k_*$, satisfies \eqref{e2k} with $k$ replaced by $k-1$, and more precisely, that we can decompose
\begin{align}
e_{2k-2} = e_{2k-2}^1 + e_{2k-2}^2 + e_{2k-2}^3\label{e2k-2tripledecomp}, 
\end{align}
where we have 
\begin{align*}
&t^2e_{2k-2}^1\in  \frac{\lambda_{0,0}^{\frac12}}{\mu_{0,0}(t)^{2k-2}} \big[ IS^0(R_{0,0}^{-1} \, (\log R_{0,0})^{q_{k-1}} ,\mathcal{Q})
   + b^2 IS^0(R_{0,0} \, (\log R_{0,0})^{q_{k-1}} ,\mathcal{Q}')    \big],\\
   &t^2e_{2k-2}^2\in  \gamma_1\frac{\lambda_{0,0}^{\frac12}}{\mu_{0,0}(t)^{k_0}} IS^0(R_{0,0}^{-1}) +   \gamma_2\frac{\lambda_{0,0}^{\frac12}}{\mu_{0,0}(t)^{k_0}}\log t IS^0(R_{0,0}^{-1}),
\end{align*}
the term $e_{2k-2}^1$ being independent of $\gamma_{1,2}$, while for the third term we have
\begin{align*}
t^2e_{2k-2}^3&\in \gamma_1  \frac{\lambda_{0,0}^{\frac12}}{\mu_{0,0}(t)^{k_0+2}} \big[IS^0(R_{0,0}^{-1}, \mathcal{Q}) + b^2IS^0(R_{0,0}, \mathcal{Q})\big]\\&
  + \gamma_2\log t \frac{\lambda_{0,0}^{\frac12}}{\mu_{0,0}(t)^{k_0+2}} \big[ IS^0(R_{0,0}^{-1}, \mathcal{Q}) +  b^2IS^0(R_{0,0}, \mathcal{Q})\big].
\end{align*}
We have verified such a structure for the case $k = 2$ in the preceding step. Then we introduce the correction $v_{2k-1}$ in order to improve the error $e_{2k-1}^1$, exactly mirroring Step 1 in section 2.7 of \cite{KS1}. We completely forget about $e_{2k-2}^3$ as it can be moved into the final error $e_{\text{prelim}}$, while we shall deal with the intermediate term $e_{2k-2}^2$ when introducing $v_{smooth,a}$. Returning to $v_{2k-1}$, and proceeding just as in Step 1, we see that $v_{2k-1}$ will satisfy \eqref{v2k-1}, and moreover be independent of $\gamma_{1,2}$. The error $e_{2k-1}$ generated by the approximation $u_0 + \sum_{j=1}^{2k-1}v_j$ will be mostly independent of $\gamma_{1,2}$, and satisfy \eqref{e2k-1}, except for the cross interaction terms of $v_{2k-1}$ and $u_0$, of the form $u_0^{5-j} v_{2k-1}^j$, $1\leq j\leq 4$. However, splitting 
\[
u_0 = [u_0 - \lambda_{0,0}^{\frac12}W(\lambda_{0,0}(t)r)] + [\lambda_{0,0}^{\frac12}W(\lambda_{0,0}(t)r)], 
\]
we may replace $u_0$ by $u_0 - \lambda_{0,0}^{\frac12}W(\lambda_{0,0}(t)r)$, and then the corresponding cross interactions, multiplied by $t^2$, can again be seen to be in 
\begin{align*}
&\gamma_1  \frac{\lambda_{0,0}^{\frac12}}{\mu_{0,0}(t)^{k_0+2}} \big[IS^0(R_{0,0}^{-1}, \mathcal{Q}) + b^2IS^0(R_{0,0}, \mathcal{Q})\big]\\&
  + \gamma_2\log t \frac{\lambda_{0,0}^{\frac12}}{\mu_{0,0}(t)^{k_0+2}} \big[ IS^0(R_{0,0}^{-1}, \mathcal{Q}) +  b^2IS^0(R_{0,0}, \mathcal{Q})\big],
\end{align*}
whence these error terms may be placed into $e_{\text{prelim}}$ and discarded. 
\\
The case of odd indices, i. e. departing from $e_{2k-1}$, $k\leq k_*$, is handled just the same. 
\\
Repeating this procedure leads to the $v_j$, $1\leq j\leq 2k_*-1$. Moreover, each of the errors generated satisfies a decomposition analogous to \eqref{e2k-2tripledecomp}, replacing \eqref{e2k} by \eqref{e2k-1} for odd indices. 
\\

{\bf{Step 6}}: Choice of $v_{smooth, a}$, $a = 1,2$. Here we depart from the approximation $u_{2k_* - 1} = u_0 + \sum_{j=1}^{2k_* - 1}v_j$, which generates an error $e_{2k_* - 1}$ satisfying \eqref{e2k-1} for $k = k_*$, as well as a decomposition 
\begin{equation}\label{e2k*decomp}
e_{2k_* - 1} = \sum_{j=1}^3 e_{2k_* - 1}^j
\end{equation}
analogous to \eqref{e2k-2tripledecomp}. Importantly, the first error 
\[
t^2e_{2k_* - 1}^1\in \frac{\lambda_{0,0}^{\frac12}}{\mu_{0,0}(t)^{2k_*}} IS^0(R_{0,0}\, (\log R_{0,0})^{p_{k_*}}, \mathcal{Q}')
\]
is independent of $\gamma_{1,2}$, and the last error $e_{2k_* - 1}^3$ may be placed into $e_{\text{prelim}}$, and so it remains to deal with the middle error which for technical reasons is still too large. Recall that the middle error satisfies
\[
t^2e_{2k_* - 1}^2\in \gamma_1\frac{\lambda_{0,0}^{\frac12}}{\mu_{0,0}(t)^{k_0}} IS^0(R_{0,0}^{-1}) +   \gamma_2\frac{\lambda_{0,0}^{\frac12}}{\mu_{0,0}(t)^{k_0}}\log t IS^0(R_{0,0}^{-1}),
\]
and in particular is $C^\infty$-smooth. 
Then set 
\[
\mu_{0,0}^2(t)L_0v_{smooth,1} = t^2e_{2k_* - 1}^2, 
\]
leading to 
\[
v_{smooth,1}\in  \gamma_1\frac{\lambda_{0,0}^{\frac12}}{\mu_{0,0}(t)^{k_0+2}} IS^2(R_{0,0}) +   \gamma_2\frac{\lambda_{0,0}^{\frac12}}{\mu_{0,0}(t)^{k_0+2}}\log t IS^2(R_{0,0})
\]
Then all errors generated by $v_{smooth,1}$ by interaction with the bulk part $u_{2k_*-1}$ can be placed into $e_{\text{prelim}}$. On the other hand, the error $t^2\partial_t^2v_{smooth,1}$ is of the same form as $v_{smooth,1}$. 
We next construct $v_{smooth,2}$, proceeding in analogy to Step 3, to improve the error generated by $\partial_t^2v_{smooth,1}$. The key here is that on the account of the rapid temporal decay of this term, the method of \cite{KS1} applied to it results in a term of sufficient smoothness, to be acceptable for a correction depending on $\gamma_{1,2}$. 
Specifically, we write the leading order term of $t^2\partial_t^2v_{smooth,1}$ in the form 
\[
(c_1+c_3\log t)\frac{\lambda_{0,0}^{\frac12}}{\mu_{0,0}(t)^{k_0+2}}R_{0,0} + (c_2 + c_4\log t)\frac{\lambda_{0,0}^{\frac12}}{\mu_{0,0}(t)^{k_0+2}}, 
\]
and then set (where the coefficients $c_{1,2}$ depend on $\gamma_{1,2}$)
\begin{align*}
&t^2 \big(\partial_t^2v_{smooth,2}-\partial_r^2v_{smooth,2}-\f2r \p_r v_{smooth,2}   \big)\\& = (c_1+c_3\log t)\frac{\lambda_{0,0}^{\frac12}}{\mu_{0,0}(t)^{k_0+2}}R_{0,0} + (c_2+ c_4\log t)\frac{\lambda_{0,0}^{\frac12}}{\mu_{0,0}(t)^{k_0+2}}. 
\end{align*}
Making the correct ansatz as in \cite{KS1} this is solved by 
\[
v_{smooth, 2}\in  \frac{\lambda_{0,0}^{\frac12}}{\mu_{0,0}(t)^{k_0+4}} IS^2(R_{0,0}^3 ,\mathcal{Q}_{smooth}) + \log t\frac{\lambda_{0,0}^{\frac12}}{\mu_{0,0}(t)^{k_0+4}} IS^2(R_{0,0}^3 ,\mathcal{Q}_{smooth}).
\]
The effect of this correction is that we replace the middle term in \eqref{e2k*decomp} by one in $e_{\text{prelim}}$, i. e. our final approximate solution 
\[
u_{\text{prelim}}: = u_0+\sum_{j=1}^{2k_*-1}v_j + \sum_{a = 1,2}v_{smooth,a}
\]
generates an error $e_{\text{prelim}}$ as claimed in the lemma. 

\end{proof}

In order to complete the proof of the Theorem~\ref{thm:main approximate}, we need to improve the approximate solution obtained in the preceding lemma a bit in order to replace the generated error $e_{\text{prelim}}$ by one which is smoother. More precisely, we need to get rid of the rough part of the error $\tilde{e}_{\text{prelim}}$. For this, we replace $u_{\text{prelim}}$ by 
\[
u_{approx}: = u_{\text{prelim}} + v,
\]
where $v$ solves the equation 
\[
\Box v + 5\tilde{u}_{\text{prelim}}^4v + \sum_{2\leq j\leq 5}{{5}\choose{j}}v^j\tilde{u}_{\text{prelim}}^{5-j} = -\tilde{e}_{\text{prelim}},
\]
where 
\[
\tilde{u}_{\text{prelim}} = u_{\text{prelim}} - v_{smooth} + \lambda_{0,0}^{\frac12}W(\lambda_{0,0}(t)r) - \lambda^{\frac12}W(\lambda(t)r),\, v_{smooth} = \sum_{a = 1}^2v_{smooth, a}
\]
is the $\gamma$-independent part of $ u_{\text{prelim}}$. 
Also, we shall impose vanishing of $v$ at $t = 0$. Then it is clear that $v$ will not depend on $\gamma_{1,2}$. The fact that such a $v$ can be computed with the required smoothness and bounds, provided $N$ is chosen large enough,  follows exactly as in \cite{KST}, see the discussion there after equation (3.1). Also, we have for any $t\in (0,t_0]$
\[
\big\|\nabla_{t,x}v(t)\big\|_{H^{\frac{\nu}{2}-}}\lesssim t^{N-3}
\]
Then we arrive at the error 
\begin{align*}
&\Box u_{approx} + u_{approx}^5\\
& = \Box u_{\text{prelim}} + u_{\text{prelim}}^5 + \sum_{2\leq j\leq 5}{{5}\choose{j}}v^ju_{\text{prelim}}^{5-j}\\
& + \Box v + 5\tilde{u}_{\text{prelim}}^4v\\
& + 5(-\tilde{u}_{\text{prelim}}^4 + u_{\text{prelim}}^4)v
\end{align*}
It follows that 
\begin{align}\label{eq:eapprox}
e_{approx} &= e_{\text{prelim}} - \tilde{e}_{\text{prelim}} + \sum_{2\leq j\leq 5}{{5}\choose{j}}v^j[u_{\text{prelim}}^{5-j} - \tilde{u}_{\text{prelim}}^{5-j} ]\nonumber\\&
+  5(-\tilde{u}_{\text{prelim}}^4 + u_{\text{prelim}}^4)v
\end{align}
This remaining error is easily seen to satisfy the claimed properties of the theorem. 
\end{proof}

\section{Modulation theory; determination of the parameters $\gamma_{1,2}$.}

\subsection{Re-scalings and the distorted Fourier transform}

The discussion following \eqref{eq:changeofframe} shows that we intend to pass to a slightly altered coordinate system, depending on the parameters $\gamma_{1,2}$, and given by 
\[
\tau_{\gamma_1,\gamma_2}, R_{\gamma_1,\gamma_2}
\]
differing from the old one which corresponded to $\tau_{0,0}, R_{0,0}$ (and which served as the basis for the discussion following \eqref{eq:epseqn}). We then have to reinterpret functions given in terms of $R_{0,0}$ as functions in terms of $R_{\gamma_1,\gamma_2}$, and understand the effect of such a change of scale on the distorted Fourier transform. Infinitesimally, this is explained in terms of Theorem~\ref{thm:transferenceop}, and we state here a simple variation on this theme: 

\begin{lem}\label{lem:changeofscale} Assume $\tilde{\epsilon}$ has the Fourier representation given by 
\[\tilde{\epsilon}(R)=\int_{0}^{\infty}x(\xi)\phi(R,\xi)\rho(\xi)\, d\xi + x_{d}\phi_{d}(R).	
\]
Then we have the formula 
\[
\mathcal{F}\big(\tilde{\epsilon}(e^{-\kappa}R)\big)(\xi) = x(e^{2\kappa}\xi) + \kappa\cdot\tilde{\mathcal{K}}_{\kappa}x + O_{\|\cdot\|_{\tilde{S}_1}}(\kappa\big|x_d\big|)
\]
where $\tilde{\mathcal{K}}_{\kappa}$ can be written as 
\[
\big(\tilde{\mathcal{K}}_{\kappa}x\big)(\xi) =  h_{\kappa}(\xi) x(e^{2\kappa}\xi) + \int_0^\infty \frac{F_{\kappa}(\xi, \eta)\rho(\eta)}{\xi - \eta}x(\eta)\,d\eta
\]
where $F_{\kappa}$ satisfies the same bounds as the function $F(\cdot, \cdot)$ in Theorem~\ref{thm:transferenceop}, and the function $h_{\kappa}$ is of class $C^\infty(0,\infty)$ and uniformly bounded (both in $\xi$ as well as $\kappa$), and satisfies symbol type bounds with respect to $\xi$. 

In particular, we have 
\[
\big\|\mathcal{F}\big(\tilde{\epsilon}(e^{\kappa}R)\big)\big\|_{\tilde{S}_1}\lesssim_{\tau_0,\kappa} (\big\|x\big\|_{\tilde{S}_1} + \big|x_d\big|). 
\] 
and more precisely, we have 
\[
\big\|\mathcal{F}\big(\tilde{\epsilon}(e^{\kappa}R)\big) - \big(\mathcal{F}(\tilde{\epsilon})\big)(e^{2\kappa}\xi)\big\|_{\tilde{S}_1}\lesssim_{\tau_0} \kappa (\big\|x\big\|_{\tilde{S}_1} + \big|x_d\big|). 
\]
as well as 
\[
\big\|\mathcal{F}\big(\tilde{\epsilon}(e^{\kappa}R)\big)\big\|_{\tilde{S}_1}\lesssim (1+ \tau_0\kappa)
\big\|\mathcal{F}\big(\tilde{\epsilon}(R)\big)\big\|_{\tilde{S}_1} + \kappa\big|x_d\big|. 
\] 

\end{lem}
\begin{proof} This is entirely analogous to the proof of Theorem 5.1 in \cite{KST}; in effect the latter deals with the 'infinitesimal version' of the current situation. Consider the expression 
\[
(\Xi_{\kappa} x)(\eta): = \langle \int_0^\infty x(\xi)\phi(e^{-\kappa}R, e^{2\kappa}\xi)\rho(\xi)\,d\xi,\,\phi(R, \eta)\rangle,
\]
where $x\in C_0^\infty(0,\infty)$. Under the latter restriction the integral converges absolutely. Then proceeding as in \cite{KST}, see in particular Lemma 4.6 and the proof of Theorem 5.1, we get 
\[
(\Xi_{\kappa} x)(\xi) = \Re\big[\frac{a_+(e^{2\kappa}\xi)}{a_+(\xi)}\big]x(\xi) + \int_0^\infty f_{\kappa}(\xi, \eta)x(\eta)\,d\eta,
\]
where $a_+$ is the function occurring in Prop.~\ref{prop:Fourier}.  
Here in order to determine the kernel $f_{\kappa}$ of the 'off-diagonal' operator at the end, we use 
\begin{align*}
&(\eta - \xi) f_{\kappa}(\xi, \eta)\\& = \langle \int_0^\infty x(\xi)5[e^{-2\kappa}W^4(e^{-\kappa}R) - W^4(R)]\phi(e^{-\kappa}R, e^{2\kappa}\xi)\rho(\xi)\,d\xi,\,\phi(R, \eta)\rangle
\end{align*}
Then by following the argument of \cite{KST}, proof of Theorem 5.1,  one infers that 
\[
f_{\kappa}(\xi, \eta) = \kappa\cdot\frac{\rho(\eta)F_{\kappa}(\xi, \eta)}{\xi - \eta}, 
\]
with $F_{\kappa}$ having the same asymptotic and vanishing properties as the kernel $F(\xi, \eta)$ in Theorem~\ref{thm:transferenceop}, uniformly in $\kappa\in [0,1]$, say. It remains to translate the properties of $\Xi_{\kappa}$ to those of the re-scaling operator. 
Let $\Psi$ be the operator which satisfies 
\[
\mathcal{F}\big(\Psi(\tilde{\epsilon})\big)(\xi) = e^{-2\kappa}\frac{\rho(\frac{\xi}{e^{2\kappa}})}{\rho(\xi)}x(\frac{\xi}{e^{2\kappa}})
\]
and leaves the discrete spectral part invariant, while $S_{e^{-\kappa}}(\tilde{\epsilon})(R) = \tilde{\epsilon}(\frac{R}{e^{\kappa}})$ is the scaling operator. Then we have 
\[
(\Xi_{\kappa} x)(\xi) = \mathcal{F}\big(S_{e^{-\kappa}}\Psi(\tilde{\epsilon})\big)(\xi). + O_{\|\cdot\|_{\tilde{S}_1}}(\kappa\big|x_d\big|).
\]
We conclude that 
\[
\mathcal{F}\big(S_{e^{-\kappa}}\tilde{\epsilon}\big)(\xi) = \Xi_{\kappa}\big(\mathcal{F}\big(\Psi^{-1}(\tilde{\epsilon})\big)\big) + O_{\|\cdot\|_{\tilde{S}_1}}(\kappa\big|x_d\big|).
\]
It follows that we can write 
\begin{align*}
\mathcal{F}\big(S_{e^{-\kappa}}\tilde{\epsilon}\big)(\xi) = &x(e^{2\kappa}\xi) + \big[e^{2\kappa}\Re\big[\frac{a_+(e^{2\kappa}\xi)}{a_+(\xi)}\big]\cdot\frac{\rho(e^{2\kappa}\xi)}{\rho(\xi)}-1\big]x(e^{2\kappa}\xi)\\
& + \int_0^\infty \widetilde{f_{\kappa}}(\xi, \eta)x(\eta)\,d\eta +  O_{\|\cdot\|_{\tilde{S}_1}}(\kappa\big|x_d\big|),
\end{align*}
where we put 
\[
\widetilde{f_{\kappa}}(\xi, \eta): = f_{\kappa}(\xi, \frac{\eta}{e^{2\kappa}})\cdot\frac{\rho(\eta)}{\rho(\frac{\eta}{e^{2\kappa}})}. 
\]
This implies the claims of the lemma. 

\end{proof}

\subsection{The effect of scaling the bulk part}

Here we investigate how changing the bulk part from $\lambda_{0,0}^{\frac12}W(\lambda_{0,0}r)$ to $\lambda^{\frac12}W(\lambda r)$ affects the distorted Fourier transform of the new perturbation term. Specifically, recall \eqref{eq:newdataperutbation}, which defines a new data pair $(\bar{\epsilon}_0, \bar{\epsilon}_1)$, which in turn uniquely define a new quadruple of Fourier components $(x_0^{(\gamma_1,\gamma_2)}, x_{0d}^{(\gamma_1,\gamma_2)}, x_1^{(\gamma_1,\gamma_2)}, x_{1d}^{(\gamma_1,\gamma_2)})$, via \eqref{eq:x0barepsilon0}, \eqref{eq:datatransference3}, \eqref{eq:datatransference4}. We can then derive the analogue of \eqref{eq:transport}, and try to repeat the iterative process in \cite{CondBlow}, but for this we shall have to ensure the two key vanishing conditions \eqref{eq:vanishing2}, as well as the condition \eqref{eq:cond1}. That this is indeed possible is the content of the following 

\begin{prop}\label{prop:choiceofgamma} Given a fixed $\nu\in (0, \nu_0]$, $t_0\in (0, 1]$, there is a $\delta_1 = \delta_1(\nu, t_0)>0$ small enough such that the following holds. 
Given a triple of data 
\[
(x_0(\xi), x_1(\xi), x_{0d})\in \tilde{S}\times \R
\]
and with  
\[
\big\|(x_0, x_1)\big\|_{\tilde{S}} + \big|x_{0d}\big|<\delta_1,
\]
there is a unique pair $\gamma_{1,2}$ with $\big|\gamma_1\big| + \big|\gamma_2\big|\lesssim_{\nu, t_0}\big\|(x_0, x_1)\big\|_{\tilde{S}}$ and a unique parameter $x_{1d}$ satisfying $\big|x_{1d}\big|\lesssim_\nu\big|x_{0d}\big|$  such that 
determining $(\epsilon_0, \epsilon_1)$ via \eqref{eq:x0epsilon0}, \eqref{eq:datatransference1}, \eqref{eq:datatransference2}, and from there $(\bar{\epsilon}_0, \bar{\epsilon}_1)$ via \eqref{eq:newdataperutbation} which in turn defines the quadruple of Fourier data $(x_0^{(\gamma_1,\gamma_2)}, x_{0d}^{(\gamma_1,\gamma_2)}, x_1^{(\gamma_1,\gamma_2)}, x_{1d}^{(\gamma_1,\gamma_2)})$, we have 
\begin{align*}
&A(\gamma_1,\gamma_2) : = \int_0^\infty\frac{\rho^{\frac{1}{2}}(\xi)x_1^{(\gamma_1,\gamma_2)} (\xi)}{\xi^{\frac{3}{4}}}\sin[\lambda_{\gamma_1,\gamma_2} (\tau_0)\xi^{\frac12}\int_{\tau_0}^\infty\lambda_{\gamma_1,\gamma_2} ^{-1}(s)\,ds]\,d\xi = 0\\
&B(\gamma_1,\gamma_2): = \int_0^\infty\frac{\rho^{\frac{1}{2}}(\xi)x_0^{(\gamma_1,\gamma_2)} (\xi)}{\xi^{\frac{1}{4}}}\cos[\lambda_{\gamma_1,\gamma_2} (\tau_0)\xi^{\frac12}\int_{\tau_0}^\infty\lambda_{\gamma_1,\gamma_2} ^{-1}(s)\,ds]\,d\xi = 0,
\end{align*}
and the discrete spectral part $(x_{0d}^{(\gamma_1,\gamma_2)}, x_{1d}^{(\gamma_1,\gamma_2)})$ satisfies the vanishing property of Lemma~\ref{lem:linhom}, \eqref{eq:cond1}, with respect to the scaling law $\lambda = \lambda_{\gamma_1,\gamma_2}$. 
We have the precise bound 
\begin{equation}\label{eq:gammabounds}
\big|\gamma_1 \lambda_{0,0}^{\frac12}t_0^{k_0\nu}\big| + \big|\gamma_2 \lambda_{0,0}^{\frac12}\log t_0 t_0^{k_0\nu}\big|\lesssim \tau_0\log\tau_0(\big\|(x_0,x_1)\big\|_{\tilde{S}} + \big|x_{0d}\big|). 
\end{equation}
Finally, we have the bound 
\begin{align*}
&\big\|x_0^{(\gamma_1,\gamma_2)} - \frac{\lambda_{0,0}}{\lambda}S_{\frac{\lambda_{0,0}^2}{\lambda^2}}x_0\big\|_{\tilde{S}_1} + \big\|x_1^{(\gamma_1,\gamma_2)} -  \frac{\lambda_{0,0}}{\lambda}S_{\frac{\lambda_{0,0}^2}{\lambda^2}}x_1\big\|_{\tilde{S}_2}\\&\hspace{5cm}\lesssim \log\tau_0\cdot\tau_0^{0+}\cdot(\big\|(x_0,x_1)\big\|_{\tilde{S}} + \big|x_{0d}\big|).
\end{align*}
where $S_{\frac{\lambda_{0,0}^2}{\lambda^2}}x_i(\xi)=x_{i}(\frac{\lambda_{0,0}^2}{\lambda^2}\xi)$ is the scaling operator.

\end{prop}  

\begin{proof} The strategy shall be to first fix the discrete spectral part to $(x_{0d}, x_{1d})$ while choosing $\gamma_{1,2}$, and at the end finalising the choice of $x_{1d}$ to satisfy the required co-dimension one condition. 

Observe that from our definition and the structure of $u_{approx}^{(\gamma_1,\gamma_2)}$, in particular Lem-\\ma~\ref{lem:corrections}, and the end of the proof of Theorem~\ref{thm:main approximate}, we can write 
\begin{equation}\label{eq:barepsilon0} 
\bar{\epsilon}_0 = \chi_{r\leq Ct_0}\big[\lambda_{0,0}^{\frac12}W(\lambda_{0,0}r) - \lambda_{\gamma_1,\gamma_2}^{\frac12}W(\lambda_{\gamma_1,\gamma_2}r) - v_{smooth}\big] + \epsilon_0,
\end{equation}
as well as 
\begin{equation}\label{eq:barepsilon1}
\bar{\epsilon}_1 = \chi_{r\leq Ct_0}\big[\partial_t\big[\lambda_{0,0}^{\frac12}W(\lambda_{0,0}r) - \lambda_{\gamma_1,\gamma_2}^{\frac12}W(\lambda_{\gamma_1,\gamma_2}r)\big] -  \partial_tv_{smooth}\big] + \epsilon_1,
\end{equation}
where we have introduced the notation $v_{smooth} = \sum_{a=1,2}v_{smooth,a}$, the latter as in the statement of Lemma~\ref{lem:corrections}. 
Also, it is implied that the expressions gets evaluated at $t = t_0$.
\\
We shall think of $\bar{\epsilon}_0, \bar{\epsilon}_1$ as functions of $R = \lambda_{\gamma_1,\gamma_2}(t_0)r$, and we shall keep the latter definition of $R$ for the rest of the paper, as this is the correct variable to use for the sequel. 

Observe that setting 
\[
\tilde{x}_0^{(\gamma_1,\gamma_2)}(\xi) = \int_0^\infty\phi(R, \xi)R\epsilon_0(R_{0,0}(R))\,dR,\,\tilde{x}_{0d}^{(\gamma_1,\gamma_2)} = \int_0^\infty\phi_d(R)R\epsilon_0(R_{0,0}(R))\,dR
\]
\begin{align*}
\tilde{x}_1^{(\gamma_1,\gamma_2)}(\xi) = &-\lambda_{\gamma_1,\gamma_2}^{-1}\big|_{t = t_0}\int_0^\infty\phi(R, \xi)R\epsilon_1(R_{0,0}(R))\,dR - \frac{\dot{\lambda}_{\gamma_1,\gamma_2}}{\lambda_{\gamma_1,\gamma_2}}\big|_{t = t_0}(\mathcal{K}_{cc}\tilde{x}_0^{(\gamma_1,\gamma_2)})(\xi)\\
 & - \frac{\dot{\lambda}_{\gamma_1,\gamma_2}}{\lambda_{\gamma_1,\gamma_2}}\big|_{t = t_0}(\mathcal{K}_{cd}\tilde{x}_{0d}^{(\gamma_1,\gamma_2)})(\xi),
\end{align*}
and finally(recall \eqref{eq:datatransference4})
\begin{align*}
\tilde{x}_{1d}^{(\gamma_1,\gamma_2)} &= - \lambda_{\gamma_1,\gamma_2}^{-1}\big|_{t = t_0}\int_0^\infty\phi_d(R)R\epsilon_1(R_{0,0}(R))\,dR - \frac{\dot{\lambda}_{\gamma_1,\gamma_2}}{\lambda_{\gamma_1,\gamma_2}}\big|_{t = t_0}\mathcal{K}_{dd}\tilde x_{0d}^{(\gamma_1,\gamma_2)}\\& - \frac{\dot{\lambda}_{\gamma_1,\gamma_2}}{\lambda_{\gamma_1,\gamma_2}}\big|_{t = t_0}\mathcal{K}_{dc}\tilde x_{0}^{(\gamma_1,\gamma_2)},
\end{align*}

then using Lemma~\ref{lem:changeofscale}, we have 
\[
\big\|\tilde{x}_0^{(\gamma_1,\gamma_2)}(\xi) - \frac{\lambda}{\lambda_{0,0}}x_0(\frac{\lambda^2}{\lambda_{0,0}^2}\xi)\big\|_{\tilde{S}_1}\lesssim_{\tau_0} \big|\gamma_1 t_0^{k_0\nu} + \gamma_2\log t_0\cdot t_0^{k_0\nu}\big|[\big\|x_0\big\|_{\tilde{S}_1} + \big|x_{0d}\big|],
\]
while we directly infer the bound 
\[
\big|\tilde{x}_{0d}^{(\gamma_1,\gamma_2)} - x_{0d}\big|\lesssim \big|\gamma_1 t_0^{k_0\nu} + \gamma_2\log t_0\cdot t_0^{k_0\nu}\big|(\tau_0\big\|x_0\big\|_{\tilde{S}_1} + \big|x_{0d}\big|).
\]
Similarly, we obtain 
\begin{align*}
&\big\|\tilde{x}_1^{(\gamma_1,\gamma_2)}(\xi) - \frac{\lambda}{\lambda_{0,0}}x_1(\frac{\lambda^2}{\lambda_{0,0}^2}\xi)\big\|_{\tilde{S}_2}\\&\lesssim \big|\gamma_1 t_0^{k_0\nu} + \gamma_2\log t_0\cdot t_0^{k_0\nu}\big|[\big\|x_1\big\|_{\tilde{S}_2} + \big\|x_0\big\|_{\tilde{S}_1}+\big|x_{1d}\big| + \tau_0^{-1}\big|x_{0d}\big|],
\end{align*}
as well as 
\[
\big|\tilde{x}_{1d}^{(\gamma_1,\gamma_2)} - x_{1d}\big|\lesssim \big|\gamma_1 t_0^{k_0\nu} + \gamma_2\log t_0\cdot t_0^{k_0\nu}\big|(\tau_0\big\|x_0\big\|_{\tilde{S}_1} + \big|x_{0d}\big|).
\]

taking advantage of the structure of $\mathcal{K}_{cc}, \mathcal{K}_{cd}$ as detailed in Theorem~\ref{thm:transferenceop}. Recall the quantities in \eqref{eq:vanishing}
\[
 B: = \int_0^\infty\frac{\rho^{\frac{1}{2}}(\xi)x_0(\xi)}{\xi^{\frac{1}{4}}}\cos[\nu\tau_0\xi^{\frac{1}{2}}]\,d\xi,\,A: = \int_0^\infty\frac{\rho^{\frac{1}{2}}(\xi)x_1(\xi)}{\xi^{\frac{3}{4}}}\sin[\nu\tau_0\xi^{\frac{1}{2}}]\,d\xi, 
\]
and thus formulated in terms of the original data $x_{0,1}(\xi)$, and independent of $\gamma_{1,2}$. Then denoting by $\tilde{A}(\gamma_1,\gamma_2)$, resp. $\tilde{B}(\gamma_1,\gamma_2)$ the quantity defined like $A(\gamma_1,\gamma_2)$, $B(\gamma_1,\gamma_2)$ in the statement of the proposition, but with $x_j^{(\gamma_1,\gamma_2)}$ replaced by $\tilde{x}_j^{(\gamma_1,\gamma_2)}$, $j = 1, 0$,  we infer after a change of variables that 
\begin{equation}\label{eq:tildeA}
\tilde{A}(\gamma_1,\gamma_2) = A + O\big(\big|\gamma_1 t_0^{k_0\nu} + \gamma_2\log t_0\cdot t_0^{k_0\nu}\big|\tau_0[\big\|x_1\big\|_{\tilde{S}_2} + \tau_0^{-1}\big\|x_0\big\|_{\tilde{S}_1} + \big|x_{1d}\big| +\tau_0^{-1}\big|x_{0d}\big| ]\big),
\end{equation}
\begin{equation}\label{eq:tildeB}
\tilde{B}(\gamma_1,\gamma_2) = B+ O\big(\big|\gamma_1 t_0^{k_0\nu} + \gamma_2\log t_0\cdot t_0^{k_0\nu}\big|\tau_0[\big\|x_0\big\|_{\tilde{S}_1} + \tau_0^{-1}\big|x_{0d}\big|]\big),
\end{equation}
Finally, in light of \eqref{eq:barepsilon0}, \eqref{eq:barepsilon1}, introduce the Fourier transforms of the 'bulk part differences' 
\[
\tilde{\tilde{x}}_0^{(\gamma_1,\gamma_2)}(\xi) = \int_0^\infty\phi(R, \xi)R\chi_{r\leq Ct_0}\big[ \lambda_{0,0}^{\frac12}W(\lambda_{0,0}r) - \lambda_{\gamma_1,\gamma_2}^{\frac12}W(\lambda_{\gamma_1,\gamma_2}r) - v_{smooth}\big] \,dR, 
\]
\begin{align*}
\tilde{\tilde{x}}_1^{(\gamma_1,\gamma_2)}(\xi) =  -\lambda_{\gamma_1,\gamma_2}^{-1}\int_0^\infty\phi(R, \xi)R\chi_{r\leq Ct_0}\big[\partial_t&\big[\lambda_{0,0}^{\frac12}W(\lambda_{0,0}r)\\& - \lambda_{\gamma_1,\gamma_2}^{\frac12}W(\lambda_{\gamma_1,\gamma_2}r)\big] -  \partial_tv_{smooth}\big] \,dR, 
\end{align*}
and label their contributions to the expressions $A(\gamma_1,\gamma_2)$, $B(\gamma_1,\gamma_2)$, by 
\begin{align*}
&\tilde{\tilde{A}}(\gamma_1,\gamma_2) : = \int_0^\infty\frac{\rho^{\frac{1}{2}}(\xi)\tilde{\tilde{x}}_1^{(\gamma_1,\gamma_2)} (\xi)}{\xi^{\frac{3}{4}}}\sin[\lambda_{\gamma_1,\gamma_2} (\tau_0)\xi^{\frac12}\int_{\tau_0}^\infty\lambda_{\gamma_1,\gamma_2} ^{-1}(s)\,ds]\,d\xi = 0\\
&\tilde{\tilde{B}}(\gamma_1,\gamma_2): = \int_0^\infty\frac{\rho^{\frac{1}{2}}(\xi)\tilde{\tilde{x}}_0^{(\gamma_1,\gamma_2)} (\xi)}{\xi^{\frac{1}{4}}}\cos[\lambda_{\gamma_1,\gamma_2} (\tau_0)\xi^{\frac12}\int_{\tau_0}^\infty\lambda_{\gamma_1,\gamma_2} ^{-1}(s)\,ds]\,d\xi = 0.
\end{align*}
Then the first two vanishing conditions of the proposition can be formulated as 
\[
0 = A(\gamma_1,\gamma_2) = \tilde{\tilde{A}}(\gamma_1,\gamma_2) + \tilde{A}(\gamma_1,\gamma_2),\,0 = B(\gamma_1,\gamma_2) = \tilde{\tilde{B}}(\gamma_1,\gamma_2) + \tilde{B}(\gamma_1,\gamma_2),
\]
and so, in light of \eqref{eq:tildeA}, \eqref{eq:tildeB}, we find 
\begin{equation}\label{eq:tildetildeA}
 \tilde{\tilde{A}}(\gamma_1,\gamma_2) = -A + O\big(\big|\gamma_1 t_0^{k_0\nu} + \gamma_2\log t_0\cdot t_0^{k_0\nu}\big|\tau_0[\big\|x_1\big\|_{\tilde{S}_2} + \tau_0^{-1}\big\|x_0\big\|_{\tilde{S}_1} + \tau_0^{-1}\big|x_{0d}\big| + \big|x_{1d}\big|]\big),
 \end{equation}
 \begin{equation}\label{eq:tildetildeB}
 \tilde{\tilde{B}}(\gamma_1,\gamma_2) = -B + O\big(\big|\gamma_1 t_0^{k_0\nu} + \gamma_2\log t_0\cdot t_0^{k_0\nu}\big|\tau_0[\big\|x_0\big\|_{\tilde{S}_1} + \tau_0^{-1}\big|x_{0d}\big|]\big).
 \end{equation}
It remains to compute $ \tilde{\tilde{A}}(\gamma_1,\gamma_2)$, $\tilde{\tilde{B}}(\gamma_1,\gamma_2)$ in terms of $\gamma_{1,2}$, which we now do: we can write 
\begin{equation}\label{eq:differenceofW}\begin{split}
\lambda_{0,0}^{\frac12}W(\lambda_{0,0}r) - \lambda_{\gamma_1,\gamma_2}^{\frac12}W(\lambda_{\gamma_1,\gamma_2}r) &= O(\big|\gamma_1 t_0^{k_0\nu} + \gamma_2\log t_0\cdot t_0^{k_0\nu}\big|)\lambda_{0,0}^{\frac12}\phi(R, 0)\\
& +  O(\big|\gamma_1 t_0^{k_0\nu} + \gamma_2\log t_0\cdot t_0^{k_0\nu}\big|^2),
\end{split}\end{equation}
Further, we find after writing $\xi\phi(R, \xi) = \mathcal{L}\phi(R, \xi)$ and performing integration by parts 
\begin{equation}\label{eq:diraclike}
\big|\int_0^\infty \phi(R, \xi)\chi_{r\leq Ct_0}\lambda_{0,0}^{\frac12}\phi(R, 0)\,dR\big|\lesssim_N \lambda_{0,0}^{\frac12}\frac{C\tau_0}{\langle C\tau_0\xi^{\frac12}\rangle^N},
\end{equation}
whence we infer 
\begin{equation}\label{eq:bulkterm1}\begin{split}
&\big|\int_0^\infty \phi(R, \xi)\chi_{r\leq Ct_0}R\big[\lambda_{0,0}^{\frac12}W(\lambda_{0,0}r) - \lambda_{\gamma_1,\gamma_2}^{\frac12}W(\lambda_{\gamma_1,\gamma_2}r)\big]\,dR\big|\\&\lesssim_N \lambda_{0,0}^{\frac12}\frac{C\tau_0}{\langle C\tau_0\xi^{\frac12}\rangle^N}\big|\gamma_1 t_0^{k_0\nu} + \gamma_2\log t_0\cdot t_0^{k_0\nu}\big| + \big|\gamma_1 t_0^{k_0\nu} + \gamma_2\log t_0\cdot t_0^{k_0\nu}\big|^2.
\end{split}\end{equation}
We also have the important {\it{non-degeneracy property}} 
 \begin{equation}\label{eq:bulkterm2}\begin{split}
 &\lim_{R\rightarrow 0 }R^{-1}\chi_{R\leq C\tau_0}R\big[ \lambda_{0,0}^{\frac12}W(\lambda_{0,0}r) - \lambda_{\gamma_1,\gamma_2}^{\frac12}W(\lambda_{\gamma_1,\gamma_2}r)\big]\\
& =  \frac12\lambda_{0,0}^{\frac12}[\gamma_1 t_0^{k_0\nu} + \gamma_2\log t_0\cdot t_0^{k_0\nu}] + O\big(\lambda_{0,0}^{\frac12}[\gamma_1 t_0^{k_0\nu} + \gamma_2\log t_0\cdot t_0^{k_0\nu}]^2\big),
 \end{split}\end{equation}
 in the sense that the principal term on the depends linearly on $\gamma_{1,2}$ with non-vanishing factor. 
 \\
 This is to be contrasted with the {\it{vanishing property}} 
  \begin{equation}\label{eq:vsmoothvanish}
 \lim_{R\rightarrow 0 }R^{-1}\chi_{R\leq C\tau_0}Rv_{smooth}(R) = 0,
 \end{equation}
which follows from Lemma~\ref{lem:corrections}, and finally, by another integration by parts argument similar to the one for the bulk term to get \eqref{eq:bulkterm1}, and exploiting the fine structure of $v_{smooth}$ from Lemma~\ref{lem:corrections}, we get 
 \begin{equation}\label{eq:vsmoothbound}\begin{split}
 \int_0^\infty\phi(R, \xi)R\chi_{R\leq C\tau_0}v_{smooth}(R)|_{t = t_{0}}\,dR =  \big|\gamma_1t_0^{k_0\nu} + \gamma_2\log t_0 t_0^{k_0\nu}\big| O_N\big(\lambda_{0,0}^{\frac12}\frac{C\tau_0}{\langle C\tau_0\xi^{\frac12}\rangle^N}\big)
 \end{split}\end{equation}
Finally, using the precise asymptotic relation $\lim_{\xi\rightarrow 0}\xi^{\frac12}\rho(\xi) = c$ where in fact one has $c = \frac{1}{3\pi}$, see Lemma 3.4 in \cite{DoKr}, we infer that 
 \begin{equation}\label{eg:Btildetilde}\begin{split}
  \tilde{\tilde{B}}(\gamma_1,\gamma_2) = &\int_0^\infty\frac{\tilde{\tilde{x}}_0^{(\gamma_1,\gamma_2)}(\xi)\rho^{\frac12}(\xi)}{\xi^{\frac14}}\cos[\nu\tau_0\xi^{\frac12}]\,d\xi + O_{t_0}\big(\big|\gamma_1 t_0^{k_0\nu} + \gamma_2\log t_0\cdot t_0^{k_0\nu}\big|^2\big)\\
  & =\lim_{R\rightarrow 0}cR^{-1}\int_0^\infty\phi(R, \xi)\tilde{\tilde{x}}_0^{(\gamma_1,\gamma_2)}(\xi)\rho(\xi)\,d\xi\\
  & + \int_0^\infty\tilde{\tilde{x}}_0^{(\gamma_1,\gamma_2)}(\xi)[\frac{\rho^{\frac12}(\xi)}{\xi^{\frac14}} - c\rho(\xi)]\cos[\nu\tau_0\xi^{\frac12}]\,d\xi\\
  & + c\int_0^\infty\tilde{\tilde{x}}_0^{(\gamma_1,\gamma_2)}(\xi)\rho(\xi)(\cos[\nu\tau_0\xi^{\frac12}] - 1)\,d\xi\\
  & +  O_{t_0}\big( \big|\gamma_1 t_0^{k_0\nu} + \gamma_2\log t_0\cdot t_0^{k_0\nu}\big|^2\big).\\
 \end{split}\end{equation}
Observe that the extra term $O_{t_0}\big(\big|\gamma_1 t_0^{k_0\nu} + \gamma_2\log t_0\cdot t_0^{k_0\nu}\big|^2\big)$ arises from replacing 
\[
\lambda_{\gamma_1,\gamma_2} (\tau_0)\xi^{\frac12}\int_{\tau_0}^\infty\lambda_{\gamma_1,\gamma_2} ^{-1}(s)\,ds
\]
by $\nu\tau_0$.  The last term on the right in the above identity is essentially quadratic and negligible in the sequel. The second and third terms are also negligible on account of the bounds \eqref{eq:bulkterm1}, \eqref{eq:vsmoothbound} from before for the Fourier transform of the bulk part as well as $v_{smooth}$: 
 for the second term, we get (for suitable $c>0$)
 \begin{align*}
 \big|\int_0^\infty\tilde{\tilde{x}}_0^{(\gamma_1,\gamma_2)}(\xi)[\frac{\rho^{\frac12}(\xi)}{\xi^{\frac14}} - c\rho(\xi)]\cos[\nu\tau_0\xi^{\frac12}]\,d\xi\big|\lesssim  \lambda_{0,0}^{\frac12}\tau_0^{-1}\big|\gamma_1t_0^{k_0\nu} + \gamma_2\log t_0 t_0^{k_0\nu}\big|,
 \end{align*}
 while the third term becomes small upon choosing $C$ sufficiently large: 
 \begin{align*}
&\big| \int_0^\infty\tilde{\tilde{x}}_0^{(\gamma_1,\gamma_2)}(\xi)\rho(\xi)(\cos[\nu\tau_0\xi^{\frac12}] - 1)\,d\xi\big|\\
&\lesssim C^{-1}\lambda_{0,0}^{\frac12}\big|\gamma_1t_0^{k_0\nu} + \gamma_2\log t_0 t_0^{k_0\nu}\big|
 \end{align*}
 Finally, for the first term above, we have according to the earlier limiting relations \eqref{eq:bulkterm2}, \eqref{eq:vsmoothvanish} the key relation
  \begin{align*}
 &\lim_{R\rightarrow 0}R^{-1}\int_0^\infty\phi(R, \xi)\tilde{\tilde{x}}_0^{(\gamma_1,\gamma_2)}(\xi)\rho(\xi)\,d\xi\\
 & = \frac12\lambda_{0,0}^{\frac12}[\gamma_1 t_0^{k_0\nu} + \gamma_2\log t_0\cdot t_0^{k_0\nu}] + O\big(\lambda_{0,0}^{\frac12}[\gamma_1 t_0^{k_0\nu} + \gamma_2\log t_0\cdot t_0^{k_0\nu}]^2\big).
 \end{align*}
Combining the preceding bounds and identities for the various terms in the above identity for $ \tilde{\tilde{B}}(\gamma_1,\gamma_2)$ and also recalling \eqref{eq:tildetildeB}, we have obtained the first relation determining $\gamma_{1,2}$, given by 
 \begin{equation}\label{eq:gammaonetwofirst}\begin{split}
 B = &-\frac{c}{2}\lambda_{0,0}^{\frac12}[\gamma_1 t_0^{k_0\nu} + \gamma_2\log t_0\cdot t_0^{k_0\nu}] + O\big(C^{-1}\lambda_{0,0}^{\frac12}\big|\gamma_1t_0^{k_0\nu} + \gamma_2\log t_0 t_0^{k_0\nu}\big|\big)\\& + O\big(\big|\gamma_1 t_0^{k_0\nu} + \gamma_2\log t_0\cdot t_0^{k_0\nu}\big|\tau_0[\big\|x_0\big\|_{\tilde{S}_1} + \tau_0^{-1}\big|x_{0d}\big|]\big)\\&+ O_{t_0}\big(\lambda_{0,0}^{\frac12}[\gamma_1 t_0^{k_0\nu} + \gamma_2\log t_0\cdot t_0^{k_0\nu}]^2\big),
 \end{split}\end{equation}
where the first term on the right dominates all the remaining error terms, provided the data are chosen small enough. \\
 To derive the second equation determining $\gamma_{1,2}$, we recall the formula \eqref{eq:datatransference3} for $x_1^{(\gamma_1,\gamma_2)}$, which hinges on $\bar{\epsilon}_1$. Then by combining \eqref{eq:barepsilon1} with \eqref{eq:differenceofW}, we have (using the notation $\Lambda: = \frac12 + R\partial_R$)
 \begin{align*}
 R\bar{\epsilon}_1&= R\partial_t\big[(t^{k_0\nu}\gamma_1 + \log t\cdot t^{k_0\nu}\gamma_2)\lambda_{0,0}^{\frac12}R^{-1}\phi(R, 0) - v_{smooth}\big]_{t = t_0} + R\epsilon_1\\
 & + O\big(\lambda_{0,0}^{\frac12}t_0^{k_0\nu-1}\log t_0(\sum\big|\gamma_j\big|)\big|t_0^{k_0\nu}\gamma_1 + \log t_0\cdot t_0^{k_0\nu}\gamma_2\big|\big)\\
 & = c_1t_0^{-1}(t^{k_0\nu}_0\gamma_1 + \log t_0\cdot t_0^{k_0\nu}\gamma_2)\lambda_{0,0}^{\frac12}\phi(R, 0)\\&
 + c_2t_0^{-1}(t_0^{k_0\nu}\gamma_1 + \log t_0\cdot t_0^{k_0\nu}\gamma_2)\lambda_{0,0}^{\frac12}(\Lambda^2W)(R)\\
 & + \gamma_2t_0^{k_0\nu-1}\lambda_{0,0}^{\frac12}\phi(R, 0) + O\big(\lambda_{0,0}^{\frac12}t_0^{k_0\nu-1}\log t_0(\sum\big|\gamma_j\big|)\big|t_0^{k_0\nu}\gamma_1 + \log t_0\cdot t_0^{k_0\nu}\gamma_2\big|)\big)\\
 & - R\partial_t v_{smooth} + R\epsilon_1. 
 \end{align*}
Using \eqref{eq:datatransference3} which gives 
\begin{align*}
 x_1^{(\gamma_1,\gamma_2)}(\xi) = &-\lambda_{\gamma_1,\gamma_2}^{-1}\int_0^\infty\phi(R, \xi)R\bar{\epsilon}_1(R)\,dR - \frac{\dot{\lambda}_{\gamma_1,\gamma_2}}{\lambda_{\gamma_1,\gamma_2}}\big|_{t = t_0}(\mathcal{K}_{cc}x_0^{(\gamma_1,\gamma_2)})(\xi)\\
 & - \frac{\dot{\lambda}_{\gamma_1,\gamma_2}}{\lambda_{\gamma_1,\gamma_2}}\big|_{t = t_0}(\mathcal{K}_{cd}x_{0d}^{(\gamma_1,\gamma_2)})(\xi),
\end{align*}
and also keeping in mind the corresponding relations \eqref{eq:datatransference1}, \eqref{eq:datatransference2}, we deduce in light of Lemma~\ref{lem:changeofscale} the identity 
\begin{align*}
x_1^{(\gamma_1,\gamma_2)}(\xi) = \sum_{j=1}^4 A_j + \sum_{j=1}^3 B_j + C
\end{align*}
where the $A_j$ are the terms coming from the 'bulk term' and are given by 
\begin{align*}
A_1 = -c_1(t_0\lambda_{\gamma_1,\gamma_2})^{-1}\lambda_{0,0}^{\frac12}(t_0^{k_0\nu}\gamma_1 + \log t_0\cdot t_0^{k_0\nu}\gamma_2)\int_0^\infty\phi(R, \xi)\chi_{R\leq C\tau_0}\lambda_{0,0}^{\frac12}\phi(R, 0)\,dR
\end{align*}
\begin{align*}
A_2 =  - c_2(t_0\lambda_{\gamma_1,\gamma_2})^{-1}\lambda_{0,0}^{\frac12}(t_0^{k_0\nu}\gamma_1 + \log t_0\cdot t_0^{k_0\nu}\gamma_2)\int_0^\infty\phi(R, \xi)\chi_{R\leq C\tau_0}\lambda_{0,0}^{\frac12}(\Lambda^2W)(R, 0)\,dR,
\end{align*}
\begin{align*}
A_3 =  - \lambda_{0,0}^{\frac12}\gamma_2t_0^{k_0\nu-1}\lambda_{\gamma_1,\gamma_2}^{-1}\int_0^\infty\phi(R, \xi)\chi_{R\leq C\tau_0}\lambda_{0,0}^{\frac12}\phi(R, 0)\,dR
\end{align*}
\begin{align*}
A_4 = - \lambda_{\gamma_1,\gamma_2}^{-1}\int_0^\infty\phi(R, \xi)\chi_{R\leq C\tau_0}R\partial_t v_{smooth}\,dR,
\end{align*}
while the terms $B_j$ arising from the perturbation are given by 
\begin{align*}
&B_1 = \frac{\lambda}{\lambda_{0,0}}x_1(\frac{\lambda^2}{\lambda_{0,0}^2}\xi),\,B_2 =  - \frac{\dot{\lambda}_{\gamma_1,\gamma_2}}{\lambda_{\gamma_1,\gamma_2}}\big|_{t = t_0}(\mathcal{K}_{cc}x_0^{(\gamma_1,\gamma_2)})(\xi),\\&B_3 = - \frac{\dot{\lambda}_{\gamma_1,\gamma_2}}{\lambda_{\gamma_1,\gamma_2}}\big|_{t = t_0}(\mathcal{K}_{cd}x_{0d}^{(\gamma_1,\gamma_2)})(\xi)
\end{align*}
Finally, $C$ is the error, which is given crudely by 
\begin{align*}
C &= O\big(\lambda_{0,0}^{\frac12}\big|\gamma_1 t_0^{k_0\nu} + \gamma_2\log t_0\cdot t_0^{k_0\nu}\big|\tau_0\big|[\big\|x_1\big\|_{\tilde{S}_1} + \tau_0^{-1}\big\|x_0\big\|_{\tilde{S}_2}+\big|x_{1d}\big| + \tau_0^{-1}\big|x_{0d}\big|]\big)\\
& + O_{t_0}\big((\sum_j|\gamma_j|)\big|\gamma_1 t_0^{k_0\nu} + \gamma_2\log t_0\cdot t_0^{k_0\nu}\big|\big).
\end{align*}
Inserting the preceding into the expression for $A(\gamma_1,\gamma_2)$ (as in the statement of the proposition) and proceeding as for the derivation for \eqref{eq:gammaonetwofirst}, as well as observing that 
\begin{align*}
&\tau_0^{-1}\int_0^\infty\frac{\rho^{\frac{1}{2}}(\xi)\mathcal{K}_{cc}x_0^{(\gamma_1,\gamma_2)})(\xi)}{\xi^{\frac{3}{4}}}\sin[\lambda_{\gamma_1,\gamma_2} (\tau_0)\xi^{\frac12}\int_{\tau_0}^\infty\lambda_{\gamma_1,\gamma_2} ^{-1}(s)\,ds]\,d\xi\\
& = D(x_0, x_{0d}) + D_1\lambda_{0,0}^{\frac12}(\gamma_1 t_0^{k_0\nu} + \gamma_2\log t_0\cdot t_0^{k_0\nu}) + O_{t_0}\big(\big|\gamma_1 t_0^{k_0\nu} + \gamma_2\log t_0\cdot t_0^{k_0\nu}\big|^2\big)
\end{align*}
where $D(x_0, x_{0d})$ is independent of $\gamma_{1,2}$ and is bounded by $\big|D(x_0, x_{0d})\big|\lesssim \big\|x_0\big\|_{\tilde{S}_1} + \big|x_{0d}\big|$, while $D_1 = D_1(\nu)$ is a suitable absolute constant, we infer the following identity: 
\begin{equation}\label{eq:mess14}\begin{split}
&\int_0^\infty\frac{\rho^{\frac{1}{2}}(\xi)x_1^{(\gamma_1,\gamma_2)} (\xi)}{\xi^{\frac{3}{4}}}\sin[\lambda_{\gamma_1,\gamma_2} (\tau_0)\xi^{\frac12}\int_{\tau_0}^\infty\lambda_{\gamma_1,\gamma_2} ^{-1}(s)\,ds]\,d\xi\\
& = A + \sum_{j=1,2}E_j + D(x_0, x_{0d}) + F.
\end{split}\end{equation}
Here the first term on the right arises on account of \eqref{eq:tildeA}, the two terms $E_{1,2}$ arise via the contribution of the bulk terms $A_1 - A_4$ (taking advantage of estimates like \eqref{eq:diraclike}), and are of the form 
\begin{align*}
E_1 = e_1\lambda_{0,0}^{\frac12}(\gamma_1 t_0^{k_0\nu} + \gamma_2\log t_0\cdot t_0^{k_0\nu}),\,E_2 = e_2\lambda_{0,0}^{\frac12}\gamma_2 t_0^{k_0\nu},\,e_2\neq 0, 
\end{align*}
and finally, the error term $F$ is of the form 
\begin{align*}
F &= O_{t_0}\big((\sum_j|\gamma_j|)\big|\gamma_1 t_0^{k_0\nu} + \gamma_2\log t_0\cdot t_0^{k_0\nu}\big|\big)\\
&+O\big(\lambda_{0,0}^{\frac12}\big|\gamma_1 t_0^{k_0\nu} + \gamma_2\log t_0\cdot t_0^{k_0\nu}\big|\tau_0\big|[\big\|x_1\big\|_{\tilde{S}_1} + \tau_0^{-1}\big\|x_0\big\|_{\tilde{S}_2}+\big|x_{1d}\big| + \tau_0^{-1}\big|x_{0d}\big|]\big)
\end{align*}
Equating the expression on the left of \eqref{eq:mess14} with $0$, we infer the second equation, analogous to \eqref{eq:gammaonetwofirst}: 
\begin{equation}\label{eq:gammaonetwosecond}\begin{split}
A &= -e_1\lambda_{0,0}^{\frac12}(\gamma_1 t_0^{k_0\nu} + \gamma_2\log t_0\cdot t_0^{k_0\nu}) - e_2\lambda_{0,0}^{\frac12}\gamma_2 t_0^{k_0\nu} - D(x_0, x_{0d})\\
& +  O_{t_0}\big((\sum_j|\gamma_j|)\big|\gamma_1 t_0^{k_0\nu} + \gamma_2\log t_0\cdot t_0^{k_0\nu}\big|\big)\\
&+O\big(\lambda_{0,0}^{\frac12}\big|\gamma_1 t_0^{k_0\nu} + \gamma_2\log t_0\cdot t_0^{k_0\nu}\big|\tau_0\big|[\big\|x_1\big\|_{\tilde{S}_1} + \tau_0^{-1}\big\|x_0\big\|_{\tilde{S}_2}+\big|x_{1d}\big| + \tau_0^{-1}\big|x_{0d}\big|]\big)
\end{split}\end{equation}
On account of the easily verified bounds
\[
\big|A\big|\lesssim \tau_0\big\|x_1\big\|_{\tilde{S}_2},\,\big|B\big|\lesssim \tau_0\big\|x_0\big\|_{\tilde{S}_1}, 
\]
we then infer 
\[
\big|\gamma_1 \lambda_{0,0}^{\frac12}t_0^{k_0\nu}\big| + \big|\gamma_2 \lambda_{0,0}^{\frac12}\log t_0 t_0^{k_0\nu}\big|\lesssim (\log \tau_0)\cdot\tau_0(\big\|(x_0,x_1)\big\|_{\tilde{S}} + \big|x_{0d}\big|).
\]
Recall that throughout the preceding discussion we kept the discrete spectral parts $(x_{0d}, x_{1d})$ of the initial perturbation $(\epsilon_1, \epsilon_2)$ fixed. If instead we allow $x_{1d}$ to vary, we can think of $\gamma_{1,2}$ as functions of $x_{1d}$, and moreover one easily checks that 
\[
\tilde{x}^{(\gamma_1,\gamma_2)}_{1d} = x_{1d} + O([\big\|x_{0}\big\|_{\tilde{S}_1} + \big\|x_1\big\|_{\tilde{S}_2} + \big|x_{0d}\big| + \big|x_{1d}\big|]^2).
\]
with a corresponding Lipschitz bound. It follows that there is a unique choice of $x_{1d}$ such that (for given $x_0,x_1, x_{0d}$) the pair $(\tilde{x}^{(\gamma_1,\gamma_2)}_{0d}, \tilde{x}^{(\gamma_1,\gamma_2)}_{1d})$ satisfies the linear compatibility relation \eqref{eq:cond1} with respect to the scaling parameter $\lambda = \lambda_{\gamma_1,\gamma_2}$. 
\\
The last bound of the proposition follows from the preceding formulas for $x_1^{(\gamma_1,\gamma_2)}$, as well as $x_0^{(\gamma_1,\gamma_2)}$ in terms of $x_1, x_0$. Specifically, one uses the fact that for the Fourier transform of the bulk term\footnote{This means the sum of the first three terms on the right.} in \eqref{eq:barepsilon0}, we have the asymptotics \eqref{eq:differenceofW}, \eqref{eq:diraclike} as well as \eqref{eq:vsmoothbound}, and we get 
\begin{equation}\label{eq:spikeS1}
\big\|\frac{C\tau_0}{\langle C\tau_0\xi^{\frac12}\rangle^N}\big\|_{\tilde{S}_1} \lesssim \tau_0^{-(1-\delta_0)}.
\end{equation}
\end{proof}

For later purposes, we also mention the following important Lipschitz continuity properties, which follow easily from the preceding proof: 
\begin{lem}\label{lem:Lipcont} Let $(\tilde{\gamma}_1,\tilde{\gamma}_2)$ the parameters associated with a data quadruple $(\underline{\tilde{x}}_0, \underline{\tilde{x}}_1)\in \tilde{S}$. Then using the notation from before and putting 
\[
\tilde{\lambda} = \lambda_{(\tilde{\gamma}_1,\tilde{\gamma_2})},
\]
we have 
\begin{align*}
&\big|(\gamma_1 - \tilde{\gamma}_1)\lambda_{0,0}^{\frac12}t_0^{k_0\nu}\big| + \big|(\gamma_2 - \tilde{\gamma}_2)\lambda_{0,0}^{\frac12}\log t_0 t_0^{k_0\nu}\big|\lesssim \tau_0\log\tau_0[\big\|(x_0 - \tilde{x}_0, x_1 - \tilde{x}_1)\big\|_{\tilde{S}}\\
&\hspace{8cm} + \big\|(x_0,x_1)\big\|_{\tilde{S}}\big|x_{0d} - \tilde{x}_{0d}\big|]. 
\end{align*}
\begin{align*}
&\big\|x_0^{(\gamma_1,\gamma_2)} - \tilde{x}_0^{(\tilde{\gamma}_1,\tilde{\gamma}_2)} - (\frac{\lambda_{0,0}}{\lambda}S_{\frac{\lambda_{0,0}^2}{\lambda^2}}x_0 - \frac{\lambda_{0,0}}{\tilde{\lambda}}S_{\frac{\lambda_{0,0}^2}{\tilde{\lambda}^2}}\tilde{x}_0)\big\|_{\tilde{S}_1}\\&\hspace{3cm} + \big\|x_1^{(\gamma_1,\gamma_2)} - \tilde{x}_1^{(\tilde{\gamma}_1,\tilde{\gamma}_2)} - (\frac{\lambda_{0,0}}{\lambda}S_{\frac{\lambda_{0,0}^2}{\lambda^2}}x_1 - \frac{\lambda_{0,0}}{\tilde{\lambda}}S_{\frac{\lambda_{0,0}^2}{\tilde{\lambda}^2}}\tilde{x}_1)\big\|_{\tilde{S}_2}\\&\lesssim \log\tau_0\cdot\tau_0^{0+}\cdot[\big\|(x_0 - \tilde{x}_0, x_1 - \tilde{x}_1)\big\|_{\tilde{S}} +  \big\|(x_0,x_1)\big\|_{\tilde{S}}\big|x_{0d} - \tilde{x}_{0d}\big|].
\end{align*}
Finally, we have the bound 
\begin{align*}
&\big|(x_{1d}^{(\gamma_1,\gamma_2)} - x_{1d}) - (\tilde{x}_{1d}^{(\gamma_1,\gamma_2)} - \tilde{x}_{1d})\big|\\&\lesssim [\big\|(x_0 - \tilde{x}_0, x_1 - \tilde{x}_1)\big\|_{\tilde{S}} + \big|x_{0d} - \tilde{x}_{0d}\big|]\cdot [\big\|(x_0,x_1)\big\|_{\tilde{S}} + \big|x_{0d}\big|].
\end{align*}
\end{lem}

\section{Iterative construction of blow up solution almost matching the perturbed initial data}

Here we carry out the actual construction of the solution, as explained in the paragraph following \eqref{eq:vanishing2}. Thus departing from perturbed data 
\[
u_{\nu}[t_0] + (\epsilon_0, \epsilon_1), 
\]
where the perturbation $(\epsilon_0, \epsilon_1)$ is associated with a data quadruple $(\underline{x}_0, \underline{x}_1)$ as in \eqref{eq:x0epsilon0}, \eqref{eq:datatransference1}, \eqref{eq:datatransference2}, where $x_{1d}$, as well as parameters $\gamma_{1,2}$ have been computed according to Proposition~\ref{prop:choiceofgamma}, in terms of the Fourier data $(x_0(\xi), x_1(\xi), x_{0d})$, we then pass to a different representation of the data which coincides with the preceding data in a dilate of the light cone at time $t = t_0$, i. e. we have 
\[
\chi_{r\leq Ct_0}u_{approx}^{(0,0)}[t_0] + (\epsilon_1, \epsilon_2) = \chi_{r\leq Ct_0}u_{approx}^{(\gamma_1,\gamma_2)}[t_0] + (\bar{\epsilon}_1, \bar{\epsilon}_2). 
\]
Then, according to Proposition~\ref{prop:choiceofgamma}, the Fourier data associated to $(\bar{\epsilon}_1, \bar{\epsilon}_2)$ in reference to the coordinate $R: = \lambda_{\gamma_1,\gamma_2}(t_0)r$, satisfy the key vanishing relations 
\[
A(\gamma_1,\gamma_2) = B(\gamma_1,\gamma_2) = 0,
\]
these quantities  being defined as in Proposition~\ref{prop:choiceofgamma}. We shall now strive to evolve the data 
\[
u_{approx}^{(\gamma_1,\gamma_2)}[t_0] + (\bar{\epsilon}_1, \bar{\epsilon}_2)
\]
backwards in time from $t = t_0$, and thereby build another blow up solution with bulk part $u_{approx}^{(\gamma_1,\gamma_2)}(t, x)$ on the time slice $(0, t_0]\times \R^3$. 

\subsection{Formulation of the perturbation problem on Fourier side}

Re-iterating that we shall work with the coordinates 
\begin{equation}\label{eq:thecoordinates}
\tau: = \int_t^\infty \lambda_{\gamma_1,\gamma_2}(s)\,ds,\,R = \lambda_{\gamma_1,\gamma_2}(t)r,
\end{equation}
we shall write the desired solution in the form 
\begin{equation}\label{eq:thesolutionansatz}
u(t, x) = u_{approx}^{(\gamma_1,\gamma_2)}(t, x) + \epsilon(t, x),\,\epsilon[t_0] = (\bar{\epsilon}_1, \bar{\epsilon}_2), 
\end{equation}
and passing to the variable $\til{\eps}(\tau, R): = R\epsilon(\tau, R)$, we derive the following equation completely analogous to \eqref{eq:epseqn}: using from now on $\lambda(\tau) = \lambda_{\gamma_1,\gamma_2}(\tau)$, 
\begin{equation}\label{eq:tildeepsilon1}\begin{split}
&(\partial_{\tau} + \dot{\lambda}\lambda^{-1}R\partial_R)^2 \tilde{\eps} - \beta(\tau)(\partial_{\tau} + \dot{\lambda}\lambda^{-1}R\partial_R)\tilde{\eps} + \calL\tilde{\eps}\\
&=\lambda^{-2}(\tau)R[N_{approx}(\eps) + e_{approx}]+\partial_\tau(\dot{\lambda}\lambda^{-1})\tilde{\eps};\,\beta(\tau) = \dot{\lambda}(\tau)\lambda^{-1}(\tau), 
\end{split}\end{equation}
where we use the notation 
\[
RN_{approx}(\eps) = 5(u_{approx}^4 - u_0^4)\tilde{\eps} + RN(u_{approx}, \tilde{\eps}),
\]
\[
RN(u_{approx}, \tilde{\eps}) = R(u_{approx}+\frac{\tilde{\eps}}{R})^5 - R u_{approx}^5 - 5u_{approx}^4\tilde{\eps},
\]
and $u_{approx} = u_{approx}^{(\gamma_{1,2})}$. The source term $ e_{approx}$ is precisely the one in Theorem~\ref{thm:main approximate}. Also, observe that we may and shall include cutoffs to the right hand source terms of the form $\chi_{R\leq C\tau}$, since we are only interested in the behaviour of the solution inside the forward light cone emanating from the origin. 
Ideally we will want to match 
\[
\epsilon[t_{0}] = (\bar{\epsilon}_1, \bar{\epsilon}_2),
\]
but we shall have to deviate from this by a small error. In order to solve \eqref{eq:tildeepsilon1}, we pass to the distorted Fourier transform of $\tilde{\eps}$, by using the representation 
\[
\tilde{\eps}(\tau, R) = x_d(\tau)\phi_d(R) + \int_0^\infty x(\tau, \xi)\phi(R, \xi)\rho(\xi)\,d\xi.
\]
Writing 
\[
\underline{x}(\tau, \xi): = \left(\begin{array}{c}x_d(\tau)\\ x(\tau, \xi)\end{array}\right),\qquad \underline{\xi} = \binom{\xi_d}{\xi},
\]
we infer 
\begin{equation}\label{eq:transport1}
\big(\mathcal{D}_{\tau}^2 + \beta(\tau)\mathcal{D}_{\tau} + \underline{\xi}\big)\underline{x}(\tau, \xi) = \calR(\tau, \underline{x}) + \underline{f}(\tau, \underline{\xi}),\,\underline{f} = \left(\begin{array}{c}f_d\\f\end{array}\right),
\end{equation}
combined with the initial data (which in turn obey \eqref{eq:x0barepsilon0}, \eqref{eq:datatransference3}, \eqref{eq:datatransference4})
\begin{equation}\label{eq:xeqndata}
\underline{x}(\tau_0, \xi) = \left(\begin{array}{c}x_{0d}^{(\gamma_1,\gamma_2)}\\ x_{0}^{(\gamma_1,\gamma_2)}\end{array}\right),\,\mathcal{D}_{\tau}\underline{x}(\tau_0, \xi) = \left(\begin{array}{c}x_{1d}^{(\gamma_1,\gamma_2)}\\ x_{1}^{(\gamma_1,\gamma_2)}\end{array}\right),\,\tau_0 = \tau(t_0). 
\end{equation}

where we have
\begin{equation}\label{eq:Rterms}
\calR(\tau, \underline{x})(\xi) = \Big(-4\beta(\tau)\mathcal{K}\mathcal{D}_{\tau}\underline{x} - \beta^2(\tau)(\mathcal{K}^2 + [\mathcal{A}, \mathcal{K}] + \mathcal{K} +  \beta'\beta^{-2}\mathcal{K})\underline{x}\Big)(\xi)
\end{equation}
 with $\beta(\tau) = \frac{\dot{\lambda}(\tau)}{\lambda(\tau)}$, and 
 \begin{equation}\label{eq:fterms}\begin{split}
 &f(\tau, \xi) = \mathcal{F}\big( \lambda^{-2}(\tau)\chi_{R\lesssim\tau}\big[5(u_{approx}^4 - u_0^4)\tilde{\eps} + RN(u_{approx}, \tilde{\eps}) + R e_{approx}\big]\big)\big(\xi\big)\\
 &f_d(\tau) = \langle\lambda^{-2}(\tau)\chi_{R\lesssim\tau}\big[5(u_{approx}^4 - u_0^4)\tilde{\eps} + RN(u_{approx}, \tilde{\eps}) + R e_{approx}\big],\,\phi_d(R)\rangle.
 \end{split}\end{equation}
Also the key operator 
 \[
 \mathcal{D}_{\tau} = \partial_{\tau} + \beta(\tau)\mathcal{A},\quad \mathcal{A} = \left(\begin{array}{cc}0&0\\0&\mathcal{A}_c\end{array}\right)
 \]
 and we have 
 \[
 \mathcal{A}_c = -2\xi\partial_{\xi} - \Big (\frac{5}{2}  + \frac{\rho'(\xi)\xi}{\rho(\xi)} \Big).
 \]
The operator $\mathcal{K}$ is described in Theorem~\ref{thm:transferenceop}.

The main technical result of this article then furnishes a solution of \eqref{eq:transport1}, \eqref{eq:xeqndata} as follows: 

\begin{thm}\label{thm:IterationResult} Let $(x_0^{(\gamma_1, \gamma_2)}, x_1^{(\gamma_1,\gamma_2)})\in \tilde{S}, x_{ld}^{(\gamma_1,\gamma_2)}$, $l = 0, 1$, be as in Proposition~\ref{prop:choiceofgamma}, and assume $t_0$ is sufficiently small, or analogously, $\tau_0$ is sufficiently large. Then there exist corrections 
\[
(\triangle x_0^{(\gamma_1, \gamma_2)}, \triangle x_1^{(\gamma_1,\gamma_2)}),\,(\triangle x_{0d}^{(\gamma_1, \gamma_2)}, \triangle x_{1d}^{(\gamma_1,\gamma_2)})
\]
satisfying 
\[
\big\|(\triangle x_0^{(\gamma_1, \gamma_2)}, \triangle x_1^{(\gamma_1,\gamma_2)})\big\|_{\tilde{S}}\ll\big\|(x_0,x_1)\big\|_{\tilde{S}} + \big|x_{0d}\big|, 
\]
\[
\big|\triangle x_{0d}^{(\gamma_1, \gamma_2)}\big| + \big|\triangle x_{1d}^{(\gamma_1,\gamma_2)}\big|\ll\big\|(x_0,x_1)\big\|_{\tilde{S}} + \big|x_{0d}\big|,
\]
and such that the $(\triangle x_0^{(\gamma_1, \gamma_2)}, \triangle x_1^{(\gamma_1,\gamma_2)})$, $(\triangle x_{0d}^{(\gamma_1, \gamma_2)}, \triangle x_{1d}^{(\gamma_1,\gamma_2)})$ depend in Lipschitz continuous fashion on $(x_0,x_1, x_{0d})$ with respect to $\|\cdot\|_{\tilde{S}} + \big|\cdot\big|$ with Lipschitz constant $\ll 1$, and such that the equation \eqref{eq:transport1} with initial data 
\[
\big(x(\tau_0, \xi),\,(\mathcal D_\tau x)(\tau_0,\xi)\big) = \big(x_0^{(\gamma_1, \gamma_2)} + \triangle x_0^{(\gamma_1, \gamma_2)},\,x_1^{(\gamma_1, \gamma_2)} + \triangle x_1^{(\gamma_1, \gamma_2)}\big)
\]
\[
\big(x_d(\tau_0),\,\partial_\tau x_d(\tau_0)\big) = \big(x_{0d}^{(\gamma_1, \gamma_2)} + \triangle x_{0d}^{(\gamma_1, \gamma_2)},\,x_{1d}^{(\gamma_1, \gamma_2)} + \triangle x_{1d}^{(\gamma_1, \gamma_2)}\big)
\]
admits a solution $\underline{x}(\tau, \xi)$ for $\tau\geq \tau_0$ satisfying 
\[
\big\|(x(\tau, \cdot), \mathcal{D}_{\tau}x(\tau, \cdot)\big\|_{\tilde{S}} + \big|x_d(\tau)\big| + \big|\partial_{\tau}x_d(\tau)\big|\lesssim_\tau \|(x_0, x_1)\|_{\tilde{S}} + \big|x_{0d}\big|, 
\]
corresponding to $\tilde{\epsilon}(\tau, R)\in H_{loc}^{\frac{3}{2}+}$ where 
\[
\tilde{\epsilon}(\tau, R) =  x_d(\tau)\phi_d(R) + \int_0^\infty x(\tau, \xi)\phi(R, \xi)\rho(\xi)\,d\xi.
\]
Finally, we have energy decay within the light cone: 
\[
\lim_{t\rightarrow 0}\int_{|x|\leq t}\frac12 |\nabla_{t,x}\epsilon|^2\,dx = 0
\]
where we recall $\epsilon = R^{-1}\tilde{\epsilon}$. 
\end{thm}

\begin{rem}\label{rem:correction detail} In fact, the Fourier coefficients $(\triangle x_0^{(\gamma_1, \gamma_2)}, \triangle x_1^{(\gamma_1, \gamma_2)})$ will have a very specific form, which makes them well-behaved with respect to re-scalings (which hence don't entail smoothness loss when passing to differences). This shall be important when reverting to the original coordinates $R_{0,0}$ at time $t = t_0$, which were used to specify the perturbation $(x_0, x_1)$ to begin with. 
\end{rem}

\subsection{The proof of Theorem~\ref{thm:IterationResult}}

It is divided into two parts: the existence part for the solution, which follows essentially verbatim the scheme in \cite{CondBlow}, and the more delicate verification of Lipschitz dependence of the solution on the data $(x_0,x_1,x_{0d})$. Here the issue is the fact that there are re-scalings involved, and the very parametrix used to solve \eqref{eq:transport1}, as well as the source terms there, depend implicitly on $\gamma_{1,2}$, which in turn depend on $(x_0, x_1, x_{0d})$. 
\\

\subsubsection{Setup of the iteration scheme; the zeroth iterate}
Proceeding in close analogy to \cite{CondBlow}, we shall obtain the final solution $\underline{x}(\tau, \xi)$ of \eqref{eq:transport1} as the limit of a sequence of iterates $\underline{x}^{(j)}(\tau, \xi)$. To begin with, we introduce the zeroth iterate in the following proposition. The only difference compared to \cite{CondBlow} is the presence of the error term $e_{approx}$, whose dependence on $\gamma_{1,2}$ needs to be taken carefully into account. 
\\

To formulate the bounds on the successive iterates, we introduce a number of notations. First, we recall \eqref{eq:Stildenorm}, which is used to control data sets, and we also introduce the slightly stronger norm
\begin{equation}\label{eq:Snorm}
\big\|((x_0, x_1)\big\|_{S}: = \big\|x_0\big\|_{S_1} + \big\|x_1\big\|_{\tilde{S}_2}: = \big\|\langle\xi\rangle^{1+2\delta_0}\xi^{-\delta_0}x_0\big\|_{L^2_{d\xi}} + \big\|\langle\xi\rangle^{\frac12+2\delta_0}\xi^{-\delta_0}x_1\big\|_{L^2_{d\xi}}. 
\end{equation}

Denote the propagator \eqref{eq:linhomparam1} by $S(\tau)(x_0, x_1)$, and further introduce the inhomogeneous propagator solving the problem with source (this only involves the continuous spectral part)
\[
\big(\mathcal{D}_{\tau}^2 + \beta(\tau)\mathcal{D}_{\tau} + \xi\big)x(\tau, \xi) = h(\tau, \xi),\,(x(\tau_0, \xi) = 0,\,\mathcal{D}_{\tau}x(\tau_0, \xi) = 0
\]
by 
\begin{equation}\label{eq:inhompara}\begin{split}
&x(\tau, \xi) = \int_{\tau_0}^{\tau}U(\tau,\sigma)h(\sigma, \frac{\lambda^2(\tau)}{\lambda^2(\sigma)}\xi)\,d\sigma,\\& U(\tau, \sigma)=  \frac{\lambda^{\frac{3}{2}}(\tau)}{\lambda^{\frac{3}{2}}(\sigma)}\frac{\rho^{\frac{1}{2}}(\frac{\lambda^{2}(\tau)}{\lambda^{2}(\sigma)}\xi)}{\rho^{\frac{1}{2}}(\xi)}\frac{\sin\Big[\lambda(\tau)\xi^{\frac{1}{2}}\int_{\sigma}^{\tau}\lambda^{-1}(u)\,du\Big]}{\xi^{\frac12}}
\end{split}\end{equation}
Further, denote the evolution of the spectral part with inhomogeneous data 
\[
\big(\partial_{\tau}^2 + \beta(\tau)\partial_{\tau} + \xi_d\big)x_d(\tau) = h_d(\tau),\,x_d(\tau_0) = 0,\,\partial_{\tau}x_d(\tau_0) = 0
\]
and without exponential decay at infinity (for bounded $h_d$ for example) by 
\begin{equation}\label{eq:inhomparad}
x_d(\tau) = \int_{\tau_0}^\tau H(\tau, \sigma) h_d(\sigma)\,d\sigma,
\end{equation}
where we have (see Lemma~\ref{lem:linhom}) the bound $\big|H(\tau, \sigma) \big|\lesssim e^{-c|\tau-\sigma|}$ for some $c>0$. 
Following \cite{CondBlow}, we also introduce the somewhat complicated square-sum norms over dyadic time intervals and given by 
\begin{equation}\label{eq:squaresumnorm}
\big\|y(\tau, \xi)\big\|_{Sqr}: = \big(\sum_{\substack{N\gtrsim \tau_0\\N\,\text{dyadic}}}\sup_{\tau\sim N}(\frac{\lambda(\tau)}{\lambda(\tau_0)})^{4\delta_0}\big\|y(\tau, \cdot)\big\|_{L^2_{d\xi}}^2\big)^{\frac12},
\end{equation}
and we shall also use $\big\|y(\tau, \xi)\big\|_{Sqr(\xi<1)}, \big\|y(\tau, \xi)\big\|_{Sqr(\xi>1)}$ where $L^2_{d\xi}$ will be refined to $L^2_{d\xi}(\xi<1), L^2_{d\xi}(\xi>1)$. To define the zeroth iterate $\underline{x}^{(0)}$, we replace the source functions in the right-hand side of \eqref{eq:transport1} by 
\[
\left(\begin{array}{c} \langle \lambda^{-2}\chi_{R\lesssim\tau}R e_{approx}, \phi_d(R)\rangle\\ \mathcal{F}\big(\lambda^{-2}\chi_{R\lesssim\tau}R e_{approx}\big)\end{array}\right)
\]
and we abolish the linear term $\mathcal R$. That is $\underline{x}^{(0)}$ satisfies the equation 
\[
\big(\mathcal{D}_{\tau}^2 + \beta(\tau)\mathcal{D}_{\tau} + \underline{\xi}\big)\underline{x}^{(0)}(\tau, \xi) = \left(\begin{array}{c} \langle \lambda^{-2}\chi_{R\lesssim\tau}R e_{approx}, \phi_d(R)\rangle\\ \mathcal{F}\big(\lambda^{-2}\chi_{R\lesssim\tau}R e_{approx}\big)\end{array}\right)
\]

\begin{prop}\label{prop:zerothiterate} Assume the same setup as in Theorem~\ref{thm:IterationResult}. In particular, as before, everything depends on a basic data triple $(x_0(\xi), x_1(\xi), x_{0d})$ from which the fourth component $x_{1d}$ and further the new Fourier components $x_{0}^{(\gamma_1,\gamma_2)}$ etc are derived. There is a pair $(\triangle \tilde{\tilde{x}}_0^{(0)}, \triangle \tilde{\tilde{x}}_1^{(0)})\in \tilde{S}$, satisfying the bounds 
\begin{align*}
\big\|(\triangle \tilde{\tilde{x}}_0^{(0)}, \triangle \tilde{\tilde{x}}_1^{(0)})\big\|_{\tilde{S}}\lesssim \tau_0^{-(2-)}[\big\|(x_0, x_1)\big\|_{\tilde{S}} + \big|x_{0d}\big|], 
\end{align*}
and such that if we set for the continuous spectral part \begin{align*}
&x^{(0)}(\tau, \xi): =  \widetilde{x^{(0)}}(\tau, \xi)+S(\tau)\big(x_0^{(\gamma_1,\gamma_2)} + \triangle \tilde{\tilde{x}}_0^{(0)}, x_1^{(\gamma_1,\gamma_2)} + \triangle \tilde{\tilde{x}}_1^{(0)}\big),\\
&\widetilde{x^{(0)}}(\tau, \xi)
: =\int_{\tau_0}^{\tau}U(\tau,\sigma)\mathcal{F}\big(\lambda^{-2}\chi_{R\lesssim\tau}R e_{approx}\big)(\sigma, \frac{\lambda^2(\tau)}{\lambda^2(\sigma)}\xi)\,d\sigma\\
\end{align*}
then the following conclusions hold: for high frequencies $\xi>1$, we have 
\begin{align*}
&\sup_{\tau\geq \tau_0}(\frac{\tau}{\tau_0})^{-\kappa}\big\|\chi_{\xi>1}\widetilde{x^{(0)}}(\tau, \xi)\big\|_{S_1} + \sup_{\tau\geq \tau_0}(\frac{\tau}{\tau_0})^{\kappa}\big\|\chi_{\xi>1}\mathcal{D}_{\tau}\widetilde{x^{(0)}}(\tau, \xi)\big\|_{S_2}\\
&+ \big\|\xi^{\frac12+\delta_0}\mathcal{D}_{\tau}\widetilde{x^{(0)}}(\tau, \xi)\big\|_{Sqr(\xi>1)}\lesssim \tau_0^{-1}[\big\|(x_0, x_1)\big\|_{\tilde{S}} + \big|x_{0d}\big|].
\end{align*}
For low frequencies $\xi<1$, there is a decomposition 
\begin{align*}
\widetilde{x^{(0)}}(\tau, \xi) + S(\tau)\big(\triangle \tilde{\tilde{x}}_0^{(0)}, \triangle \tilde{\tilde{x}}_1^{(0)}\big) = \triangle_{>\tau}\widetilde{x^{(0)}}(\tau, \xi) + S(\tau)\big(\triangle\widetilde{x^{(0)}}_0(\xi), \triangle\widetilde{x^{(0)}}_1(\xi)\big)
\end{align*}
where the data $\big(\triangle\widetilde{x^{(0)}}_0(\xi), \triangle\widetilde{x^{(0)}}_1(\xi)\big)$ satisfy the vanishing conditions
\begin{equation}\label{eq:vanishing11}
\int_0^\infty \frac{\rho^{\frac12}(\xi)\triangle\widetilde{x^{(0)}}_0(\xi)}{\xi^{\frac14}}\cos[\lambda(\tau_0)\xi^{\frac12}\int_{\tau_0}^\infty\lambda^{-1}(u)\,du]\,d\xi = 0,
\end{equation}
\begin{equation}\label{eq:vanishing12}
\int_0^\infty \frac{\rho^{\frac12}(\xi)\triangle\widetilde{x^{(0)}}_1(\xi)}{\xi^{\frac34}}\sin[\lambda(\tau_0)\xi^{\frac12}\int_{\tau_0}^\infty\lambda^{-1}(u)\,du]\,d\xi = 0,
\end{equation}
and such that we have the bound 
\begin{align*}
&\big\|\big(\triangle\widetilde{x^{(0)}}_0(\xi), \triangle\widetilde{x^{(0)}}_1(\xi)\big)\big\|_{\tilde{S}} + \sup_{\tau\geq \tau_0}(\frac{\tau}{\tau_0})^{-\kappa}\big\|\chi_{\xi<1}\triangle_{>\tau}\widetilde{x^{(0)}}(\tau, \xi)\big\|_{S_1}\\
& + \big\|\xi^{-\delta_0}\mathcal{D}_{\tau}\triangle_{>\tau}\widetilde{x^{(0)}}(\tau, \xi)\big\|_{Sqr(\xi<1)}\lesssim \tau_0^{-1}[\big\|(x_0, x_1)\big\|_{\tilde{S}} + \big|x_{0d}\big|].
\end{align*}
Furthermore, letting $\triangle\tilde{\tilde{x}}_j^{(0)}, \triangle \tilde{\tilde{\bar{x}}}_j^{(0)}$, $j = 1,2$, be the corrections corresponding to two initial perturbation quadruples (where the component $x_{1d}$ is determined in terms of the other three ones via Proposition~\ref{prop:choiceofgamma})
\[
(\underline{x}_0, \underline{x}_1),\,(\underline{\bar{x}}_0, \underline{{\bar{x}}}_1), 
\]
we have 
we have 
\begin{align*}
\big\|(\triangle \tilde{\tilde{x}}_0^{(0)} -\triangle \tilde{\tilde{\bar{x}}}_0^{(0)}, \triangle \tilde{\tilde{x}}_1^{(0)} - \triangle \tilde{\tilde{\bar{x}}}_1^{(0)})\big\|_{\tilde{S}}\lesssim  \tau_0^{-(1-)}[\big\|(x_0- \bar{x}_0, x_1-\bar{x}_1)\big\|_{\tilde{S}} + \big|x_{0d} - \bar{x}_{0d}\big|].
\end{align*}
For the discrete spectral part, setting 
\begin{align*}
\triangle x^{(0)}_d(\tau): = \int_{\tau_0}^\infty H_d(\tau, \sigma)\langle \lambda^{-2}(\sigma)Re_{approx}, \phi_d(R)\rangle\,d\sigma,
\end{align*}
we have the bound 
\begin{align*}
\tau^2[\big|\triangle x^{(0)}_d(\tau)\big| + \big|\partial_{\tau}\triangle x^{(0)}_d(\tau)\big|]\lesssim \big\|(x_0, x_1)\big\|_{\tilde{S}} + \big|x_{0d}\big|.
\end{align*}
We also have the difference bound 
\begin{align*}
&\tau^2[\big|\triangle x^{(0)}_d(\tau) - \triangle \bar{x}^{(0)}_d(\tau)\big| + \big|\partial_{\tau}\triangle x^{(0)}_d(\tau) - \partial_{\tau}\triangle \bar{x}^{(0)}_d(\tau)\big|]\\&\lesssim \big\|(x_0 - \bar{x}_0, x_1 - \bar{x}_1)\big\|_{\tilde{S}} + \big|x_{0d} - \bar{x}_{0d}\big|.
\end{align*}
We shall then set 
\[
x^{(0)}_{d}(\tau): = x^{(\gamma_1,\gamma_2)}_{d}(\tau) + \triangle x^{(0)}_d(\tau),
\]
where $ x^{(\gamma_1,\gamma_2)}_{d}(\tau)$ is the 'free evolution' of the discrete spectral part constructed as in Lemma~\ref{lem:linhom} with data $(x_{0d}^{(\gamma_1,\gamma_2)}, x_{1d}^{(\gamma_1,\gamma_2)})$. 

\end{prop}
\begin{rem} Observe that the formula for the continuous spectral part $x^{(0)}(\tau, \xi)$ arises by adding the term $S(\tau)\big( \triangle \tilde{\tilde{x}}_0^{(0)},  \triangle \tilde{\tilde{x}}_1^{(0)}\big)$ to the Duhamel type parametrix coming from Lemma~\ref{lem:linhom}. The reason for such a correction term, which is already present in \cite{CondBlow}, comes from poor low frequency bounds for the term 
\[
\int_{\tau_0}^{\tau}U(\tau,\sigma)\mathcal{F}\big(\lambda^{-2}\chi_{R\lesssim\tau}R e_{approx}\big)(\sigma, \frac{\lambda^2(\tau)}{\lambda^2(\sigma)}\xi)\,d\sigma, 
\]
and more generally, for any such term occurring in the iterative scheme. The idea then is to write this bad term (for small $\xi$) in the form 
\[
\triangle_{>\tau}\widetilde{x^{(0)}}(\tau, \xi) + S(\tau)\big(\triangle\widetilde{x^{(0)}}_0(\xi), \triangle\widetilde{x^{(0)}}_1(\xi)\big),
\]
by replacing the integral $\int_{\tau_0}^{\tau}$ by one over $\int_{\tau}^\infty$. Since the components 
\[
\triangle\widetilde{x^{(0)}}_0(\xi), \triangle\widetilde{x^{(0)}}_1(\xi)
\]
don't necessarily satisfy the vanishing conditions \eqref{eq:vanishing11}, \eqref{eq:vanishing12}, we need to add the corrections $ \triangle \tilde{\tilde{x}}_0^{(0)},  \triangle \tilde{\tilde{x}}_1^{(0)}$. Importantly, these can be chosen to be much smaller than the initial data $x_{0,1}, x_{0d}$. This procedure is explained in greater detail in \cite{CondBlow}.
\end{rem}
\begin{proof} We follow the same outline of steps as for example in the proof of Prop. 8. 1 in \cite{CondBlow}. 
\\

{\bf{Step 1}}: {\it{Proof of the high frequency bound}}. Recall \eqref{eq:eapprox}. Correspondingly, we shall write $\mathcal{F}\big(\lambda^{-2}\chi_{R\lesssim\tau}R e_{approx}\big)$ as the sum of several terms. We shall prove the somewhat more delicate square-sum type bound, the remaining bounds being more of the same.
\\

{\it{The contribution of $ e_{\text{prelim}} - \tilde{e}_{\text{prelim}}$.}} Write 
\begin{align*}
\Xi_1(\tau, \xi): = \int_{\tau_0}^{\tau}U(\tau,\sigma)\mathcal{F}\big(\lambda^{-2}\chi_{R\lesssim\tau}R(e_{\text{prelim}} - \tilde{e}_{\text{prelim}}) \big)(\sigma, \frac{\lambda^2(\tau)}{\lambda^2(\sigma)}\xi)\,d\sigma
\end{align*}
We need to bound $\big\|\xi^{\frac12+\delta_0}\mathcal{D}_{\tau}\Xi_1(\tau, \xi)\big\|_{Sqr(\xi>1)}$. Observe that 
\begin{equation}\label{eq:Dtaueffect}
 \mathcal{D}_{\tau}\big[\int_{\tau_0}^{\tau}U(\tau,\sigma)g(\sigma, \frac{\lambda^2(\tau)}{\lambda^2(\sigma)}\xi)\,d\sigma\big] = \int_{\tau_0}^{\tau}V(\tau,\sigma)g(\sigma, \frac{\lambda^2(\tau)}{\lambda^2(\sigma)}\xi)\,d\sigma,
 \end{equation}
 where we set 
 \[
 V(\tau,\sigma): = \frac{\lambda^{\frac{3}{2}}(\tau)}{\lambda^{\frac{3}{2}}(\sigma)}\frac{\rho^{\frac{1}{2}}(\frac{\lambda^{2}(\tau)}{\lambda^{2}(\sigma)}\xi)}{\rho^{\frac{1}{2}}(\xi)}\cos\Big[\lambda(\tau)\xi^{\frac{1}{2}}\int_{\sigma}^{\tau}\lambda^{-1}(u)\,du\Big].
 \]
 In light of Prop.~\ref{prop:Fourier}, we infer the inequality 
 \[
\frac{\lambda^{\frac32}(\tau)}{\lambda^{\frac32}(\sigma)}\frac{\rho^{\frac12}(\frac{\lambda^2(\tau)}{\lambda^2(\sigma)}\xi)}{\rho^{\frac12}(\xi)}\lesssim\frac{\lambda^2(\tau)}{\lambda^2(\sigma)},\,\xi>1, 
\]
and this implies 
\begin{align*}
&\big\|\xi^{\frac12+}\mathcal{D}_{\tau}\Xi_1(\tau, \xi)\big\|_{L^2_{d\xi}(\xi>1)}\\&\lesssim \int_{\tau_0}^{\tau}\frac{\lambda^2(\tau)}{\lambda^2(\sigma)}\big\|\xi^{\frac12+\delta_0}\mathcal{F}\big(\lambda^{-2}(\sigma)\chi_{R\lesssim\tau}R(e_{\text{prelim}} - \tilde{e}_{\text{prelim}})\big)(\sigma, \frac{\lambda^2(\tau)}{\lambda^2(\sigma)}\xi)\big\|_{L^2_{d\xi}(\xi>1)}\,d\sigma
\end{align*}
Referring to the same proposition for the isometry properties of the distorted Fourier transform, as well as Lemma~\ref{lem:Sobolev}, we obtain
 \begin{align*}
&\frac{\lambda^2(\tau)}{\lambda^2(\sigma)}\big\|\xi^{\frac12+\delta_0}\mathcal{F}\big(\lambda^{-2}(\sigma)\chi_{R\lesssim\tau}R(e_{\text{prelim}} - \tilde{e}_{\text{prelim}})\big)(\sigma, \frac{\lambda^2(\tau)}{\lambda^2(\sigma)}\xi)\big\|_{L^2_{d\xi}(\xi>1)}\\&
\lesssim (\frac{\lambda^2(\tau)}{\lambda^2(\sigma)})^{-\delta_0-\frac14}\big\|\xi^{\frac12+\delta_0}\mathcal{F}\big(\lambda^{-2}(\sigma)\chi_{R\lesssim\tau}R(e_{\text{prelim}} - \tilde{e}_{\text{prelim}})\big)(\sigma,\cdot)\big\|_{L^2_{d\rho}}\\
&\lesssim  (\frac{\lambda^2(\tau)}{\lambda^2(\sigma)})^{-\delta_0-\frac14}\big\|\lambda^{-2}(\sigma)\chi_{R\lesssim\tau}R(e_{\text{prelim}} - \tilde{e}_{\text{prelim}})\big\|_{H^{1+2\delta_0}_{dR}}.
\end{align*}
Referring to the end of Lemma~\ref{lem:corrections}, as well as the definitions preceding that lemma, for the structure of $e_{\text{prelim}} - \tilde{e}_{\text{prelim}}$, and finally also using the key bound \eqref{eq:gammabounds}, we infer the estimate 
\begin{equation}\label{eq: eprelimdiff}\begin{split}
&\big\|\lambda^{-2}(\sigma)\chi_{R\lesssim\tau}R(e_{\text{prelim}} - \tilde{e}_{\text{prelim}})\big\|_{H^{1+2\delta_0}_{dR}}\\&\lesssim \sigma^{-2}\cdot\sigma^{\frac12(1+\nu^{-1})-k_0-2+}\cdot \log\tau_0\tau_0^{k_0+1 - \frac12(1+\nu^{-1})-}\cdot\sigma^{\frac12}\cdot[\big\|(x_0, x_1)\big\|_{\tilde{S}} + \big|x_{0d}\big|].
\end{split}\end{equation}
Finally integrating this over $\sigma \in [\tau_0, \tau]$, we get 
\begin{align*}
&\big\|\xi^{\frac12+\delta_0}\mathcal{D}_{\tau}\Xi_1(\tau, \xi)\big\|_{L^2_{d\xi}(\xi>1)}\lesssim \tau_0^{-\frac32-}\cdot (\frac{\lambda^2(\tau)}{\lambda^2(\tau_0)})^{-\delta_0-\frac14}\cdot[\big\|(x_0, x_1)\big\|_{\tilde{S}} + \big|x_{0d}\big|].
\end{align*}
In turn inserting this bound into the definition \eqref{eq:squaresumnorm}, we find 
\begin{align*}
\big\|\xi^{\frac12+}\mathcal{D}_{\tau}\Xi_1(\tau, \xi)\big\|_{Sqr(\xi>1)}\lesssim \tau_0^{-\frac32-}\cdot[\big\|(x_0, x_1)\big\|_{\tilde{S}} + \big|x_{0d}\big|], 
\end{align*}
which is indeed better than what we need. 
\\

{\it{The contribution of the expression 
\[
 G(\tau, R): = \sum_{2\leq j\leq 5}{{5}\choose{j}}v^j[u_{\text{prelim}}^{5-j} - \tilde{u}_{\text{prelim}}^{5-j} ]
+  5(-\tilde{u}_{\text{prelim}}^4 + u_{\text{prelim}}^4)v,
\]
}}
where we recall $v, u_{\text{prelim}}, \tilde{u}_{\text{prelim}}$, is described in the last part of the proof of Theorem~\ref{thm:main approximate}. In particular, we have the bound 
\[
\big\|v(\tau, R)\big\|_{H^{1+2\delta_0}_{dR}}\lesssim \tau^{\frac12(2+\nu^{-1})-2k_*}, 
\]
with $k_*$ defined as in Theorem~\ref{thm:main approximate}.
Then setting 
\begin{align*}
&\Xi_2(\tau, \xi): = \int_{\tau_0}^{\tau}U(\tau,\sigma)g(\sigma, \frac{\lambda^2(\tau)}{\lambda^2(\sigma)}\xi)\,d\sigma,\,
\end{align*}
where we set 
\[
g(\tau, \xi) = \mathcal{F}\big(\lambda^{-2}(\tau)\chi_{R\lesssim\tau}R G(\tau, R)\big)(\xi),
\]
we infer by a similar argument as for the preceding contribution the bound 
\begin{align*}
\big\|\xi^{\frac12+\delta_0}\mathcal{D}_{\tau}\Xi_2(\tau, \xi)\big\|_{L^2_{d\xi}(\xi>1)}&\lesssim \int_{\tau_0}^{\tau}\frac{\lambda^2(\tau)}{\lambda^2(\sigma)}\big\|\xi^{\frac12+\delta_0}g(\sigma, \frac{\lambda^2(\tau)}{\lambda^2(\sigma)}\xi)\big\|_{L^2_{d\xi}(\xi>1)}\,d\sigma\\
&\lesssim  \int_{\tau_0}^{\tau}(\frac{\lambda^2(\tau)}{\lambda^2(\sigma)})^{-\delta_0 - \frac14}\big\|\lambda^{-2}(\sigma)R G(\sigma, R)\big\|_{H^{1+2\delta_0}_{dR}(R\lesssim\sigma)}\,d\sigma.
\end{align*}
On the other hand, the proof of Theorem~\ref{thm:main approximate} easily implies the crude bound 
\[
\big\|\lambda^{-2}(\sigma)R G(\sigma, R)\big\|_{H^{1+2\delta_0}_{dR}(R\lesssim\sigma)}\lesssim \sigma^{-N}[\big\|(x_0, x_1)\big\|_{\tilde{S}} + \big|x_{0d}\big|],\,N\gg 1,\,
\]
and so we obtain 
\[
\big\|\xi^{\frac12+\delta_0}\mathcal{D}_{\tau}\Xi_2(\tau, \xi)\big\|_{L^2_{d\xi}(\xi>1)}\lesssim \tau_0^{-N}(\frac{\lambda^2(\tau_0)}{\lambda^2(\tau)})^{\frac14 + \delta_0}[\big\|(x_0, x_1)\big\|_{\tilde{S}} + \big|x_{0d}\big|]. 
\]
This in turn furnishes the bound 
\begin{align*}
\big\|\xi^{\frac12+}\mathcal{D}_{\tau}\Xi_2(\tau, \xi)\big\|_{Sqr(\xi>1)}\lesssim \tau_0^{-N}\cdot[\big\|(x_0, x_1)\big\|_{\tilde{S}} + \big|x_{0d}\big|], 
\end{align*}
which is much better than what we need. 
\\

{\bf{Step 2}}: {\it{Choice of the corrections $(\triangle \tilde{\tilde{x}}_0^{(0)}, \triangle \tilde{\tilde{x}}_1^{(0)})$}}. In analogy to \cite{CondBlow}, we shall pick these corrections in the specific form 
\begin{equation}\label{eq:alphabeta}
\triangle \tilde{\tilde{x}}_0^{(0)}(\xi) = \alpha\mathcal{F}\big(\chi_{R\leq C\tau}\phi(R, 0)\big)(\xi),\,\triangle \tilde{\tilde{x}}_1^{(0)}(\xi) = \beta\mathcal{F}\big(\chi_{R\leq C\tau}\phi(R, 0)\big)(\xi),
\end{equation}
and we need to determine the parameters $\alpha, \beta$ in order to force the required vanishing conditions \eqref{eq:vanishing11}, \eqref{eq:vanishing12} for $\triangle\widetilde{x^{(0)}}_0(\xi), \triangle\widetilde{x^{(0)}}_1(\xi)$. The latter quantities are given by
\begin{align*}
&\triangle\widetilde{x^{(0)}}_0(\xi) = \int_{\tau_0}^\infty U(\tau_0, \sigma)\mathcal{F}\big(\lambda^{-2}\chi_{R\lesssim\tau}R e_{approx}\big)(\sigma, \frac{\lambda^2(\tau_0)}{\lambda^2(\sigma)}\xi)\,d\sigma + \triangle \tilde{\tilde{x}}_0^{(0)}(\xi)\\
&\triangle\widetilde{x^{(0)}}_1(\xi) = \int_{\tau_0}^\infty V(\tau_0, \sigma)\mathcal{F}\big(\lambda^{-2}\chi_{R\lesssim\tau}R e_{approx}\big)(\sigma, \frac{\lambda^2(\tau_0)}{\lambda^2(\sigma)}\xi)\,d\sigma + \triangle \tilde{\tilde{x}}_1^{(0)}(\xi),\\
\end{align*}
where we recall \eqref{eq:Dtaueffect}. Thus writing $\tilde{\triangle}\widetilde{x^{(0)}}_j(\xi): = \triangle\widetilde{x^{(0)}}_j(\xi) -  \triangle \tilde{\tilde{x}}_j^{(0)}(\xi)$, $j = 0, 1$, we need the following simple 
\begin{lem}\label{lem:ABfunctionaleapprox} We have the bounds 
\begin{align*}
\big|\int_0^\infty \frac{\rho^{\frac12}(\xi)\tilde{\triangle}\widetilde{x^{(0)}}_0(\xi)}{\xi^{\frac14}}\cos[\lambda(\tau_0)\xi^{\frac12}\int_{\tau_0}^{\infty}\lambda^{-1}(u)\,du]\,d\xi\big|\lesssim \tau_0^{-(1-)}[\big\|(x_0,x_1)\big\|_{\tilde{S}}+|x_{0d}|],
\end{align*}
\begin{align*}
\big|\int_0^\infty \frac{\rho^{\frac12}(\xi)\tilde{\triangle}\widetilde{x^{(0)}}_1(\xi)}{\xi^{\frac34}}\sin[\lambda(\tau_0)\xi^{\frac12}\int_{\tau_0}^{\infty}\lambda^{-1}(u)\,du]\,d\xi\big|\lesssim \tau_0^{-(1-)}[\big\|(x_0,x_1)\big\|_{\tilde{S}}+|x_{0d}|],
\end{align*}
\end{lem}
\begin{proof} We again refer to \eqref{eq:eapprox} to split this into a number of bounds. We consider here the contribution of 
\[
 e_{\text{prelim}} - \tilde{e}_{\text{prelim}},
 \]
 the remaining terms being treated similarly. We distinguish between three frequency regimes: 
 \\
 {\it{(i): $\xi<1$.}} Here we get 
 \begin{align*}
 &\big|\tilde{\triangle}\widetilde{x^{(0)}}_0(\xi)\big|\\&\lesssim \xi^{-\frac12+}\tau_0^{0+}\int_{\tau_0}^\infty \frac{\lambda(\tau_0)}{\lambda(\sigma)}\big|\mathcal{F}\big(\lambda^{-2}(\sigma)\chi_{R\lesssim\tau}R(e_{\text{prelim}} - \tilde{e}_{\text{prelim}})\big)(\sigma, \frac{\lambda^2(\tau)}{\lambda^2(\sigma)}\xi)\big|\,d\sigma\\
\end{align*}
Referring to Lemma~\ref{lem:corrections}, we have (using the point wise bounds on $\phi(R, \xi)$ in Proposition~\ref{prop:Fourier})
\begin{equation}\label{eq:Linftyfor F}\begin{split}
&\big|\mathcal{F}\big(\lambda^{-2}(\sigma)\chi_{R\lesssim\tau}R(e_{\text{prelim}} - \tilde{e}_{\text{prelim}})\big)(\sigma, \frac{\lambda^2(\tau)}{\lambda^2(\sigma)}\xi)\big|\\
&\lesssim \big\|\lambda^{-2}(\sigma)\chi_{R\lesssim\tau}R(e_{\text{prelim}} - \tilde{e}_{\text{prelim}})\big)\big\|_{L^1_{dR}}\\
&\lesssim \sigma^{-2}\cdot\sigma^{\frac12(1+\nu^{-1})-k_0-2+}\cdot \log\tau_0\tau_0^{k_0+1 - \frac12(1+\nu^{-1})-}\cdot\sigma\cdot[\big\|(x_0, x_1)\big\|_{\tilde{S}} + \big|x_{0d}\big|].
\end{split}\end{equation}
Inserting this in the preceding $\sigma$-integral for $\tilde{\triangle}\widetilde{x^{(0)}}_0(\xi)$, we find 
\begin{align*}
\big|\tilde{\triangle}\widetilde{x^{(0)}}_0(\xi)\big|\lesssim \xi^{-\frac12+}\tau_0^{-1+}\cdot[\big\|(x_0, x_1)\big\|_{\tilde{S}} + \big|x_{0d}\big|].
\end{align*}
In turn recalling the asymptotics for the spectral density $\rho(\xi)$ from Proposition~\ref{prop:Fourier}, we obtain 
 \begin{align*}
&\big|\int_0^1\frac{\rho^{\frac12}(\xi)\tilde{\triangle}\widetilde{x^{(0)}}_0(\xi)}{\xi^{\frac14}}\cos[\lambda(\tau_0)\xi^{\frac12}\int_{\tau_0}^{\infty}\lambda^{-1}(u)\,du]\,d\xi\big|\\&\lesssim \tau_0^{-(1-)}[\big\|(x_0,x_1)\big\|_{\tilde{S}} + \big|x_{0d}\big|]\cdot\int_0^1\xi^{-(1-)}\,d\xi\\
&\lesssim \tau_0^{-(1-)}[\big\|(x_0,x_1)\big\|_{\tilde{S}} + \big|x_{0d}\big|].
 \end{align*}
 
  {\it{(ii): $1\leq \xi<\frac{\lambda^2(\sigma)}{\lambda^2(\tau_0)}$.}} Denote the contribution to $\tilde{\triangle}\widetilde{x^{(0)}}_0$ under this restriction $\tilde{\triangle}\widetilde{x^{(0)}}_{01}$.
 Again referring to the $\rho$-asymptotics from Proposition~\ref{prop:Fourier} an recalling \eqref{eq:inhompara}, we infer 
 \begin{align*}
&\big|\tilde{\triangle}\widetilde{x^{(0)}}_{01}\big|\\&\lesssim \xi^{-1}\int_{\tau_0}^\infty \chi_{\xi<\frac{\lambda^2(\sigma)}{\lambda^2(\tau_0)}}\frac{\lambda(\tau_0)}{\lambda(\sigma)}\big|\mathcal{F}\big(\lambda^{-2}(\sigma)\chi_{R\lesssim\tau}R(e_{\text{prelim}} - \tilde{e}_{\text{prelim}})\big)(\sigma, \frac{\lambda^2(\tau)}{\lambda^2(\sigma)}\xi)\big|\,d\sigma\\
&\lesssim \xi^{-\frac32}\tau_0^{-(1-)}[\big\|(x_0,x_1)\big\|_{\tilde{S}} + \big|x_{0d}\big|],
 \end{align*}
 where we have used the same asymptotics for $\big|\mathcal{F}\big(\ldots\big)\big|$ as in {\it{(i)}}. In turn, this implies 
  \begin{align*}
 &\big|\int_1^\infty\frac{\rho^{\frac12}(\xi)\tilde{\triangle}\widetilde{x^{(0)}}_{01}(\xi)}{\xi^{\frac14}}\cos[\lambda(\tau_0)\xi^{\frac12}\int_{\tau_0}^{\infty}\lambda^{-1}(u)\,du]\,d\xi\big|\lesssim \tau_0^{-(1-)}[\big\|(x_0,x_1)\big\|_{\tilde{S}} + \big|x_{0d}\big|].
 \end{align*}

{\it{(iii): $\xi > \frac{\lambda^2(\sigma)}{\lambda^2(\tau_0)}$.}} Here we use that for the corresponding contribution to $\tilde{\triangle}\widetilde{x^{(0)}}_0$, which we call $\tilde{\triangle}\widetilde{x^{(0)}}_{02}$, we have 
\begin{align*}
&\big\|\xi\tilde{\triangle}\widetilde{x^{(0)}}_{02}(\xi)\big\|_{L^2_{d\xi}}\\&\lesssim \int_{\tau_0}^\infty\big\|\xi^{\frac12}\frac{\lambda^2(\tau_0)}{\lambda^2(\sigma)}\mathcal{F}\big(\lambda^{-2}(\sigma)\chi_{R\lesssim\tau}R(e_{\text{prelim}} - \tilde{e}_{\text{prelim}})\big)(\sigma, \frac{\lambda^2(\tau)}{\lambda^2(\sigma)}\xi)\big\|_{L^2_{d\xi}(\xi > \frac{\lambda^2(\sigma)}{\lambda^2(\tau_0)})}\,d\sigma\\
&\lesssim  \int_{\tau_0}^\infty\big\|\xi^{\frac12}\mathcal{F}\big(\lambda^{-2}(\sigma)\chi_{R\lesssim\tau}R(e_{\text{prelim}} - \tilde{e}_{\text{prelim}})\big)(\sigma,\cdot)\big\|_{L^2_{d\rho}}\,d\sigma\\
&\lesssim  \int_{\tau_0}^\infty\big\|\big(\lambda^{-2}(\sigma)\chi_{R\lesssim\tau}R(e_{\text{prelim}} - \tilde{e}_{\text{prelim}})\big)(\sigma,\cdot)\big\|_{H^1_{dR}}\,d\sigma\\
&\lesssim \tau_0^{-\frac32}[\big\|(x_0,x_1)\big\|_{\tilde{S}} + \big|x_{0d}\big|], 
\end{align*}
where we have used \eqref{eq: eprelimdiff}. We conclude by Cauchy-Schwarz that 
\begin{align*}
&\big|\int_1^\infty\frac{\rho^{\frac12}(\xi)\tilde{\triangle}\widetilde{x^{(0)}}_{02}(\xi)}{\xi^{\frac14}}\cos[\lambda(\tau_0)\xi^{\frac12}\int_{\tau_0}^{\infty}\lambda^{-1}(u)\,du]\,d\xi\big|\\
&\lesssim \big\|\xi\tilde{\triangle}\widetilde{x^{(0)}}_{02}(\xi)\big\|_{L^2_{d\xi}}\lesssim \tau_0^{-\frac32}[\big\|(x_0,x_1)\big\|_{\tilde{S}} + \big|x_{0d}\big|]. 
\end{align*}
The contributions of the remaining terms forming $e_{approx}$ are handled similarly, as is the second estimate of the lemma involving $\tilde{\triangle}\widetilde{x^{(0)}}_{1}$. 
\end{proof}

We next use the same argument as in \eqref{eg:Btildetilde} to infer the asymptotic relations (for $C\gg 1, \tau_0\gg 1$)
\begin{align*}
\big|\int_0^\infty\frac{\rho^{\frac12}(\xi)\mathcal{F}\big(\chi_{R\leq C\tau}\phi(R, 0)\big)(\xi)}{\xi^{\frac14}}\cos[\lambda(\tau_0)\xi^{\frac12}\int_{\tau_0}^{\infty}\lambda^{-1}(u)\,du]\,d\xi\big|\sim 1, 
\end{align*}
\begin{align*}
\big|\int_0^\infty\frac{\rho^{\frac12}(\xi)\mathcal{F}\big(\chi_{R\leq C\tau}\phi(R, 0)\big)(\xi)}{\xi^{\frac34}}\sin[\lambda(\tau_0)\xi^{\frac12}\int_{\tau_0}^{\infty}\lambda^{-1}(u)\,du]\,d\xi\big|\sim \tau_0, 
\end{align*}
The preceding lemma in conjunction with these asymptotics implies that the vanishing relations \eqref{eq:vanishing11}, \eqref{eq:vanishing12} will be satisfied for $\alpha, \beta$ in \eqref{eq:alphabeta} satisfying
\[
\big|\alpha\big|\lesssim \tau_0^{-(1-)}[\big\|(x_0,x_1)\big\|_{\tilde{S}} + \big|x_{0d}\big|],\,\big|\beta\big|\lesssim \tau_0^{-(2-)}[\big\|(x_0,x_1)\big\|_{\tilde{S}} + \big|x_{0d}\big|].
\]
Then {\bf{Step 2}} is concluded by observing the bounds \eqref{eq:diraclike}, \eqref{eq:spikeS1}, as well as the analogous bound (recalling \eqref{eq:Snorm})
\[
\big\|\frac{C\tau_0}{\langle C\tau_0\xi^{\frac12}\rangle^N}\big\|_{\tilde{S}_1} \lesssim \tau_0^{\delta_0},
\]
whence 
\[
\big\|\triangle \tilde{\tilde{x}}_0^{(0)}\big\|_{\tilde{S_1}} + \big\|\triangle \tilde{\tilde{x}}_1^{(0)}\big\|_{\tilde{S_2}}\lesssim \tau_0^{-(2-)}[\big\|(x_0,x_1)\big\|_{\tilde{S}} + \big|x_{0d}\big|].
\]

{\bf{Step 3}}: {\it{Proof of the low frequency bounds}}. Here we control $\triangle_{>\tau}\widetilde{x^{(0)}}(\tau, \xi)$ in the low frequency regime $\xi<1$. The choice of $\triangle\widetilde{x^{(0)}}_0, \triangle\widetilde{x^{(0)}}_1$ at the beginning of {\bf{Step 2}} imply that 
\begin{align*}
&\triangle_{>\tau}\widetilde{x^{(0)}}(\tau, \xi) = 
-  \int_{\tau}^{\infty}U(\tau, \sigma)\cdot\mathcal{F}\big(\lambda^{-2}(\sigma)\chi_{R\lesssim\sigma}R(e_{approx})\big)(\sigma, \frac{\lambda^2(\tau)}{\lambda^2(\sigma)}\xi)\,d\sigma,
\end{align*}
and in light of the asserted bounds of the proposition, we need to control 
\[
(\frac{\tau}{\tau_0})^{-\kappa}\big\|\chi_{\xi<1}\triangle_{>\tau}\widetilde{x^{(0)}}(\tau, \xi)\big\|_{S_1},\,\big\|\xi^{-\delta_0}\mathcal{D}_{\tau}\triangle_{>\tau}\widetilde{x^{(0)}}(\tau, \xi)\big\|_{Sqr(\xi<1)}.
\]
We show here how to bound the first quantity, the second being more of the same. We use that 
\begin{align*}
\big\|\xi^{-\delta_0}U(\tau, \sigma)\big\|_{L^2_{d\xi}(\xi<1)}\lesssim \tau^{\delta_0}\cdot\frac{\lambda(\tau)}{\lambda(\sigma)},
\end{align*}
which then implies
\begin{align*}
&\big\|\xi^{-\delta_0}\triangle_{>\tau}\widetilde{x^{(0)}}(\tau, \xi)\big\|_{L^2_{d\xi}(\xi<1)}\\
&\lesssim  \tau^{\delta_0}\int_{\tau}^{\infty}\frac{\lambda(\tau)}{\lambda(\sigma)}\big\|\mathcal{F}\big(\lambda^{-2}(\sigma)\chi_{R\lesssim\sigma}R(e_{approx})\big)(\sigma, \frac{\lambda^2(\tau)}{\lambda^2(\sigma)}\xi)\big\|_{L^\infty_{d\xi}}\,d\sigma
\end{align*}
Then as usual we distinguish between the different parts of $e_{approx}$. For example, for the contribution of the principal part $e_{\text{prelim}} - \tilde{e}_{\text{prelim}}$, we get by arguing as in {\it{(i)}} of the proof of the preceding lemma
\begin{align*}
&\tau^{0+}\int_{\tau}^{\infty}\frac{\lambda(\tau)}{\lambda(\sigma)}\big\|\mathcal{F}\big(\lambda^{-2}(\sigma)\chi_{R\lesssim\sigma}R(e_{\text{prelim}} - \tilde{e}_{\text{prelim}})\big)(\sigma, \frac{\lambda^2(\tau)}{\lambda^2(\sigma)}\xi)\big\|_{L^\infty_{d\xi}}\,d\sigma\\
&\lesssim \tau^{0+}\cdot\frac{\tau_0^{k_0+1}\log\tau_0}{\lambda_{0,0}^{\frac12}(\tau_0)}\cdot\int_{\tau}^{\infty}\frac{\lambda(\tau)}{\lambda(\sigma)}\sigma^{-k_0-3}\lambda_{0,0}^{\frac12}(\sigma)\,d\sigma\cdot[\big\|(x_0,x_1)\big\|_{\tilde{S}} + \big|x_{0d}\big|]\\
&\lesssim \tau_0^{-(1-)}\cdot[\big\|(x_0,x_1)\big\|_{\tilde{S}} + \big|x_{0d}\big|].
\end{align*}
This is even better than what we need, since we have omitted the weight $(\frac{\tau}{\tau_0})^{-\kappa}$. The remaining terms in $e_{approx}$ lead to similar contributions.
\\
{\bf{Step 4}}: {\it{Control over the data $(\triangle\widetilde{x^{(0)}}_0(\xi), \triangle\widetilde{x^{(0)}}_1(\xi))$ for the free part in the low frequency regime.}} In light of the low frequency bound established in the preceding step, it suffices to establish the high-frequency bound, i. e. restrict to $\xi>1$. Thus in light of \eqref{eq:Stildenorm}  we need to bound 
\[
\big\|\xi^{1+\delta_0}\triangle\widetilde{x^{(0)}}_0(\xi)\big\|_{L^2_{d\xi}(\xi>1)} + \big\|\xi^{\frac12+\delta_0}\triangle\widetilde{x^{(0)}}_1(\xi)\big\|_{L^2_{d\xi}(\xi>1)}.
\]
where $(\triangle\widetilde{x^{(0)}}_0(\xi), \triangle\widetilde{x^{(0)}}_1(\xi))$ are defined as at the beginning of {\bf{Step 2}}. We shall establish the desired estimate for $\triangle\widetilde{x^{(0)}}_0$ and the contribution of $e_{\text{prelim}} - \tilde{e}_{\text{prelim}}$, the remaining contributions as well as the term $\triangle\widetilde{x^{(0)}}_1$ being more of the same. Note that on account of the final bound of {\bf{Step 2}}, the correction terms $\triangle \tilde{\tilde{x}}_0^{(0)}, \triangle \tilde{\tilde{x}}_1^{(0)}$ satisfy the required bounds. The norm $\big\|\xi^{1+}\triangle\widetilde{x^{(0)}}_0(\xi)\big\|_{L^2_{d\xi}(\xi>1)}$ can be bounded by 
\begin{equation}\label{eq:againtwoterms}\begin{split}
&\big\| \xi^{1+\delta_0}\int_{\tau_0}^{\infty}U(\tau, \sigma)\cdot\mathcal{F}\big(\lambda^{-2}(\sigma)\chi_{R\lesssim\sigma}R(e_{\text{prelim}} - \tilde{e}_{\text{prelim}})\big)(\sigma, \frac{\lambda^2(\tau_0)}{\lambda^2(\sigma)}\xi)\,d\sigma\big\|_{L^2_{d\xi}(\xi>1)}\\
&\lesssim \int_{\tau_0}^{\infty}\frac{\lambda(\tau_0)}{\lambda(\sigma)}\big\|\xi^{\delta_0}\mathcal{F}\big(\lambda^{-2}(\sigma)\chi_{R\lesssim\sigma}R(e_{\text{prelim}} - \tilde{e}_{\text{prelim}})\big)(\sigma, \frac{\lambda^2(\tau_0)}{\lambda^2(\sigma)}\xi)\big\|_{L^2_{d\xi}(1<\xi<\frac{\lambda^2(\sigma)}{\lambda^2(\tau_0)})}\,d\sigma\\
& + \int_{\tau_0}^{\infty}\frac{\lambda^2(\tau_0)}{\lambda^2(\sigma)}\big\|\xi^{\frac12+\delta_0}\mathcal{F}\big(\lambda^{-2}(\sigma)\chi_{R\lesssim\sigma}R(e_{\text{prelim}} - \tilde{e}_{\text{prelim}})\big)(\sigma, \frac{\lambda^2(\tau_0)}{\lambda^2(\sigma)}\xi)\big\|_{L^2_{d\xi}(\xi>\frac{\lambda^2(\sigma)}{\lambda^2(\tau_0)})}\,d\sigma.\\
\end{split}\end{equation}
Then recalling parameter $\kappa = 2(1+\nu^{-1})\delta_0$ , the first term on the right (intermediate frequencies) is bounded by 
\begin{align*}
& \int_{\tau_0}^{\infty}\frac{\lambda(\tau_0)}{\lambda(\sigma)}\big\|\xi^{\delta_0}\mathcal{F}\big(\lambda^{-2}(\sigma)\chi_{R\lesssim\sigma}R(e_{\text{prelim}} - \tilde{e}_{\text{prelim}})\big)(\sigma, \frac{\lambda^2(\tau_0)}{\lambda^2(\sigma)}\xi)\big\|_{L^2_{d\xi}(1<\xi<\frac{\lambda^2(\sigma)}{\lambda^2(\tau_0)})}\,d\sigma\\
&\lesssim \int_{\tau_0}^{\infty}(\frac{\sigma}{\tau_0})^{\kappa}\big\|\mathcal{F}\big(\lambda^{-2}(\sigma)\chi_{R\lesssim\sigma}R(e_{\text{prelim}} - \tilde{e}_{\text{prelim}})\big)(\sigma, \frac{\lambda^2(\tau_0)}{\lambda^2(\sigma)}\xi)\big\|_{L^\infty_{d\xi}}\,d\sigma
\end{align*}
and further recalling \eqref{eq:Linftyfor F}, this is bounded by 
\begin{align*}
&\lesssim \log\tau_0\cdot\tau_0^{k_0+1-}\cdot[\big\|(x_0,x_1)\big\|_{\tilde{S}} + \big|x_{0d}\big|]\cdot\int_{\tau_0}^\infty(\frac{\sigma}{\tau_0})^{\kappa}\cdot\frac{\lambda_{0,0}^{\frac12}(\sigma)}{\lambda_{0,0}^{\frac12}(\tau_0)}\cdot\sigma^{-k_0 - 3+}\,d\sigma\\
&\lesssim \tau_0^{-(1-)}[\big\|(x_0,x_1)\big\|_{\tilde{S}} + \big|x_{0d}\big|]
\end{align*}
The second term on the right of \eqref{eq:againtwoterms} (large frequencies) is bounded by 
\begin{align*}
&\int_{\tau_0}^{\infty}\frac{\lambda^2(\tau_0)}{\lambda^2(\sigma)}\big\|\xi^{\frac12+\delta_0}\mathcal{F}\big(\lambda^{-2}(\sigma)\chi_{R\lesssim\sigma}R(e_{\text{prelim}} - \tilde{e}_{\text{prelim}})\big)(\sigma, \frac{\lambda^2(\tau_0)}{\lambda^2(\sigma)}\xi)\big\|_{L^2_{d\xi}(\xi>\frac{\lambda^2(\sigma)}{\lambda^2(\tau_0)})}\,d\sigma\\
&\lesssim \int_{\tau_0}^{\infty}(\frac{\sigma}{\tau_0})^{\kappa}\big\|\lambda^{-2}(\sigma)\chi_{R\lesssim\sigma}R(e_{\text{prelim}} - \tilde{e}_{\text{prelim}})(\sigma, \cdot)\big\|_{H^{1+2\delta_0}_{dR}}\,d\sigma\\
&\lesssim \tau_0^{-(\frac32-)}[\big\|(x_0,x_1)\big\|_{\tilde{S}} + \big|x_{0d}\big|], 
\end{align*}
where we have taken advantage of \eqref{eq: eprelimdiff}.
\\

{\bf{Step 5}}: {\it{Lipschitz continuity of the corrections $(\triangle \tilde{\tilde{x}}_0^{(0)}(\xi),  \triangle \tilde{\tilde{x}}_1^{(0)}(\xi))$ with respect to the original perturbations $(x_0, x_1, x_{0d})$.}} Here we prove the final assertion of the proposition. 
We note that on account of our construction of $(\triangle \tilde{\tilde{x}}_0^{(0)}(\xi),  \triangle \tilde{\tilde{x}}_1^{(0)}(\xi))$ in {\bf{Step 2}}, their dependence on $(x_0, x_1)$ comes solely through the coefficients $\alpha, \beta$.
We consider the first of these, the second being treated similarly. Then recall that we have 
\[
\alpha = \frac{c_1(\gamma_1,\gamma_2)}{c_2(\gamma_1,\gamma_2)},
\]
where we have introduced the functions 
\[
c_1(\gamma_1,\gamma_2) = -\int_0^\infty \frac{\rho^{\frac12}(\xi)\tilde{\triangle}\widetilde{x^{(0)}}_0(\xi)}{\xi^{\frac14}}\cos[\lambda(\tau_0)\xi^{\frac12}\int_{\tau_0}^{\infty}\lambda^{-1}(u)\,du]\,d\xi,
\]
\[
c_2(\gamma_1,\gamma_2) = \int_0^\infty \frac{\rho^{\frac12}(\xi)\tilde{\triangle}\mathcal{F}\big(\chi_{R\leq C\tau}\phi(R, 0)\big)(\xi)}{\xi^{\frac14}}\cos[\lambda(\tau_0)\xi^{\frac12}\int_{\tau_0}^{\infty}\lambda^{-1}(u)\,du]\,d\xi,
\]
and we also recall the notation, introduced shortly before Lemma~\ref{lem:ABfunctionaleapprox}
\begin{equation}\label{eq:tildetilde}
\tilde{\triangle}\widetilde{x^{(0)}}_0(\xi) = \int_{\tau_0}^\infty U(\tau_0, \sigma)\mathcal{F}\big(\lambda^{-2}\chi_{R\lesssim\sigma}R e_{approx}\big)(\sigma, \frac{\lambda^2(\tau_0)}{\lambda^2(\sigma)}\xi)\,d\sigma.
\end{equation}
Observe that there is dependence on $\gamma_{1,2}$ via $\lambda(\tau) = \lambda_{\gamma_1,\gamma_2}(\tau)$, $\tau_0 = \int_{t_0}^\infty \lambda(s)\,ds$, as well as 
\[
e_{approx} = e_{approx}(\tau_{0,0}, R_{0,0}, \gamma_{1,2}),
\]
with $R_{0,0}$ defined as in \eqref{eq:restlambda}, and $\tau_{0,0} = \int_{t}^\infty s^{-1-\nu}\,ds$, and we interpret $R_{0,0}$ as a function of $\tau, R,\gamma_{1,2}$, and $\tau_{0,0}$ as a function of $\tau, \gamma_{1,2}$. Then writing 
\[
\tilde{e}_{approx}(\tau, R, \gamma_{1,2}): = e_{approx}(\tau_{0,0}(\tau, \gamma_{1,2}), R_{0,0}(R,\tau, \gamma_{1,2}), \gamma_{1,2}),
\]
one derives after some algebraic manipulations a relation of the form 
\begin{equation}\label{eq:tildeegammapartial}
\partial_{\gamma_j}\tilde{e}_{approx} = A_j(\tau, \gamma_{1,2})R\partial_R\tilde{e}_{approx} + B_j(\tau, \gamma_{1,2})\tau\partial_{\tau}\tilde{e}_{approx} + \partial_{\gamma_j}e_{approx},
\end{equation}
where the coefficients are given in terms of 
\begin{equation}\label{eq:coeffs}\begin{split}
&A_j(\tau, \gamma_{1,2}) = \frac{\lambda}{\lambda_{0,0}}\partial_{\gamma_j}\big(\frac{\lambda_{0,0}}{\lambda}\big) - \partial_{\gamma_j}\tau_{0,0}(\partial_{\tau}\tau_{0,0})^{-1}\cdot \frac{\lambda}{\lambda_{0,0}}\partial_{\tau}\big(\frac{\lambda_{0,0}}{\lambda}\big),\\
&B_j(\tau, \gamma_{1,2}) = \tau^{-1}\cdot \partial_{\gamma_j}\tau_{0,0}(\partial_{\tau}\tau_{0,0})^{-1}.
\end{split}\end{equation}
In light of the definition \eqref{eq:lambdagammaonetwo} as well as \eqref{eq:gammabounds}, we infer the bounds 
\begin{align*}
&\big|A_j(\tau, \gamma_{1,2})\big|\lesssim \tau^{-k_0}\log\tau + \tau^{-k_0}\log\tau\cdot O_{\tau_0}\big(\big\|(x_0,x_1)\big\|_{\tilde{S}} + \big|x_{0d}\big|\big),\\
&\big|B_j(\tau, \gamma_{1,2})\big|\lesssim \tau^{-k_0}\log\tau.
\end{align*}
As for the integration kernel $U(\tau_0, \sigma)$, recalling \eqref{eq:inhompara}, we find 
\begin{align*}
&\big|\partial_{\gamma_j}U(\tau_0, \sigma)\big|\\&\lesssim \log\tau_0\tau_0^{-k_0}\frac{\lambda^{\frac32}(\tau_0)}{\lambda^{\frac32}(\sigma)}\frac{\rho^{\frac12}(\frac{\lambda^2(\tau_0)}{\lambda^2(\sigma)}\xi)}{\rho^{\frac12}(\xi)}\big(\tau_0+\big|\frac{\sin[\lambda(\tau_0)\xi^{\frac12}\int_{\tau_0}^{\sigma}\lambda^{-1}(u)\,du]}{\xi^{\frac12}}\big|\big)
\end{align*}
Finally, we can bound $\partial_{\gamma_j}c_1(\gamma_1,\gamma_2)$. Observe the crude bounds 
\begin{align*}
&\big|\tilde{\triangle}\widetilde{x^{(0)}}_0(\xi)\big|\lesssim \xi^{-(\frac12-)}\tau_0^{0+}\cdot \tau_0^{-2}\big(\big\|(x_0,x_1)\big\|_{\tilde{S}} + \big|x_{0d}\big|\big),\,\xi<1,\\
&\big|\tilde{\triangle}\widetilde{x^{(0)}}_0(\xi)\big|\lesssim \xi^{-(\frac32+)}\cdot \tau_0^{-2}\big(\big\|(x_0,x_1)\big\|_{\tilde{S}} + \big|x_{0d}\big|\big),\,\xi\geq 1,\\ 
\end{align*}
which follow from \eqref{eq:tildetilde}, \eqref{eq:inhompara}, as well as Theorem~\ref{thm:main approximate} and the bound \eqref{eq:gammabounds}. Again taking advantage of the $\rho$-asymptotics from Proposition~\ref{prop:Fourier}, we infer 
\begin{equation}\label{eq:partialgammac1first}\begin{split}
&\big|\int_0^\infty \frac{\rho^{\frac12}(\xi)\tilde{\triangle}\widetilde{x^{(0)}}_0(\xi)}{\xi^{\frac14}}\partial_{\gamma_j}\big(\cos[\lambda(\tau_0)\xi^{\frac12}\int_{\tau_0}^{\infty}\lambda^{-1}(u)\,du]\big)\,d\xi\big|\\&\lesssim \tau_0^{-(1-)}\big(\big\|(x_0,x_1)\big\|_{\tilde{S}} + \big|x_{0d}\big|\big).
\end{split}\end{equation}
The preceding point wise bound for $\partial_{\gamma_j}U(\tau_0, \sigma)$ easily reveals that a similar bound is obtained when $\tilde{\triangle}\widetilde{x^{(0)}}_0(\xi)$ in the preceding is replaced by 
\[
\int_{\tau_0}^\infty \partial_{\gamma_j}U(\tau_0, \sigma)\mathcal{F}\big(\lambda^{-2}\chi_{R\lesssim\sigma}R e_{approx}\big)(\sigma, \frac{\lambda^2(\tau_0)}{\lambda^2(\sigma)}\xi)\,d\sigma.
\]
It then remains to consider the case when the operator $\partial_{\gamma_j}$ falls on the Fourier transform 
\[
\mathcal{F}\big(\lambda^{-2}\chi_{R\lesssim\sigma}R e_{approx}\big)(\sigma, \frac{\lambda^2(\tau_0)}{\lambda^2(\sigma)}\xi)
\]
in \eqref{eq:tildetilde}, which we handle schematically as follows. Note that when $\partial_{\gamma_j}$ falls on $\lambda(\tau)$, we obtain a function bounded by a $\lesssim \tau^{-k_0}\log\tau\cdot\lambda(\tau)$, in light of \eqref{eq:lambdagammaonetwo}. Further, recall \eqref{eq:tildeegammapartial} as well as \eqref{eq:coeffs} and the bounds following it, as well as Theorem~\ref{thm:transferenceop} which gives a translation of $R\partial_R$ to the Fourier side. In all, we infer a schematic relation of the form 
\begin{align*}
\partial_{\gamma_1}\big[\mathcal{F}\big(\lambda^{-2}(\sigma)\chi_{R\lesssim\sigma}Re_{approx}\big)(\sigma, \frac{\lambda^2(\tau_0)}{\lambda^2(\sigma)}\xi)\big] = \sum_{j=1}^5 A_j, 
\end{align*}
with the following terms on the right: writing $G(\sigma, R) = \lambda^{-2}(\sigma)\chi_{R\lesssim\sigma}Re_{approx}$, 
\begin{align*}
&A_1 = \sigma^{-k_0}\lambda(\sigma)\mathcal{F}(G)(\sigma, \frac{\lambda^2(\tau_0)}{\lambda^2(\sigma)}\xi),\,A_2 =  \sigma^{-k_0}\lambda(\sigma)(\xi\partial_{\xi})[\mathcal{F}(G)(\sigma, \frac{\lambda^2(\tau_0)}{\lambda^2(\sigma)}\xi)]\\
&A_3 = \sigma^{-k_0}\lambda(\sigma)[\mathcal{K}\mathcal{F}(G)](\sigma, \frac{\lambda^2(\tau_0)}{\lambda^2(\sigma)}\xi),\,A_4 = \sigma^{-k_0}\lambda(\sigma)[(\sigma\partial_{\sigma})\mathcal{F}(G)](\sigma, \frac{\lambda^2(\tau_0)}{\lambda^2(\sigma)}\xi)\\
&A_5 =  \mathcal{F}\big(\lambda^{-2}(\sigma)\chi_{R\lesssim\sigma}R\partial_{\gamma_1}e_{approx}\big),
\end{align*}
with a similar relation for $\partial_{\gamma_2}$ but with $\sigma^{-k_0}$ replaced by $\sigma^{-k_0}\log\sigma$. 
\\
But then performing integration by parts with respect to $\xi$ or $\sigma$ as needed, and recalling the point wise bounds on $\tilde{\triangle}\widetilde{x^{(0)}}_0(\xi)$, we infer 
\begin{equation}\label{eq:delicatebound1}
\big|\int_{\tau_0}^\infty U(\tau_0, \sigma)\big(\sum_{j=1}^4 A_j\big)\,d\sigma\big|\lesssim_{\tau_0}\big\|(x_0,x_1)\big\|_{\tilde{S}} + \big|x_{0d}\big|
\end{equation}
Finally also recalling the structure of $e_{approx}$ from Theorem~\ref{thm:main approximate}, we get the bound 
\begin{equation}\label{eq:delicatebound2}
\big|\int_{\tau_0}^\infty U(\tau_0, \sigma)A_5\big|\lesssim \tau_0^{-(k_0+2-)}\lambda_{0,0}^{\frac12}(\tau_0)
\end{equation}
It is this last term which is dominant, of course. Combining \eqref{eq:partialgammac1first} and the remark following it with \eqref{eq:delicatebound1}, \eqref{eq:delicatebound2}, we finally obtain the bound (for $j = 1,2$)
\begin{equation}\label{eq:partialgammacone}
\big|\partial_{\gamma_j}c_1\big|\lesssim \tau_0^{-(k_0+2-)}\lambda_{0,0}^{\frac12}(\tau_0) + O_{t_0}\big(\big\|(x_0,x_1)\big\|_{\tilde{S}} + \big|x_{0d}\big|\big)
\end{equation}
A simple variation of the preceding arguments also implies the much easier bound 
\begin{equation}\label{eq:partialgammactwo}
\big|\partial_{\gamma_j}c_2\big|\lesssim \tau_0^{-(k_0-1)}\log\tau_0\cdot\lambda_{0,0}(\tau_0). 
\end{equation}
Combining \eqref{eq:partialgammacone}, \eqref{eq:partialgammactwo}, and also recalling $c_2\sim 1$ from the end of {\bf{Step 2}}, we finally obtain the desired estimate 
\begin{equation}\label{eq:partialgammaalphabound}
\big|\partial_{\gamma_j}\alpha\big|\lesssim \tau_0^{-(k_0+2-)}\lambda_{0,0}^{\frac12}(\tau_0) + O_{t_0}\big(\big\|(x_0,x_1)\big\|_{\tilde{S}} + \big|x_{0d}\big|\big).
\end{equation}
Finally, comparing the corrections 
\[
\triangle \tilde{\tilde{x}}_0^{(0)}(\xi) = \alpha\mathcal{F}\big(\chi_{R\leq C\tau}\phi(R, 0)\big)(\xi),\,\triangle \tilde{\tilde{\bar{x}}}_0^{(0)}(\xi) = \bar{\alpha}\mathcal{F}\big(\chi_{R\leq C\tau}\phi(R, 0)\big)(\xi),
\]
corresponding to different data quadruples $\underline{x}_j, \underline{\bar{x}}_j$, we find 
\begin{align*}
\big\|\triangle \tilde{\tilde{x}}_0^{(0)} - \triangle \tilde{\tilde{\bar{x}}}_0^{(0)}\big\|_{\tilde{S}_1}\lesssim &\big|\alpha - \bar{\alpha}\big|\cdot \tau_0^{\delta_0}\\
&\lesssim\big(\sum_j \big|\partial_{\gamma_j}\alpha\big|\cdot \big|\gamma_j - \bar{\gamma}_j\big|\big)\cdot \tau_0^{\delta_0},\\
\end{align*}
where we have used a bound at the end of {\bf{Step 2}} for the first inequality, and the preceding can be bounded by 
\begin{align*}
\lesssim \big[\tau_0^{-(1-)} + O_{t_0}\big(\big\|(x_0,x_1)\big\|_{\tilde{S}} + \big|x_{0d}\big|\big)\big]\cdot \big(\big\|(x_0-\bar{x}_0,x_1 - \bar{x}_1)\big\|_{\tilde{S}} + \big|x_{0d} - \bar{x}_{0d}\big|\big),
\end{align*}
in light of \eqref{eq:partialgammaalphabound} as well as Lemma~\ref{lem:Lipcont}. This is the desired Lipschitz dependence of $\triangle \tilde{\tilde{x}}_0^{(0)}(\xi)$ on the data, provided the latter are chosen small enough (depending on $\tau_0$), and that of $\triangle \tilde{\tilde{x}}_1^{(0)}(\xi)$ is similar.  
\\

This concluded the proof of the proposition for the continuous spectral part, and we omit the much simpler routine estimates for the discrete spectral part. 
\end{proof}

\subsubsection{Setup of the iteration scheme; the higher iterates}

We next add a sequence of corrections $\triangle x^{(j)}(\tau, \xi)$ to the zeroth iterate in order to arrive at a solution of \eqref{eq:transport1}, but with data differing slightly from \eqref{eq:xeqndata}. Specifically, we set for the first iterate
\[
\underline{x}^{(1)} = \underline{x}^{(0)} + \underline{\triangle x}^{(1)}
\]
where 
\begin{equation}\label{eq:Deltatransport1}
\big(\mathcal{D}_{\tau}^2 + \beta(\tau)\mathcal{D}_{\tau} + \underline{\xi}\big)\underline{\triangle x}^{(1)}(\tau, \xi) = \calR(\tau, \underline{x}^{(0)}) + \underline{\triangle f}^{(0)}(\tau, \underline{\xi}),
\end{equation}
and we recall \eqref{eq:Rterms} and further use the notation 
\begin{align*}
&\triangle f^{(0)}(\tau, \xi) = \mathcal{F}\big( \lambda^{-2}(\tau)\big[5(u_{approx}^4 - u_0^4)\tilde{\eps}^{(0)} + RN(u_{approx}, \tilde{\eps}^{(0)})\big]\big)\big(\xi\big),\\
&\triangle f^{(0)}_d(\tau) = \langle  \lambda^{-2}(\tau)\big[5(u_{approx}^4 - u_0^4)\tilde{\eps}^{(0)} + RN(u_{approx}, \tilde{\eps}^{(0)}), \phi_d(R)\rangle. 
\end{align*}
and further, we naturally set 
\[
\tilde{\eps}^{(0)}(\tau, R) = x^{(0)}_d(\tau)\phi_d(R) + \int_0^\infty \phi(R, \xi)x^{(0)}(\tau, \xi)\rho(\xi)\,d\xi. 
\]
For the higher corrections $\underline{\triangle x}^{(j)}$, $j\geq 2$, defining the higher iterates, we set correspondingly 
\begin{equation}\label{eq:Deltatransportj}
\big(\mathcal{D}_{\tau}^2 + \beta(\tau)\mathcal{D}_{\tau} + \underline{\xi}\big)\underline{\triangle x}^{(j)}(\tau, \xi) = \calR(\tau, \underline{\triangle x}^{(j-1)}) + \underline{\triangle f}^{(j-1)}(\tau, \xi),
\end{equation}
and we use the definitions 
\begin{align*}
&\triangle f^{(j-1)}(\tau, \xi) = \mathcal{F}\big( \lambda^{-2}(\tau)\chi_{R\lesssim\tau}\big[5(u_{approx}^4 - u_0^4)\triangle\tilde{\eps}^{(j-1)} + RN(u_{approx}, \triangle\tilde{\eps}^{(j-1)})\big]\big)\big(\xi\big),\\
&\triangle f^{(j-1)}_d(\tau) \\&\hspace{1cm} =\int_0^\infty  \lambda^{-2}(\tau)\chi_{R\lesssim\tau}\big[5(u_{approx}^4 - u_0^4)\triangle\tilde{\eps}^{(j-1)} + RN(u_{approx}, \triangle\tilde{\eps}^{(j-1)})\big]\phi_d(R)\,dR
\end{align*}
where we set 
\[
\triangle\tilde{\eps}^{(j-1)}(\tau, R) = \int_0^\infty \phi(R, \xi)\triangle x^{(j-1)}(\tau, \xi)\rho(\xi)\,d\xi + \triangle x^{(j-1)}_d(\tau)\phi_d(R),\,j\geq 2. 
\]
The fact that {\it{upon using suitable initial conditions}} these equations yield in fact iterates which rapidly converge to zero in a suitable sense follows exactly as in \cite{CondBlow}, and so we formulate the corresponding result, which is a summary of Propositions 9. 1 - 9. 6 and most importantly Corollary 12.2, Corollary 12.3 in \cite{CondBlow}:  

\begin{prop}\label{prop:keyiteratebounds}
For each $j\geq 1$, there exists a pair $(\triangle\tilde{\tilde{x}}_0^{(j)}, \triangle\tilde{\tilde{x}}_1^{(j)})\in \tilde{S}$, and such that if we set up the inductive scheme (recall \eqref{eq:inhompara})
\begin{equation}\label{eq:iterativestepcont}\begin{split}
&\triangle x^{(j)}(\tau, \xi) = \\&\int_{\tau_0}^{\tau}U(\tau, \sigma)[\calR(\tau, \underline{\triangle x}^{(j-1)}) + \triangle f^{(j-1)}(\tau, \underline{\xi})](\sigma, \frac{\lambda^2(\tau)}{\lambda^2(\sigma)}\xi)\,d\sigma\\
& + S(\tau)\big(\triangle\tilde{\tilde{x}}_0^{(j)}, \triangle\tilde{\tilde{x}}_1^{(j)}\big)
\end{split}\end{equation}
for the continuous spectral part, while we set (recall \eqref{eq:inhomparad})
\begin{equation}\label{eq:iterativestepdisc}\begin{split}
&\triangle_d x^{(j)}(\tau) = \int_{\tau_0}^\infty H_d(\tau, \sigma)\cdot[\calR_d(\tau, \underline{\triangle x}^{(j-1)}) + \triangle_d f^{(j-1)}(\tau, \underline{\xi})](\sigma)\,d\sigma,
\end{split}\end{equation}
then we obtain control over the iterates in the following precise sense: there is a splitting 
\begin{align*}
\triangle x^{(j)}(\tau, \xi) = \triangle_{>\tau} x^{(j)}(\tau, \xi) + S(\tau)\big(\triangle\tilde{x}_0^{(j)}, \triangle\tilde{x}_1^{(j)})
\end{align*}
in which $\triangle\tilde{x}_0^{(j)}, \triangle\tilde{x}_1^{(j)}$ satisfy the vanishing conditions 
\begin{align}
&\int_0^\infty\frac{\rho^{\frac12}(\xi)\triangle\tilde{x}_0^{(j)}(\xi)}{\xi^{\frac14}}\cos[\lambda(\tau_0)\int_{\tau_0}^\infty\lambda^{-1}(u)\,du]\,d\xi = 0\label{eq:vanishing11}\\
&\int_0^\infty\frac{\rho^{\frac12}(\xi)\triangle\tilde{x}_1^{(j)}(\xi)}{\xi^{\frac34}}\sin[\lambda(\tau_0)\int_{\tau_0}^\infty\lambda^{-1}(u)\,du]\,d\xi = 0\label{eq:vanishing21},
\end{align}
and such that if we set 
\begin{align*}
&\widetilde{\triangle x^{(j)}}(\tau, \xi) = \int_{\tau_0}^{\tau}U(\tau, \sigma)\cdot[\calR(\tau, \underline{\triangle x}^{(j-1)}) + \triangle f^{(j-1)}(\tau, \underline{\xi})](\sigma, \frac{\lambda^2(\tau)}{\lambda^2(\sigma)}\xi)\,d\sigma
\end{align*}
and introduce the quantities (with $\kappa = 2(1+\nu^{-1})\delta_0$)
\begin{equation}\label{eq:measureofiterates}\begin{split}
\triangle A_j:&=
 \sup_{\tau\geq \tau_0}(\frac{\tau_0}{\tau})^{\kappa}\big\|\chi_{\xi>1}\triangle x^{(j)}(\tau, \xi)\big\|_{S_1} + \big\|\xi^{\frac12+\delta_0}\mathcal{D}_{\tau}\widetilde{\triangle x^{(j)}}(\tau, \xi)\big\|_{Sqr(\xi>1)}\\
& + \sup_{\tau\geq \tau_0}(\frac{\tau_0}{\tau})^{\kappa}\big\|\chi_{\xi<1}\triangle_{>\tau}\triangle x^{(j)}(\tau, \xi)\big\|_{S_1} +\big\|\xi^{-\delta_0}\mathcal{D}_{\tau}\triangle_{>\tau}\triangle x^{(j)}(\tau, \xi)\big\|_{Sqr(\xi<1)}\\
& + \big\|(\triangle\tilde{x}_0^{(j)}, \triangle\tilde{x}_1^{(j)})\big\|_{\tilde{S}} + \big\|(\triangle\tilde{\tilde{x}}^{(j)}_0, \triangle\tilde{\tilde{x}}^{(j)}_1)\big\|_{\tilde{S}} + \sup_{\tau\geq\tau_0}\tau^{(1-)}(|\triangle x_d^{(j)}(\tau)| +|\partial_{\tau}\triangle x_d^{(j)}(\tau)|),
\end{split}\end{equation}
where we recall \eqref{eq:squaresumnorm} for the definition of $\big\|\cdot\big\|_{Sqr}$, then we have exponential decay 
\[
\triangle A_j\lesssim_{\delta} \delta^j[\big\|(x_0,x_1)\big\|_{\tilde{S}} + \big|x_{0d}\big|]
\]
for any given $\delta>0$, provided $\tau_0$ is sufficiently large (or equivalently, $t_0$ is sufficiently small). 
In particular, the series 
\[
\underline{x}(\tau, \xi) = \underline{x}^{(0)}(\tau, \xi) + \sum_{j\geq 1}\underline{\triangle x}^{(j)}(\tau, \xi),
\]
converges with 
\begin{align*}
&\sup_{\tau\geq \tau_0}(\frac{\tau_0}{\tau})^{\kappa}\big\|\xi^{1+\delta_0}x(\tau, \xi)\big\|_{L^2_{d\xi}(\xi>1)} +  \sup_{\tau\geq \tau_0}(\frac{\tau_0}{\tau})^{-\kappa}\big\|\xi^{\frac12+\delta_0}\mathcal{D}_{\tau}x(\tau, \xi)\big\|_{L^2_{d\xi}(\xi>1)}\\&\lesssim \big\|(x_0,x_1)\big\|_{\tilde{S}} + \big|x_{0d}\big|.
 \end{align*}
Also, for low frequencies, i. e. $\xi<1$, there is a decomposition
\[
x(\tau, \xi) = x_{>\tau}(\tau, \xi) + S(\tau)\big(\tilde{x}_0, \tilde{x}_1\big)
\]
such that $\tilde{x}_0, \tilde{x}_1$ satisfy the natural analogues of \eqref{eq:vanishing11}, \eqref{eq:vanishing21}, and we have the bounds
\begin{align*}
&\sup_{\tau\geq \tau_0}(\frac{\tau_0}{\tau})^{\kappa}\big\|\xi^{-\delta_0}x(\tau, \xi)\big\|_{L^2_{d\xi}(\xi<1)} +  \sup_{\tau\geq \tau_0}(\frac{\tau_0}{\tau})^{-\kappa}\big\|\xi^{-\delta_0}\mathcal{D}_{\tau}x(\tau, \xi)\big\|_{L^2_{d\xi}(\xi<1)}\\
& +  \big\|(\tilde{x}_0, \tilde{x}_1)\big\|_{\tilde{S}}\lesssim \big\|(x_0,x_1)\big\|_{\tilde{S}} + \big|x_{0d}\big|.
 \end{align*}
 Finally, we also have 
 \[
 \sup_{\tau\geq \tau_0}\tau^{1-}\big|x_d(\tau) - x^{(0)}_d(\tau)\big|\lesssim \big\|(x_0,x_1)\big\|_{\tilde{S}} + \big|x_{0d}\big|.
 \]
 The function 
 \[
 u(\tau, R) = u_{approx}(\tau, R) + \tilde{\epsilon}(\tau, R)
 \]
 with 
 \[
  \tilde{\epsilon}(\tau, R): = x_d(\tau)\phi_d(R) + \int_0^\infty \phi(R, \xi)x(\tau, \xi)\rho(\xi)\,d\xi
  \]
  is then the desired solution of \eqref{eq:tildeepsilon1}, satisfying the properties in terms of its Fourier transform specified in Theorem~\ref{thm:IterationResult}. In fact, we set 
  \[
  \triangle x_\kappa^{(\gamma_1,\gamma_2)} = \sum_{j\geq 1}\triangle\tilde{\tilde{x}}_{\kappa}^{(j)},\, \triangle x_{\kappa d}^{(\gamma_1,\gamma_2)} = \sum_{j\geq 1}\partial_{\tau}^{\kappa}\triangle x_{d}^{(j)}|_{\tau = \tau_0},\,\kappa = 0,1. 
  \]
\end{prop}

In fact, all of the assertions in the preceding long proposition follow exactly from the arguments in \cite{CondBlow} (the only difference being the slightly different scaling law $\lambda(\tau)$), and this will easily establish almost all of Theorem~\ref{thm:IterationResult}, {\it{except}} its last statement concerning  the {\it{Lipschitz continuous dependence}} of the initial data perturbation with respect to the initial perturbation $(x_0, x_1)$. This is a somewhat delicate point which requires a special argument, analogous to the one given for the corresponding assertion in Proposition~\ref{prop:zerothiterate}.  We formulate this as a separate proposition at the level of the iterative corrections: 

\begin{prop}\label{prop:DataLipschitzCont} If $(\triangle\tilde{\tilde{x}}_0^{(j)}, \triangle\tilde{\tilde{x}}_1^{(j)}), (\triangle\tilde{\tilde{\bar{x}}}_0^{(j)}, \triangle\tilde{\tilde{\bar{x}}}_1^{(j)})$, $j\geq 1$, are as in the preceding proposition and with respect to perturbations specified in terms of data quadruples
$(\underline{x}_0,\underline{x}_1)$ respectively $(\underline{\bar{x}}_0, \underline{\bar{x}}_1)\in \tilde{S}$, then for any given $\delta>0$ we have the Lipschitz bound 
\begin{align*}
&\big\|(\triangle\tilde{\tilde{x}}_0^{(j)} - \triangle\tilde{\tilde{\bar{x}}}_0^{(j)}, \triangle\tilde{\tilde{x}}_1^{(j)} - \triangle\tilde{\tilde{\bar{x}}}_1^{(j)})\big\|_{\tilde{S}}\\&\lesssim_{\delta}\tau_0^{-(1-)}\delta^{j}[\big\|(x_0 - \bar{x}_0,x_1 - \bar{x}_1)\big\|_{\tilde{S}} + \big|x_{0d} - \bar{x}_{0d}\big|], 
\end{align*}
provided $\tau_0$ is sufficiently large compared to $\delta$, and 
\[
\big\|(x_0,x_1)\big\|_{\tilde{S}} + \big\|(\bar{x}_0,\bar{x}_1)\big\|_{\tilde{S}} + \big|x_{0d}\big| + \big|\bar{x}_{0d}\big|
\]
is sufficiently small depending on $\tau_0$. 
\end{prop}

To begin the sketch of the proof, we observe from the proofs of Proposition 7.1, 8.1, 9.1 in \cite{CondBlow} that the profiles of the corrections $\triangle\tilde{\tilde{x}}^{(j)}_{\kappa}$, $\kappa= 0,1$, are fixed up to a multiplication parameter, and more precisely we set 
\[
\triangle\tilde{\tilde{x}}^{(j)}_{0} = \alpha^{(j)}\mathcal{F}\big(\chi_{R\leq C\tau_0}\phi(R, 0)\big),\,\triangle\tilde{\tilde{x}}^{(j)}_{1} = \beta^{(j)}\mathcal{F}\big(\chi_{R\leq C\tau_0}\phi(R, 0)\big),
\]
whence the only dependence of the corrections $\triangle\tilde{\tilde{x}}^{(j)}_{\kappa}$ on the data $x_{0,1}$ reside in the coefficients $\alpha^{(j)}, \beta^{(j)}$. The latter, however, depend in a complex manner on the iterative functions $\triangle x^{(j)}$,$\triangle_d x^{(j)}$, and so we cannot get around analysing the (Lipschitz)-dependence of the latter on $x_{0,1}$. This latter task is rendered somewhat cumbersome by the fact that in each iterative step we use a parametrix which re-scales the ingredients (via the factors $\frac{\lambda^2(\tau)}{\lambda^2(\sigma)}$), which depend on $\gamma_{1,2}$ whence on $x_{0,1}$, and so differentiating with respect to $\gamma_j$ will result in a loss of smoothness. What saves things here is the fact that the coefficients $\alpha^{(j)}, \beta^{(j)}$ are given by certain integrals, which are well-behaved with respect to inputs with lesser regularity, as already seen in Step 5 of the proof of Proposition~\ref{prop:zerothiterate}: there differentiating the term $\mathcal{F}\big(\lambda^{-2}(\sigma)Re_{approx}\big)(\sigma, \frac{\lambda^2(\tau_0)}{\lambda^2(\sigma)}\xi)$ with respect to $\gamma_1$ results in a term (see the term $A_2$ in the list of terms preceding \eqref{eq:delicatebound1})
\[
  \tau_0^{-k_0}(\xi\partial_{\xi})\big[\mathcal{F}\big(\lambda^{-2}(\sigma)Re_{approx}\big)(\sigma, \frac{\lambda^2(\tau_0)}{\lambda^2(\sigma)}\xi)\big]
  \]
  which is of lesser regularity with respect to $\xi$, but the corresponding contribution to $\partial_{\gamma_j}\tilde{\triangle}\widetilde{x^{(0)}}_0(\xi)$ and thence to the integral 
  \[
  \int_0^\infty \frac{\rho^{\frac12}(\xi)\partial_{\gamma_j}\tilde{\triangle}\widetilde{x^{(0)}}_0(\xi)}{\xi^{\frac14}}\cos[\lambda(\tau_0)\xi^{\frac12}\int_{\tau_0}^{\infty}\lambda^{-1}(u)\,du]\,d\xi
  \]
  is then handled by integration by parts with respect to $\xi$. 
\\
The exact same type of observation applies to the higher order corrections $\triangle x^{(j)}(\tau, \xi)$ as well. 
\\

To render this intuition precise, we first need to exhibit a functional framework which will be preserved by the iterative steps and which adequately describes the $\gamma_j$ differentiated corrections $\triangle x^{(j)}$. To begin with, we introduce two types of norms:

\begin{defn}\label{defn:strongandweakbounds} Call a pair of functions $(\triangle y(\tau, \xi), \triangle y_d(\tau))$ {\it{strongly bounded}}, provided there exist $(\triangle \tilde{\tilde{y}}_0(\xi), \triangle \tilde{\tilde{y}}_1(\xi))\in \tilde{S}$, as well as $(\triangle \tilde{y}_0(\xi), \triangle \tilde{y}_1(\xi))\in \tilde{S}$, {\it{the latter satisfying the vanishing conditions}} \eqref{eq:vanishing11}, \eqref{eq:vanishing21}, such that if we set 
\begin{align*}
&\triangle y(\tau, \xi) = \triangle_{>\tau} y(\tau, \xi) + S(\tau)\big(\triangle \tilde{y}_0(\xi), \triangle \tilde{y}_1(\xi)\big),\\&\triangle y(\tau, \xi) = \widetilde{\triangle y}(\tau, \xi) + S(\tau)\big(\triangle \tilde{\tilde{y}}_0(\xi), \triangle \tilde{\tilde{y}}_1(\xi)\big)
\end{align*}
then we have (recall \eqref{eq:squaresumnorm})
\begin{align*}
+\infty&>\big\|(\triangle y(\tau, \xi), \triangle y_d(\tau))\big\|_{S_{strong}}\\&:=\sup_{\tau\geq \tau_0}(\frac{\tau_0}{\tau})^{\kappa}\big\|\chi_{\xi>1}\triangle y(\tau, \xi)\big\|_{S_1} +\big\|\xi^{\frac12+\delta_0}\mathcal{D}_{\tau}\widetilde{\triangle y}(\tau, \xi)\big\|_{Sqr(\xi>1)}\\
& + \sup_{\tau\geq \tau_0}(\frac{\tau_0}{\tau})^{\kappa}\big\|\chi_{\xi<1}\triangle_{>\tau}\triangle y(\tau, \xi)\big\|_{S_1} + \big\|\xi^{-\delta_0}\mathcal{D}_{\tau}\triangle_{>\tau}\triangle y(\tau, \xi)\big\|_{Sqr(\xi<1)}\\
& + \big\|(\triangle\tilde{y}_0, \triangle\tilde{y}_1)\big\|_{\tilde{S}} + \big\|(\triangle\tilde{\tilde{y}}_0, \triangle\tilde{\tilde{y}}_1)\big\|_{\tilde{S}} + \sup_{\tau\geq\tau_0}\tau^{(1-)}(|\triangle y_d(\tau)| + |\partial_{\tau}\triangle y_d(\tau)|).
\end{align*}
We call a pair of functions $(\triangle z(\tau, \xi), \triangle z_d(\tau))$ {\it{weakly bounded}}, provided there exist $(\triangle \tilde{\tilde{z}}_0(\xi), \triangle \tilde{\tilde{z}}_1(\xi))\in \tilde{S}$ as well as $(\triangle \tilde{z}_0(\xi), \triangle \tilde{z}_1(\xi))\in \tilde{S}$ {\it{not necessarily satisfying any vanishing conditions}}, such that if we set 
\begin{align*}
&\triangle z(\tau, \xi) = \triangle_{>\tau} z(\tau, \xi) + S(\tau)\big(\triangle \tilde{z}_0(\xi), \triangle \tilde{z}_1(\xi)\big),\\&\triangle z(\tau, \xi) = \widetilde{\triangle z}(\tau, \xi) + S(\tau)\big(\triangle \tilde{\tilde{z}}_0(\xi), \triangle \tilde{\tilde{z}}_1(\xi)\big)
\end{align*}
then we have 
\begin{align*}
&+\infty>\big\|(\triangle z(\tau, \xi), \triangle z_d(\tau))\big\|_{S_{weak}}:=\\
&\tau_0^{-1}\big[\sup_{\tau\geq \tau_0}(\frac{\lambda(\tau_0)}{\lambda(\tau)})^{2\delta_0 + 1}\big\|\chi_{\xi>1}\triangle z(\tau, \xi)\big\|_{\langle\xi\rangle^{-\frac12-}L^2_{d\xi}} + \big\|\xi^{\delta_0}\mathcal{D}_{\tau}\widetilde{\triangle z}(\tau, \xi)\big\|_{\tilde{Sqr}(\xi>1)}]\\
& + \tau_0^{-1}\big[\sup_{\tau\geq \tau_0}((\frac{\lambda(\tau_0)}{\lambda(\tau)})^{2\delta_0+1}\big\|\chi_{\xi<1}\triangle_{>\tau}\triangle z(\tau, \xi)\big\|_{S_1}+ \big\|\xi^{-\delta_0}\mathcal{D}_{\tau}\triangle_{>\tau}\triangle z(\tau, \xi)\big\|_{\tilde{Sqr}(\xi<1)}]\\
& + \big\|(\langle\xi\rangle^{-\frac12}\triangle\tilde{z}_0, \langle\xi\rangle^{-\frac12}\triangle\tilde{z}_1)\big\|_{\tilde{S}} + \big\|(\triangle\tilde{\tilde{z}}_0, \triangle\tilde{\tilde{z}}_1)\big\|_{\tilde{S}}+ \sup_{\tau\geq\tau_0}\frac{\tau\lambda(\tau_0)}{\tau_0\lambda(\tau)}|\langle\partial_{\tau}\rangle\triangle z_d(\tau)| .
\end{align*}
Here the norm $\big\|\cdot\big\|_{\tilde{Sqr}}$ is defined just like in \eqref{eq:squaresumnorm}, except that the power $4\delta_0$ is replaced by $-2 + 4\delta_0$. 
\end{defn}
Observe that by comparison to $\big\|\cdot\big\|_{S_{strong}}$, the norm $\big\|\cdot\big\|_{S_{weak}}$ loses $\xi^{-\frac12}$ in terms of decay for large $\xi$,  and we lose a factor $\tau_0\frac{\lambda(\tau)}{\lambda(\tau_0)}$ in terms of temporal decay.
\\
Using the preceding terminology, we can now introduce the proper norm to measure the expressions arising upon applying $\partial_{\gamma_j}$ to the corrections $\triangle x^{(j)}(\tau, \xi)$. Note that the dependence on the $\gamma_j$ results on the one hand from the parametrices 
\[
U(\tau, \sigma), S(\tau),
\]
 as well as from the expressions $e_{approx}$, $u_{approx}$, $u_0$ and $\lambda(\tau)$ in \eqref{eq:fterms}.  
To emphasise that we want to measure the differences of functions, we introduce the symbol $\triangle\tilde{S}$ for the relevant space:
\begin{defn}\label{defn:roughspace} We define $\triangle\tilde{S}$ as the space of pairs of functions $(\triangle x(\tau, \xi), \triangle x_d(\tau))$ which admit a decomposition 
\[
\triangle x(\tau, \xi) = (\xi\partial_{\xi})\triangle y(\tau, \xi) + \triangle z(\tau, \xi),\,\triangle x_d(\tau) = \triangle y_d(\tau) + \triangle z_d(\tau)
\]
such that $\triangle y$ is strongly bounded and $\triangle z$ is weakly bounded, and we then set 
\[
\big\|(\triangle x(\tau, \xi), \triangle x_d(\tau))\big\|_{\triangle \tilde{S}}: = \inf\big( \big\|(\triangle y(\tau, \xi), \triangle y_d(\tau))\big\|_{S_{strong}} + \big\|(\triangle z(\tau, \xi), \triangle z_d(\tau))\big\|_{S_{weak}}\big)
\]
where the infimum is over all decompositions into differentiated strongly bounded and weakly bounded functions.  
\end{defn}

We use the norm $\big\|\cdot\big\|_{\triangle \tilde{S}}$ to measure the pair quantities $\big(\partial_{\gamma_\kappa}\triangle x^{(j)}(\tau, \xi), \partial_{\gamma_\kappa}\triangle x^{(j)}_d(\tau)\big)$, where $\kappa = 1, 2$. To achieve this for all the corrections, we need an inductive step which infers the required bound for the next iterate, as well as rapid decay of these quantities. Correspondingly we have the following two lemmas: 
\begin{lem}\label{lem:partialgammainductivebound} Provided the $(\triangle x^{(j)}, \triangle x^{(j)}_d)$ are constructed as in Proposition~\ref{prop:keyiteratebounds}, and assuming the bounds there, we have 
\begin{align*}
&\big\|\big(\partial_{\gamma_\kappa}\triangle x^{(j)}(\tau, \xi), \partial_{\gamma_\kappa}\triangle x^{(j)}_d(\tau)\big)\big\|_{\triangle \tilde{S}}\\&\lesssim \tau_0^{-k_0+}\big\|(\triangle x^{(j-1)}, \triangle x^{(j-1)}_d)\big\|_{S_{strong}} + \tau_0^{0+}\big\|\big(\partial_{\gamma_\kappa}\triangle x^{(j-1)}(\tau, \xi), \partial_{\gamma_\kappa}\triangle x^{(j-1)}_d(\tau)\big)\big\|_{\triangle \tilde{S}},\\&\kappa = 1,2. 
\end{align*}
\end{lem}

\begin{lem}\label{lem:partialgammadecay} For any $\delta>0$, there is $\tau_*  = \tau_*(\delta)$ large enough such that if $\tau_0\geq \tau_*$, then we have 
\begin{align*}
\big\|\big(\partial_{\gamma_\kappa}\triangle x^{(j)}(\tau, \xi), \partial_{\gamma_\kappa}\triangle x^{(j)}_d(\tau)\big)\big\|_{\triangle \tilde{S}}\lesssim_{\delta} \tau_0^{-(k_0+2-)}\lambda^{\frac12}(\tau_0)\delta^j [1+O_{\tau_0}\big(\big\|(x_0, x_1)\big\|_{\tilde{S}} + \big|x_{0d}\big|\big)].
\end{align*}
\end{lem}
Observe that the principal contribution here arises when the operator $\partial_{\gamma_\kappa}$ gets passed from one correction to the earlier one, until it arrives on the source term $e_{approx}$. All other terms arising can be bounded by 
\[
O_{\tau_0}\big(\big\|(x_0, x_1)\big\|_{\tilde{S}} + \big|x_{0d}\big|\big)
\]

The proofs of these lemmas follow very closely the arguments in \cite{CondBlow}, and we shall only indicate the outlines: 
\\

{\it{Outline of proof of Lemma~\ref{lem:partialgammainductivebound}}}: One may assume a decomposition 
\begin{align*}
&\big(\partial_{\gamma_\kappa}\triangle x^{(j-1)}(\tau, \xi), \partial_{\gamma_\kappa}\triangle x^{(j-1)}_d(\tau)\big)\\& = \big((\xi\partial_{\xi})\triangle_{\kappa} y^{(j-1)}(\tau, \xi) + \triangle_{\kappa} z^{(j-1)}(\tau, \xi), \triangle_{\kappa} y^{(j-1)}_d(\tau) + \triangle_{\kappa} z^{(j-1)}_d(\tau)\big)
\end{align*}
with, say, 
\begin{align*}
&\big\|(\triangle_{\kappa} y^{(j-1)}, \triangle_{\kappa} y^{(j-1)}_d)\big\|_{S_{strong}} + \big\|(\triangle_{\kappa} z^{(j-1)}, \triangle_{\kappa} z^{(j-1)}_d)\big\|_{S_{weak}}\\&
\lesssim \big\|\big(\partial_{\gamma_\kappa}\triangle x^{(j-1)}(\tau, \xi), \partial_{\gamma_\kappa}\triangle x^{(j-1)}_d(\tau)\big)\big\|_{\triangle \tilde{S}}
\end{align*}
Now let the operator $\partial_{\gamma_{\kappa}}$ fall on the expression for $\triangle x^{(j)}(\tau, \xi)$ in Proposition~\ref{prop:keyiteratebounds}, given by the parametrix \eqref{eq:iterativestepcont}.  Then if $\partial_{\gamma_{\kappa}}$ acts on the scaling factor in 
\[
[\calR(\tau, \underline{\triangle x}^{(j-1)}) + \triangle f^{(j-1)}(\tau, \underline{\xi})](\sigma, \frac{\lambda^2(\tau)}{\lambda^2(\sigma)}\xi), 
\]
 as well as in 
\[
 S(\tau)\big(\triangle\tilde{\tilde{x}}_0^{(j)}, \triangle\tilde{\tilde{x}}_1^{(j)}\big) ,
 \]
defined as in \eqref{lem:linhom}, then one can incorporate the corresponding term into $(\xi\partial_{\xi})\triangle_{\kappa} y^{(j)}(\tau, \xi)$.
On the other hand, if $\partial_{\gamma_{\kappa}}$ falls on the parametrix factors
\[
U(\tau, \sigma),\,V(\tau,\tau_0),\,U(\tau, \tau_0),
\]
where we recall \eqref{eq:Dtaueffect},  or on one of the $\gamma_{\kappa}$-dependent factors $u_{approx} - u_0, u_{approx}^l$ in $N_{approx}(\epsilon^{(j-1)}) - N_{approx}(\epsilon^{(j-2)})$ (recalling \eqref{eq:Deltatransportj}), we place the corresponding contribution into $\triangle z^{(j)}$. 
The required bounds follow essentially directly from the proofs of Proposition 7.1, 8.1, 9.1, 9.6 in \cite{CondBlow}. 
\\
On the other hand, if $\partial_{\gamma_{\kappa}}$ falls on $\underline{\triangle x}^{(j-1)}$ in $\calR(\tau, \underline{\triangle x}^{(j-1)})$, and we assume that 
\[
\partial_{\gamma_{\kappa}}\triangle x^{(j-1)} = (\xi\partial_{\xi})\triangle y^{(j-1)},\,\triangle y^{(j-1)}\in S_{strong},
\]
one notices that one can 'essentially' move the operator $(\xi\partial_{\xi})$ past the non-local operator $\calR$ modulo better errors which can be placed into $\triangle z^{(j)}$, and further to the outside of the parametrix. The situation is slightly more delicate provided $\partial_{\gamma_\kappa}$ falls on a factor $\triangle \tilde{\eps}^{(l)}$ in $\triangle f^{(j-1)}$, again recalling \eqref{eq:Deltatransportj} and the definition of $\triangle f^{(j-1)}$. Then writing 
 \[
 \triangle \tilde{\eps}^{(l)}(\tau, R) = \triangle x^{(l)}_d(\tau)\phi_d(R) + \int_0^\infty \phi(R, \xi)\triangle x^{(l)}(\tau, \xi)\rho(\xi)\,d\xi, 
 \]
 we exploit the spatial localisation forced by the cutoff $\chi_{R\lesssim\tau}$ in order to perform an integration by parts, provided 
 \[
 \partial_{\gamma_{\kappa}}\triangle x^{(l)} = (\xi\partial_{\xi})\triangle y^{(l)}.
 \]
 Thus write 
 \begin{align*}
 &\chi_{R\leq C\tau}\int_0^\infty \phi(R, \xi)(\xi\partial_{\xi})\triangle y^{(l)}(\tau, \xi)\rho(\xi)\,d\xi\\& = - \chi_{R\leq C\tau}\int_0^\infty(\partial_{\xi}\xi)[\phi(R, \xi)\rho(\xi)]\triangle y^{(l)}(\tau, \xi)\,d\xi,
 \end{align*}
 and then use the bound 
 \begin{align*}
\sup_{\tau\geq \tau_0}\tau^{-1}\big\|R^{-1}\chi_{R\leq C\tau}\int_0^\infty(\partial_{\xi}\xi)[\phi(R, \xi)\rho(\xi)]\triangle y^{(l)}(\tau, \xi)\,d\xi\big\|_{L^\infty_{dR}}\lesssim \big\|\triangle y^{(l)}\big\|_{S_{strong}}. 
 \end{align*}
Indeed, such a bound follows easily from the asymptotic expansions for $\phi(R, \xi)$ given by Prop.~\ref{prop:Fourier}.  If we assume 
 \[
  \partial_{\gamma_{\kappa}}\triangle x^{(l)} = \triangle z^{(l)}\in S_{weak}, 
 \]
we have the weaker estimate 
\begin{align*}
\sup_{\tau\geq \tau_0}\frac{\lambda(\tau_0)}{\lambda(\tau)}\big\|R^{-1}\chi_{R\leq C\tau}\int_0^\infty\phi(R, \xi)\rho(\xi)\triangle z^{(l)}(\tau, \xi)\,d\xi\big\|_{L^\infty_{dR}}\lesssim \big\|\triangle z^{(l)}\big\|_{S_{weak}}.
\end{align*}
Using these and arguing just as in the proof of Proposition 9.6 in \cite{CondBlow} yields the desired bound for the corresponding contribution of $\partial_{\gamma_{\kappa}}\triangle f^{(j-1)}$ to $\triangle x^{(j)}(\tau, \xi)$, which is placed in $S_{weak}$.\\ 
Next, consider the effect of $\partial_{\gamma_{\kappa}}$ on the free term, when it falls on the source term $(\triangle \tilde{\tilde{x}}^{(j)}_0, \triangle \tilde{\tilde{x}}^{(j)}_0)$. In light of the choice of these terms, see the paragraph after the statement of Proposition~\ref{prop:DataLipschitzCont}, we have 
\begin{align*}
\partial_{\gamma_{\kappa}}\triangle \tilde{\tilde{x}}^{(j)}_0 = (\partial_{\gamma_{\kappa}}\alpha^{(j)})\mathcal{F}(\chi_{R\leq C\tau_0}\phi(R, 0)),\,\partial_{\gamma_{\kappa}}\triangle \tilde{\tilde{x}}^{(j)}_1 = (\partial_{\gamma_{\kappa}}\beta^{(j)})\mathcal{F}(\chi_{R\leq C\tau_0}\phi(R, 0)),
\end{align*}
and we have 
\begin{align*}
\partial_{\gamma_{\kappa}}\alpha^{(j)}\sim\partial_{\gamma_{\kappa}}\int_0^\infty \frac{\rho^{\frac12}(\xi)\tilde{\triangle}\tilde{x}^{(j)}_0(\xi)}{\xi^{\frac14}}\cos[\lambda(\tau_0)\xi^{\frac12}\int_{\tau_0}^\infty\lambda^{-1}(u)\,du]\,d\xi, 
\end{align*}
where 
\begin{align*}
\tilde{\triangle}\tilde{x}^{(j)}_0(\xi) = \int_{\tau_0}^\infty U(\tau_0,\sigma)H(\sigma, \frac{\lambda^2(\tau_0)}{\lambda^2(\sigma)}\xi)\,d\sigma,
\end{align*}
and 
\[
H(\sigma, \xi): = [\calR(\tau, \underline{\triangle x}^{(j-1)}) + \triangle f^{(j-1)}(\tau, \underline{\xi})](\sigma, \xi).
\]
The performing integration by parts with respect to $\xi$ if necessary, one checks that 
\begin{align*}
&\big|\int_0^\infty \frac{\rho^{\frac12}(\xi)\partial_{\gamma_{\kappa}}\tilde{\triangle}\tilde{x}^{(j)}_0(\xi)}{\xi^{\frac14}}\cos[\lambda(\tau_0)\xi^{\frac12}\int_{\tau_0}^\infty\lambda^{-1}(u)\,du]\,d\xi\big|\\&
\lesssim \tau_0^{0+}[\tau_0^{-k}\big\|(\triangle x^{(j-1)}, \triangle x^{(j-1)}_d)\big\|_{S_{strong}} + \big\|(\partial_{\gamma_{\kappa}}\triangle x^{(j-1)}, \partial_{\gamma_{\kappa}}\triangle x^{(j-1)}_d)\big\|_{\triangle\tilde{S}}].
\end{align*}
This implies the required bound for $\partial_{\gamma_{\kappa}}\triangle \tilde{\tilde{x}}^{(j)}_0$, and the bound for $\partial_{\gamma_{\kappa}}\triangle \tilde{\tilde{x}}^{(j)}_1$ is similar. One then places 
\[
S(\tau)\big(\partial_{\gamma_{\kappa}}\triangle \tilde{\tilde{x}}^{(j)}_0, \partial_{\gamma_{\kappa}}\triangle \tilde{\tilde{x}}^{(j)}_1\big)
\]
into $S_{weak}$. 
\\
{\it{Outline of proof of Lemma~\ref{lem:partialgammadecay}}}. This follows in analogy to the arguments in sections 11 and 12 in \cite{CondBlow}, a key being re-iterating the iterative step leading from $\partial_{\gamma_{\kappa}}\triangle x^{(j-1)}_l$ to 
$\partial_{\gamma_{\kappa}}\triangle x^{(j-1)}_l$ by differentiating \eqref{eq:iterativestepcont}. 
\\

{\it{Completion of proof of Proposition~\ref{prop:DataLipschitzCont}}}. Recalling Lemma~\ref{lem:Lipcont} and also invoking Lemma~\ref{lem:partialgammadecay}, we find 
\begin{align*}
&\big\|(\triangle \tilde{\tilde{x}}^{(j)}_0 - \triangle \tilde{\tilde{\bar{x}}}^{(j)}_0, \triangle \tilde{\tilde{x}}^{(j)}_1 - \triangle \tilde{\tilde{\bar{x}}}^{(j)}_1)\big\|_{\tilde{S}}\\
&\lesssim_{\delta} [\big\|(x_0 - \bar{x}_0, x_1 - \bar{x}_1)\big\|_{\tilde{S}} + \big|x_{0d} - \bar{x}_{0d}\big|]\cdot\delta^j\big[\frac{\tau_0^{k_0+1}\log\tau_0}{\lambda_{0,0}^{\frac12}(\tau_0)}\cdot \tau_0^{-(k_0+2-)}\lambda^{\frac12}(\tau_0)\\
&\hspace{7cm} + \tau_0^{-(2-)}\big]
\end{align*}
Observe that the final $\tau_0^{-(2-)}$ arises when keeping $\lambda$ fixed and varying the initial data satisfying the vanishing conditions, just as in \cite{CondBlow}, while the more complicated expression preceding $\tau_0^{-(2-)}$ reflects the effect of changing $\gamma_{1,2}$. 
and so we finally get 
\begin{align*}
&\big\|(\triangle \tilde{\tilde{x}}^{(j)}_0 - \triangle \tilde{\tilde{\bar{x}}}^{(j)}_0, \triangle \tilde{\tilde{x}}^{(j)}_1 - \triangle \tilde{\tilde{\bar{x}}}^{(j)}_1)\big\|_{\tilde{S}}\lesssim_{\delta} \delta^j\tau_0^{-(1-)} [\big\|(x_0 - \bar{x}_0, x_1 - \bar{x}_1)\big\|_{\tilde{S}} + \big|x_{0d} - \bar{x}_{0d}\big|]
\end{align*}
This implies Proposition~\ref{prop:DataLipschitzCont}.

\subsubsection{Proof of Theorem~\ref{thm:IterationResult}} This is a consequence of Proposition~\ref{prop:DataLipschitzCont}. Recalling Proposition~\ref{prop:zerothiterate}, Proposition~\ref{prop:keyiteratebounds}, it suffices to set 
\begin{align*}
&(\triangle x^{(\gamma_1,\gamma_2)}_0, \triangle x^{(\gamma_1,\gamma_2)}_1) = \sum_{j=0}^\infty(\triangle \tilde{\tilde{x}}^{(j)}_0, \triangle \tilde{\tilde{x}}^{(j)}_1)\\
&(\triangle x^{(\gamma_1,\gamma_2)}_{0d}, \triangle x^{(\gamma_1,\gamma_2)}_{1d}) = \sum_{j=0}^\infty(\triangle x_d^{(j)}(\tau_0),\,\partial_{\tau}\triangle x_d^{(j)}(\tau_0))
\end{align*}
Then the correction $\tilde{\epsilon}(\tau, R)$ is given by its Fourier coefficients 
\[
\underline{x}(\tau, \xi) = \underline{x}^{(0)}(\tau, \xi) + \sum_{j=1}^\infty \underline{\triangle x}^{(j)}(\tau, \xi)
\]
The decaying bounds over $\big\|(\triangle x^{(j)}, \triangle x^{(j)}_d)\big\|_{S_{strong}} = \triangle A_j$ imply that (recalling \eqref{eq:HS})
\[
\tilde{\epsilon}(\tau, R) = x_d(\tau)\phi_d(R) + \int_0^\infty \phi(R, \xi)x(\tau, \xi)\rho(\xi)\,d\xi\in H^{\frac32+}_{dR, loc}
\] 
for any $\tau\geq \tau_0$, as desired. The fact that the local energy (restricted to $|x|\leq |t|$) vanishes asymptotically follows from 
\[
\big\|r\epsilon_r\big\|_{L^2_{dr}(r\leq t)}\leq \lambda^{-\frac32}\big\|\tilde{\epsilon}_R\big\|_{L^2_{dR}(R\lesssim\tau)} + \lambda^{-\frac32}\big\|\frac{\tilde{\epsilon}}{R}\big\|_{L^2_{dR}(R\lesssim\tau)},
\]
and invoking the Fourier representation to bound the $L^2$-norms on the right, resulting in 
\[
\big\|r\epsilon_r\big\|_{L^2_{dr}(r\leq t)}\lesssim \tau^{\frac53 - \frac32(1+\nu^{-1})}, 
\]
and similarly for $r\epsilon_t$. 

\subsection{Translation to original coordinate system}

In the preceding sections, we have obtained a singular solution of the form (the sum of the first four terms on the right representing $u_{approx}^{(\gamma_1,\gamma_2)}$ given by Theorem~\ref{thm:main approximate})
\[
u(\tau, R) = \lambda^{\frac12}(\tau)W(R) + \sum_{j=1}^{2k_*-1}v_j(\tau, R) + \sum_{a = 1,2}v_{smooth,a}(\tau, R) + v(\tau, R) + R^{-1}\tilde{\epsilon}(\tau, R),
\]
with the error term $\tilde{\epsilon}(\tau, R)$ given by the Fourier expansion
\[
\tilde{\epsilon}(\tau, R) = \int_0^\infty \phi(R, \xi)[x^{(0)}(\tau, \xi) + \sum_{j=1}^\infty \triangle x^{(j)}(\tau, \xi)]\rho(\xi)\,d\xi.
\]
At initial time $\tau = \tau_0$, setting $x(\tau, \xi): = x^{(0)}(\tau, \xi) + \sum_{j=1}^\infty \triangle x^{(j)}(\tau, \xi)$, we have from our construction 
\[
\big(x(\tau_0, \xi),\mathcal{D}_{\tau}x(\tau_0, \xi)\big) = (x_0^{(\gamma_1,\gamma_2)} + \triangle x_0^{(\gamma_1,\gamma_2)}, x_1^{(\gamma_1,\gamma_2)} + \triangle x_1^{(\gamma_1,\gamma_2)}),
\]
\[
x_d(\tau_0) = x_d^{(\gamma_1,\gamma_2)} + \triangle x_d^{(\gamma_1,\gamma_2)}
\]
where we recall 
\[
 \triangle x_l^{(\gamma_1,\gamma_2)}(\xi) = \sum_{j=1}^\infty  \triangle \tilde{\tilde{x}}_l^{(j)}(\xi),\,l = 1,2, \triangle x_{0d}^{(\gamma_1,\gamma_2)} = \sum_{j=0}^\infty \triangle x^{(j)}_{0d}(\tau_0)
 \]
 
The fact that we have added on the correction terms $ \triangle x_l^{(\gamma_1,\gamma_2)}(\xi)$ means that the data 
\[
\big(R^{-1}\tilde{\epsilon}(\tau_0, R),\,\partial_tR^{-1}\tilde{\epsilon}(\tau_0, R)\big)
\]
will no longer match the original data $(\bar{\epsilon}_1, \bar{\epsilon}_2)$, and we need to precisely quantify this correction {\it{at the level of the Fourier variables associated with the old radial variable $R_{0,0}$}}. Doing so requires recalling \eqref{eq:x0barepsilon0} - \eqref{eq:datatransference4} as well as Lemma~\ref{lem:changeofscale}. Assume that our construction has replaced the data $(\bar{\epsilon}_1, \bar{\epsilon}_2)$ in \eqref{eq:thesolutionansatz} by $(\bar{\epsilon}_1+\triangle \epsilon_1, \bar{\epsilon}_2 + \triangle \epsilon_2)$, we have the relations 
\begin{align*}
&R\triangle \epsilon_1(R) =  \int_0^\infty\phi(R, \xi)\triangle x_0^{(\gamma_1,\gamma_2)}(\xi)\rho(\xi)\,d\xi + \triangle x_d^{(\gamma_1,\gamma_2)}\phi_d(R),\\
&\triangle x_1^{(\gamma_1,\gamma_2)}(\xi) = -\lambda^{-1}(\tau_0)\int_0^\infty\phi(R, \xi)R\triangle \epsilon_2\,dR - \frac{\dot{\lambda}}{\lambda}\mathcal{K}_{cc}\triangle x_0^{(\gamma_1,\gamma_2)} - \frac{\dot{\lambda}}{\lambda}\mathcal{K}_{cd}\triangle x_d^{(\gamma_1,\gamma_2)},
\end{align*}
where we recall that $\lambda = \lambda_{\gamma_1,\gamma_2}$. Recalling the relation
\[
\chi_{r\leq Ct_0}u_{approx}^{(0,0)}[t_{0}] + (\epsilon_1, \epsilon_2) = \chi_{r\leq Ct_0}u_{approx}^{(\gamma_1,\gamma_2)}[t_{0}] + (\bar{\epsilon}_1, \bar{\epsilon}_2)
\]
for the initial data, we see that the initial data perturbation $(\epsilon_1, \epsilon_2)$ in \eqref{eq:perturbeddata1} has been replaced by 
\begin{equation}\label{eq:dataperturbed}
(\epsilon_1+\triangle \epsilon_1, \epsilon_2+\triangle \epsilon_2) + (1-\chi_{r\leq Ct_0})\cdot (u_{approx}^{(0,0)}[t_{0}] - u_{approx}^{(\gamma_1,\gamma_2)}[t_{0}]),
\end{equation}
Here we may suppress the term 
\[
(1-\chi_{r\leq Ct_0})\cdot (u_{approx}^{(0,0)}[t_{0}] - u_{approx}^{(\gamma_1,\gamma_2)}[t_{0}])
\]
since this will not affect the evolution in the backward light cone. 
In light of the fact that the corresponding Fourier variables $(x_0, x_1)$ were computed from $(\epsilon_1, \epsilon_2)$ via \eqref{eq:x0barepsilon0} - \eqref{eq:datatransference4} with $\gamma_{1,2} = 0$, we infer that the perturbed data \eqref{eq:dataperturbed} with the second part suppressed correspond to Fourier variables (with respect to the physical radial variable $R_{0,0}$) given by $(x_0+\triangle x_0, x_1 + \triangle x_1)$ for the continuous part and $x_d + \triangle x_d$ for the discrete part, where we have 
 \begin{align*}
 &\triangle x_0(\xi) = \int_0^\infty \phi(R_{0,0}, \xi)R_{0,0}\triangle \epsilon_1(R(R_{0,0}))\,dR_{0,0},\\
 &\triangle x_d = \int_0^\infty \phi_d(R_{0,0})R_{0,0}\triangle \epsilon_1(R(R_{0,0}))\,dR_{0,0} \\
 &\triangle x_1(\xi) = -\lambda_{0,0}^{-1}(\tau_0)\int_0^\infty\phi(R_{0,0}, \xi)R_{0,0}\triangle \epsilon_2\,dR_{0,0} - \frac{\dot{\lambda}_{0,0}}{\lambda_{0,0}}\mathcal{K}_{cc}\triangle x_0 - \frac{\dot{\lambda}_{0,0}}{\lambda_{0,0}}\mathcal{K}_{cd}\triangle x_d\\
 &\triangle x_{1d} = -\lambda_{0,0}^{-1}(\tau_0)\int_0^\infty\phi_d(R_{0,0})R_{0,0}\triangle \epsilon_2\,dR_{0,0} -  \frac{\dot{\lambda}_{0,0}}{\lambda_{0,0}}\mathcal{K}_{cc}\triangle x_d - \frac{\dot{\lambda}_{0,0}}{\lambda_{0,0}}\mathcal{K}_{dc}\triangle x_0\\
 \end{align*}
 Then using Lemma~\ref{lem:changeofscale} we easily infer 
 \begin{align*}
 \big\|\triangle x_0(\xi)\big\|_{\tilde{S}_1}\lesssim \big\|\triangle x_0^{(\gamma_1,\gamma_2)}\big\|_{\tilde{S}_1} + \big|\triangle x_{0d}^{(\gamma_1,\gamma_2)}\big|\lesssim \tau_0^{-(1-)}[\big\|(x_0, x_1)\big\|_{\tilde{S}} + \big|x_{0d}\big|],
 \end{align*}
 and similarly 
 \begin{align*}
 &\big\|\triangle x_1(\xi)\big\|_{\tilde{S}_2}\lesssim \tau_0^{-(1-)}[\big\|(x_0, x_1)\big\|_{\tilde{S}} + \big|x_{0d}\big|]
 \end{align*}
 For the discrete part of the correction, we get 
 \begin{align*}
 \big|\triangle x_{0d}\big| &= \big|\int_0^\infty \phi_d(R_{0,0})R_{0,0}\triangle \epsilon_1(R(R_{0,0}))\,dR_{0,0}\big|\\&\lesssim \tau_0^{-(1-)}\big|x_{0d}\big| + [\big\|(x_0, x_1)\big\|_{\tilde{S}} + \big|x_{0d}\big|]^2. 
 \end{align*}
 Finally, we observe that the discrete spectral part of $\epsilon_2 + \triangle \epsilon_2$ with respect to the radial variable $R_{0,0}$ is completely determined in terms of $(x_0,x_1), x_{0d}$ and in fact a Lipschitz function of these. 
 To conclude this discussion, we note that our precise choice of $\triangle \epsilon_l$, $l = 1,2$, as well as Theorem~\ref{thm:IterationResult} imply that the mapping 
 \[
\big(x_0, x_1, x_{0d}\big)\longrightarrow \big(\triangle x_0, \triangle x_1, \triangle x_{0d}\big)
\]
is Lipschitz with respect to the norm $\big\|(\cdot, \cdot)\big\|_{\tilde{S}} + \big|\cdot\big|$, with Lipschitz constant $\ll 1$. 

\section{Proof of Theorem~\ref{thm:Main}}

This is immediate from the preceding discussion: the implicit function theorem guarantees that the mapping 
\[
\big(x_0, x_1, x_{0d}\big)\longrightarrow  \big(x_0 + \triangle x_0, x_1 + \triangle x_1, x_{0d}+\triangle x_{0d}\big)
\]
is invertible on a sufficiently small open neighbourhood of the origin in $\tilde{S}\times \R$. Moreover, the second discrete spectral component $x_{1d} + \triangle x_{1d}$ is then uniquely determined as a Lipschitz function of 
\[
\big(x_0 + \triangle x_0, x_1 + \triangle x_1, x_{0d}+\triangle x_{0d}\big). 
\]

\centerline{\scshape Stefano Burzio }
\medskip
{\footnotesize
 \centerline{B\^{a}timent des Math\'ematiques, EPFL}
\centerline{Station 8, 
CH-1015 Lausanne, 
  Switzerland}
  \centerline{\email{stefano.burzio@epfl.ch}}
} 

\centerline{\scshape Joachim Krieger }
\medskip
{\footnotesize
 \centerline{B\^{a}timent des Math\'ematiques, EPFL}
\centerline{Station 8, 
CH-1015 Lausanne, 
  Switzerland}
  \centerline{\email{joachim.krieger@epfl.ch}}
} 

\end{document}